\documentclass[11pt,reqno,twoside]{amsart}
\usepackage{amssymb,amsfonts}
\usepackage{amsmath}
\usepackage{color} 
\usepackage{times}
\usepackage{cite}

\newcommand{\beq}{\begin{equation}}
\newcommand{\eeq}{\end{equation}}
\newcommand{\beqs}{\begin{equation*}}
\newcommand{\eeqs}{\end{equation*}}
\newcommand{\ba}{\begin{array}}
\newcommand{\ea}{\end{array}}
\newcommand{\beas}{\begin{eqnarray*}}
\newcommand{\eeas}{\end{eqnarray*}}
\newcommand{\bea}{\begin{eqnarray}}
\newcommand{\eea}{\end{eqnarray}}
\newcommand{\bal}{\begin{align}}
\newcommand{\eal}{\end{align}}

\newcommand{\bals}{\begin{align*}}
\newcommand{\eals}{\end{align*}}

\newcommand{\R}{\ensuremath{\mathbb R}}

\newcommand{\norm}[1]{\| {#1} \|}

\newcommand{\bds}{\begin{displaystyle}}
\newcommand{\eds}{\end{displaystyle}}

\renewcommand{\eqref}[1]{(\ref{#1})}

\def\longequals{\mathbin{=\kern-2pt=}}
\def\eqdef{\mathbin{\buildrel \rm def \over \longequals}}

\def\varep{\varepsilon}

\def\ddt{\frac{d}{dt}}

\newcommand{\remove}[1]{} 
\renewcommand{\remove}[1]{#1} 

\newtheorem{theorem}{Theorem}[section]

\newtheorem{lemma}[theorem]{Lemma}
\newtheorem{corollary}[theorem]{Corollary}
\newtheorem{proposition}[theorem]{Proposition}

\newtheorem{remark}[theorem]{Remark}

\newtheorem*{notation}{Notation}

\theoremstyle{remark}
\newtheorem{example}[theorem]{Example}

\numberwithin{equation}{section}

\newcommand{\esssup}{\mathop{\mathrm{ess\,sup}}}

\newcommand{\gamex}{\gamma}
\newcommand{\nuex}{\nu}
\newcommand{\muex}{\mu}
\newcommand{\Bfun}{B}
\newcommand{\alphazero}{\mathcal E}
\newcommand{\egexp}{\theta}
\newcommand{\Nf}{\mathcal K}
\newcommand{\Ulim}{\mathcal A}

\newcommand{\asdc}{Assume \eqref{strictdeg}}

\newcommand{\myB}{F}

\newcommand{\setV}{V}

\def\myclearpage{}

\def\mynewline{}

\definecolor{darkred}{rgb}{.70,.12,.20}

\definecolor{darkgreen}{rgb}{.20,.52,.14}

\title{Properties of Generalized Forchheimer Flows in Porous Media}
\author{Luan T. Hoang$^{\dag,*}$}
\author{Thinh T. Kieu$^\dag$} 
\author{Tuoc V. Phan$^\ddag$}
\address{$^\dag$ Department of Mathematics and Statistics, Texas Tech University, Box 41042, Lubbock, TX 79409--1042, U. S. A.}
\email{luan.hoang@ttu.edu}
\email{thinh.kieu@ttu.edu}
\address{$^\ddag$ Department of Mathematics, University of Tennessee, Knoxville, 227 Ayress Hall, 1403 Circle Drive, Knoxville, TN 37996 }
\email{phan@math.utk.edu}
\address{$^*$Corresponding author}

\date{\today}

\begin{document}

\begin{abstract}
The nonlinear Forchheimer equations are used to describe the dynamics of fluid flows in porous media when Darcy's law is not applicable.
In this article, we consider the generalized Forchheimer flows for slightly compressible fluids
and study the initial boundary value problem for the resulting degenerate parabolic equation for pressure 
with the time-dependent flux boundary condition.
We estimate $L^\infty$-norm for pressure and its time derivative, as well as other Lebesgue norms for its gradient and second spatial derivatives. The asymptotic estimates as time tends to infinity are emphasized.
We then show that the solution (in interior $L^\infty$-norms) and its gradient (in interior $L^{2-\delta}$-norms) depend continuously on the initial and boundary data, and  coefficients of the Forchheimer polynomials. These are proved for both finite time intervals and time infinity.
The De Giorgi and Ladyzhenskaya-Uraltseva iteration techniques are combined with uniform Gronwall-type estimates, specific monotonicity properties, 
suitable parabolic Sobolev embeddings and a new fast geometric convergence result.
\end{abstract}

\maketitle



\tableofcontents

\pagestyle{myheadings}\markboth{Luan T. Hoang, Thinh T. Kieu and  Tuoc V. Phan}
{Generalized Forchheimer Flows in Porous Media}

\clearpage
\section{Introduction}
\label{intro}

In literature, Darcy's equation is considered as law of hydrodynamics in porous media. This linear relation between the fluid velocity and pressure gradient was derived by Darcy in \cite{Darcybook}.
However, even in their early works, Darcy and Dupuit already observed deviations of fluid flows from this linear equation \cite{Darcybook, Dupuit1857}.
In early 1900s, Forchheimer proposed three models for nonlinear flows, the so-called two-term, three-term and power laws (cf. \cite{Forchh1901,ForchheimerBook}) . More experiments gave rise to more nonlinear models during 1940-60s (cf. \cite{BearBook,Nieldbook}). Despite that fact, most mathematical papers on fluids in porous media deal only with Darcy's law starting from 1960s. The mathematical investigations of Forchheimer equations and related Brinkman-Forchheimer equation started much later in 1990s and have been growing ever since (cf. \cite{ChadamQin,Qin1998,Payne1999b,Payne1999a,Franchi2003,CKU2005,CKU2006,Straughan2013}, see also \cite{Straughanbook}). Even so, most of these papers consider incompressible fluids only. In our previous works \cite{ABHI1,HI1,HI2,HIKS1},
we proposed and studied generalized Forchheimer equations for slightly compressible fluids in porous media. Such mathematical generalization is appropriate and useful. This is due to the nature of Forchheimer equations which are derived from experiments and have their physical parameters found by fitting real life data.
From mathematical point of view, it introduces a new class of degenerate parabolic equations  into studies of porous medium equations \cite{VazquezPorousBook}.
In our mentioned papers, we study the properties of  pressure in space $L^\alpha$ ($1\le \alpha<\infty$), of pressure gradient in space $L^{2-a}$  and of pressure's time derivative in $L^2$. Here $a$ is a number between $0$ and $1$ defined in terms the degree of the polynomial in the generalized Forchheimer equations (see \eqref{ab}). 
In this paper, we focus on properties of fluid flows in higher regularity spaces.
Specifically, we will study the pressure and its time derivative in space $L^\infty$, the pressure gradient in $L^s$ for any $1\le s<\infty$ and the pressure Hessian in $L^{2-\delta}$ for $\delta>0$. 
Moreover, our high priority is the long-time dynamical properties, including uniform estimates in time, asymptotic bounds and asymptotic stability.
Such topics of long-time dynamics of degenerate parabolic equation, particularly in $L^\infty$, is important, and the specific results are usually hard to obtain.
(See, for e.g.,  \cite{VazquezSmoothBook}, chapters 18--20 of \cite{VazquezPorousBook} for porous medium equations, \cite{RVV2009,RVV2010} for degenerate equations with Dirichlet boundary condition, \cite{LD2000} for systems.) 
In order to work in these much higher regularity spaces and obtain estimates for large time, more sophisticated techniques are called for.
We combine iteration techniques by De Giorgi  \cite{DeGiorgi57} and Ladyzhenskaya-Uraltseva \cite{LadyParaBook68}, which were primarily used for studying local properties of solutions to elliptic and parabolic problems, with those from long-time dynamics studies for nonlinear partial differential equations such as Navier-Stokes equations \cite{FMRTbook}. For our degenerate equations, we also use and refine relevant techniques in DiBenedetto's book \cite{DiDegenerateBook}.
Such a combination gives fruitful results on the estimates of solutions for large time as well as detailed continuous dependence of the solutions on time-dependent boundary data and coefficients of the Forchheimer polynomials. (The latter results are called structural stability.)
We also emphasize  that the mentioned general techniques from parabolic equations must be used in accordance with the structure of our equation, in this case, the important monotonicity and perturbed monotonicity in Lemma \ref{quasimono-lem} below.

In the current work, we focus on the case of Degree Condition, see  \eqref{deg} in the next section, for the following reasons. 
First, it already covers most commonly used  Forchheimer equations in practice, namely, the two-term, three-term and power laws.
Second, to take advantage of available estimates in our previous work \cite{HI2}. 
Third, to make clear our ideas and techniques without involving much more complicated technical details in case that the Degree Condition is not met (see \cite{HIKS1}); such case will be investigated in our future work.

The paper is organized as follows.
\mynewline
In section \ref{prelim}, we present the formulation of generalized Forchheimer equations and consequently obtain a degenerate parabolic equation for pressure $p$. Basic properties of this equation are reviewed and  suitable parabolic Sobolev embeddings are presented.
\mynewline
In section \ref{supestimate}, we study the initial boundary value problem (IBVP) for pressure with the time-dependent flux boundary data $\psi(x,t)$. 
We derive various estimates for the shifted solution $\bar p$ (see \eqref{pbargam}).
This section is divided into four subsections.
\mynewline
Subsection \ref{pressuresubsec} deals with $L^\infty$-estimates up to the boundary for $\bar p$. 
In  Theorem \ref{thm-1}, bounds are established for all time, large time (independent of initial data) and time infinity.
The precise estimates for both small and large data are achieved by using Lemma \ref{multiseq} in the Appendix -- a generalization of the classical fast geometric convergence of sequences of numbers.
\mynewline
Subsection \ref{Lsgrad-sec} deals with interior $L^s$-estimates  for $\nabla p$ for any $s\ge 1$, see Theorems \ref{cor311} and \ref{cor312}.
The key Sobolev embedding is Lemma \ref{LUK} with specific weight $K(|\nabla w|)$ related to our degenerate structure.
The main iteration step is \eqref{Kgrad1}.
\mynewline
Subsection \ref{Lpt-sec} deals with interior $L^\infty$-estimates  for $\bar p_t$, where De Giorgi's technique is applied to the time derivative with weighted parabolic Sobolev embedding.
The main estimates are in Theorem \ref{ptbar}, other particular large time and asymptotic estimates are in Theorems \ref{ptbar2} and \ref{smallqbar}.
\mynewline
Subsection \ref{seconderiv} deals with interior $L^{2-\delta}$-estimates with $\delta\in(0,a]$ for Hessian $\nabla^2 p$. 
The estimates in Theorems \ref{hessest}, \ref{HessThm3}  and \ref{HessLim} for second derivatives were not considered in our previous works. 
\mynewline
\mynewline
Sections \ref{datacont} and \ref{polycont} are devoted to the structural stability issues.
\mynewline
In section \ref{datacont}, we prove the continuous dependence of the solutions on the boundary data. Specifically, it is established for $\bar p$
 in interior $L^\infty$-norms  (see Theorem \ref{thm46}) and for $\nabla p$ in interior $L^{2-\delta}$-norms  (see Theorems \ref{GradThm1} and \ref{lem412} for finite time intervals, and Theorems \ref{newP},  \ref{newthmAbar2} and  \ref{GradThm2} for $t\to\infty$). 
The results show that even when each individual flux $\psi_1,\psi_2$ grows unbounded as time $t\to\infty$, the difference between two corresponding solutions $\bar p_1,\bar p_2$ can be small provided the difference $\Psi=\psi_1-\psi_2$ is small.
In order to obtain this, De Giorgi's iteration is combined, in Proposition \ref{cont:strictcond}, with the monotonicity available for our degenerate equation.
\mynewline
In section \ref{polycont}, we prove the continuous dependence of the solutions on the Forchheimer polynomials (see Theorems \ref{thm54}, \ref{thmsupP},  \ref{GradThm3} and \ref{lastthm}).
Here, the perturbed monotonicity is combined with De Giorgi's technique as presented in Proposition \ref{theo49}.
It is proved in Theorems \ref{thm55} and \ref{GradThm4} that the smallness of the difference $\bar P(x,t)$ between two solutions corresponding to two Forchheimer equations, when $t\to\infty$, can be controlled by the difference between coefficient vectors of the two Forchheimer polynomials.

\myclearpage
\section{Preliminaries}
\label{prelim}

We consider a fluid  in a porous medium having velocity $u(x,t)\in \R^n$, pressure $p(x,t)\in \R$ and density $\rho (x,t)\in \R^+=[0,\infty)$, with the spatial variable $x\in \R^n$ and time variable $t\in\R$.
The space dimension $n=3$ in physics problems, but here we consider any $n\ge 2$. 
Generalized Forchheimer equations, studied in \cite{ABHI1,HI1}, are of the form:
\beq\label{gForch} g(|u|)u=-\nabla p,\eeq
where $g(s)\ge 0$ is a function defined on $[0,\infty)$.
When 
\beqs
g(s)=\alpha,\ \alpha +\beta s,\ \alpha +\beta s+\gamma s^2,\ \alpha +\gamma_m s^{m-1},
\eeqs 
where $\alpha,\beta,\gamma,m,\gamma_m$ are empirical constants, we have Darcy's law, Forchheimer's two-term, three-term and power laws, respectively.
In this paper, the function $g$ in \eqref{gForch} is a generalized polynomial with non-negative coefficients, that is,
\beq\label{gsa} g(s) =a_0s^{\alpha_0}+a_1 s^{\alpha_1}+\ldots +a_N s^{\alpha_N},\ s\ge 0,\eeq
where  
$N\ge 1$, $\alpha_0=0<\alpha_1<\ldots<\alpha_N$ are real (not necessarily integral) numbers, the coefficients satisfy  $a_0,a_N>0$ and $a_1,\ldots,a_{N-1}\ge 0$.
The number $\alpha_N$ is the degree of $g$ and is denoted by $\deg(g)$.
Denote the vectors of powers and coefficients by $\vec \alpha=(\alpha_0,\ldots,\alpha_N)$ and
 $\vec a=(a_0,\ldots,a_N)$. 
The class of functions $g(s)$ as in \eqref{gsa} is denoted by FP($N,\vec \alpha$), which is the abbreviation of ``Forchheimer polynomials''. When the function $g$ in \eqref{gForch} belongs to FP($N,\vec \alpha$), it is referred to as the Forchheimer polynomial.

From relation \eqref{gForch} one can solve velocity $u$  in terms of pressure gradient $\nabla p$ and
derives a nonlinear version of Darcy's equation:
\beq\label{u-forma} u= -K(|\nabla p|)\nabla p.\eeq
The function $K:\R_+\to\R_+$  is defined by
\beq\label{Kdef} K(\xi)=\frac1{g(s(\xi))},
\text{ where } s=s(\xi)\ge 0 \text{ satisfies } sg(s)=\xi, \ \text{ for }\xi\ge 0. \eeq
We will use notation $g(s,\vec a)$, $K(\xi,\vec a)$, $s(\xi,\vec a)$ to denote the corresponding functions in \eqref{gsa} and \eqref{Kdef} when the dependence on $\vec a$ needs be specified.

Equation \eqref{gForch} replaces the momentum equation in fluid mechanics. In addition to this, we have the equation of  continuity
\beq\label{conti-eq} \frac{\partial \rho}{\partial t}+\nabla\cdot(\rho u)=0,\eeq
and the  equation of state which, for slightly compressible fluids, is
\beq\label{slight-compress} 
\frac{d\rho}{dp}=\frac{\rho}{\kappa},\quad \kappa>0.
\eeq
From equations \eqref{u-forma}, \eqref{conti-eq} and \eqref{slight-compress} one derives an equation for the pressure:
\beq\label{dafo-nonlin} \frac{\partial p}{\partial t}=\kappa\nabla \cdot (K(|\nabla p|)\nabla p) + K(|\nabla p|)|\nabla p|^2.\eeq
Since the constant $\kappa$ is very large  for most slightly compressible fluids in porous media \cite{Muskatbook},  we neglect the second term on the right-hand side of \eqref{dafo-nonlin}. This results in the following reduced equation
\beq\label{lin-p} \frac{\partial p}{\partial t} = \kappa \nabla\cdot (K(|\nabla p|)\nabla p).\eeq
Note that this reduction is commonly used in engineering.
By rescaling the time variable,  hereafter we assume that $\kappa=1$.

Let $g=g(s,\vec a)$ in FP($N,\vec \alpha$). 
The following numbers  are in our calculations:
\beq\label{ab} a=\frac{\alpha_N}{1+\alpha_N}\in(0,1),\ b= \frac{a}{2-a}= \frac{\alpha_N}{2+\alpha_N}\in(0,1),\eeq
\beq\label{Ag} \chi(\vec a)=\max\Big \{a_0,a_1,\ldots,a_N,\frac1{a_0},\frac1{a_N} \Big \}\in [1,\infty).\eeq
The following properties for $K(\xi,a)$ are proved Lemmas III.5 and III.9 of \cite{ABHI1}. 
For $\xi\ge 0$, we have 
\beq\label{K-est-2}
- a\, K(\xi,\vec a)\le  \xi\, \frac{\partial  K(\xi,\vec a)}{\partial\xi} \le 0.
\eeq
Consequently, $K(\xi,\vec a)$ is decreasing in the variable $\xi$ and hence
\beq\label{Kestzero}
K(\xi, \vec a)\le K(0,\vec a)=a_0^{-1}\le \chi(\vec a),\eeq
Moreover, $K(\xi,\vec a)\xi^m$ is increasing in $\xi$ for all $m\ge 1$.

For the type of degeneracy of $K(\xi,\vec a)$ in $\xi$, we recall the following facts.
 
\begin{lemma}[cf. \cite{HI1}, Lemma 2.1]\label{lem21}
Let $g(s,\vec a)$ be in class FP($N,\vec \alpha$).  One has for any $\xi\ge 0$ that
\beq\label{Kesta} 
\frac{ C_0^{-1} \chi(\vec a)^{-1-a} }{(1+\xi)^a}\le K(\xi, \vec a)\le \frac{ C_0 \chi(\vec a)^{1+a} }{(1+\xi)^a},
\eeq
and for any $m\ge 1$, $\delta>0$ that
\beq\label{Km}
c_0^{-1} \chi(\vec a)^{-1-a} \frac{\delta^a}{(1+\delta)^a} (\xi^{m-a}-\delta^{m-a}) \le  K(\xi,\vec a)\xi^m\le c_0 \chi(\vec a)^{1+a}\xi^{m-a},
\eeq
where $c_0=c_0(N,\alpha_N)>0$ depends on $N$ and $\alpha_N$ only.
In particular, when $m=2$, $\delta=1$, one has
\beq\label{Kestn}
2^{-a}c_0^{-1}\chi(\vec a)^{-1-a}(\xi^{2-a}-1)\le K(\xi,\vec a)\xi^2\le c_0 \chi(\vec a)^{1+a}\xi^{2-a}.
\eeq
\end{lemma}

Same as in \cite{ABHI1,HI1}, we define
\beq\label{Hxi} H(\xi,\vec a)=\int_0^{\xi^2}K(\sqrt s,\vec a)ds \quad \hbox{for } \xi\ge 0.\eeq

When vector $\vec a$ is fixed, we denote $K(\cdot,\vec a)$ and $H(\cdot,\vec a)$ by $K(\cdot)$ and $H(\cdot)$, respectively.
The function $H(\xi)$ can be compared with $\xi$ and $K(\xi)$ by 
\beq\label{Hcompare} 
K(\xi)\xi^2\leq H(\xi)\le 2K(\xi)\xi^2,\quad c_1(\xi^{2-a}-1)\leq H(\xi)\le  c_2\xi^{2-a},
\eeq
where $c_1,c_2>0$ depend on $\chi(\vec a)$.


Next, we recall important monotonicity properties.
For convenience, we use the following notation: let $\vec{x}=(x_1,x_2,\ldots)$ and $\vec{x}'=(x'_1,x'_2,\ldots)$ be two arbitrary vectors of the same length, including possible length $1$. We denote by $\vec{x}\vee \vec{x}'$ and $\vec{x}\wedge \vec{x}'$ their maximum and minimum vectors, respectively, with components
$(\vec{x}\vee \vec{x}')_j=\max\{x_j, x'_j\}$  and $(\vec{x}\wedge \vec{x}')_j=\min\{ x_j, x'_j\}$.

\begin{lemma}[cf. \cite{HI1}, Lemma 5.2]\label{quasimono-lem}
Let $g(s,\vec a)$ and $g(s,\vec a')$ belong to class FP($N,\vec \alpha$). Then for any $y,y'$ in $\R^n$, one has
\begin{multline}\label{quasimonotone}
(K(|y|,\vec a) y-K(|y'|,\vec a')y')\cdot (y-y')
\ge (1-a)\cdot K(|y|\vee|y'|,\vec  a\vee \vec a')\cdot |y-y'|^2 \\
- N\cdot \max\{ \chi(\vec a),\chi(\vec a')\}\cdot |\vec  a-\vec a'|\cdot  K(|y|\vee |y'|,\vec a\wedge \vec a')\cdot (|y|\vee |y'|)\cdot |y-y'|, 
\end{multline}
where $a\in(0,1)$ is defined in \eqref{ab}. 
Particularly, when $\vec  a=\vec a'$, we have
\beq\label{mono1}
(K(|y|,\vec a) y-K(|y'|,\vec a)y')\cdot (y-y') \ge (1-a)\cdot K(|y|\vee|y'|,\vec a)\cdot |y-y'|^2.
\eeq
\end{lemma}


When the fluid is confined in an open, bounded domain $U$ of $\R^n$, the Sobolev embeddings play an important role in our study. For $0<r<n$, we denote by $r^*$ the critical Sobolev exponent, i.e. $r^* = \frac{nr}{n-r}$. The following 
fact is used frequently
\beq\label{rfact}
r^*>2\Leftrightarrow \frac{nr}{n-r}>2  \Leftrightarrow r>\frac{2n}{n+2}\Leftrightarrow (2-r)n<2r.
\eeq
With this notation, we define the \textbf{Degree Condition} as one of the following equivalent conditions:
\beq\label{deg}     \deg(g)\le \frac{4}{n-2},\quad  a\le \frac 4{n+2},\quad  2\le (2-a)^*,\quad  2-a\ge \frac{2n}{n+2}.\tag{\textbf{DC}}
\eeq
Similarly, we define the \textbf{Strict Degree Condition} as one of the following equivalent conditions:
\beq\label{strictdeg}
    \deg(g)< \frac{4}{n-2},\quad a <  \frac 4{n+2},\quad 2 < (2-a)^*,\quad  2-a> \frac{2n}{n+2}.\tag{\textbf{SDC}}
\eeq

We will assume the Strict Degree Condition very often in this paper, but not always. Whenever this condition is met, the Sobolev space $W^{1,2-a}(U)$ is continuously embedded into $L^2(U)$.

We now consider  parabolic Sobolev embeddings. 
Let us denote throughout $Q_T=U\times(0,T)$.

\begin{lemma}\label{ParSob-1}
If $2n/(n+2)\le r \le 2$, $r<n$ and $p=r(n+2)/n$ then
\beq\label{Sineq1}
\| u\|_{L^p(Q_T)} \le C(1+T^{1/p})[[u]]_{2,r;T},
\eeq
where  $C=C(U,n,r)>0$ is independent of $T$ and 
\beq\label{normtriple}
[[u]]_{2,r;T}=\esssup_{t\in [0,T]} \| u(t)\|_{L^2(U)}+ \|\nabla u\|_{L^r(Q_T)}.
\eeq

In \eqref{Sineq1} above, we can remove $T^{1/p}$ whenever $u$ vanishes on the boundary of $U$.
\end{lemma}
\begin{proof}
The proof is standard, cf. \cite{LadyParaBook68,DiDegenerateBook}. 
The short proof is presented here for the sake of completeness, and also serves the next lemma.
For convenience we denote $[[\cdot]] \equiv [[\cdot]]_{2,r,T}$. 
Note that $2\le p \le r^*$. We write 
\beq\label{powers}
\frac 1 p=\frac \alpha 2+\frac\beta {r^*},\quad \alpha,\beta\ge 0,\quad \alpha+\beta=1.
\eeq
By interpolation inequality and Sobolev embedding, we have
\beqs
\|u(t)\|_{L^p(U)} \le \| u(t)\|_{L^2(U)}^\alpha \|u(t)\|_{L^{r^*}(U)}^\beta
\le C \| u(t)\|_{L^2(U)}^\alpha (\delta \|u(t)\|_{L^r(U)}^\beta +\norm{\nabla u(t)}_{L^r(U)}^\beta),
\eeqs
where $\delta=1$ in general, and $\delta=0$ in case $u$ vanishes on the boundary $\partial U$.
Note that $r\le 2$ then $\|u(t)\|_{L^r(U)}\le C\|u(t)\|_{L^2(U)}$.  Thus,
\beq\label{e1}
\|u(t)\|_{L^p(U)}  
\le C \delta \| u(t)\|_{L^2(U)} + C  \| u(t)\|_{L^2(U)}^\alpha \norm{\nabla u(t)}_{L^r(U)}^\beta.
\eeq
Raising \eqref{e1} to power $p$, integrating it in $t$ from $0$ to $T$ and using the fact that 
$\| u(t)\|_{L^2(U)} \le [[u]]$ a.e. in $[0,T]$,  we obtain 
\beq\label{e2}
 \norm{u}_{L^p(Q_T)}^p 
\le C \delta T [[u]]^{p} + C [[u]]^{\alpha  p} \int_0^T \|\nabla u(t)\|_{L^r(U)}^{\beta p} dt.
\eeq
From \eqref{powers}, we find $\beta=\frac n{n+2}$. Thus, $\beta p= r$ and we have from \eqref{e2} that
\beq\label{e3}
\begin{aligned}
 \norm{u}_{L^p(Q_T)}^p 
&\le C \delta T [[u]]^{p} + C [[u]]^{\alpha p} \Big(\int_0^T\int_U  |\nabla u(x,t)|^{r} dx dt\Big )^\frac{\beta p}{r} \\
&\le C \delta T [[u]]^{p} 
+ C [[u]]^{\alpha p}  [[u]]^{\beta p}
= C [[u]]^{p} (\delta T+1).
\end{aligned}
\eeq
Therefore, we obtain \eqref{Sineq1}.
\end{proof}

The following parabolic Sobolev  embedding with spatial weights will be useful in this paper.

\begin{lemma}\label{ParaSob-3}
 Given $W(x,t)>0$ on $Q_T$. Suppose two numbers $m$ and $r $ satisfy $2n/(n+2)\le r\le 2$, $r<n$ and $r<m < r^*$. 
Let 
\beq\label{pm}
 p= 2+m -\frac{2m}{r^*}.
\eeq 
Then 
\beq\label{weight1}
\|u\|_{L^p(Q_T)}
\le C  [[u]]_{2,m,W;T} \Big\{ T^{1/p}+  \esssup_{t\in[0,T]} \Big (\int_U W(x,t)^{-\frac r{m-r}}dx\Big )^\frac{m-r}{pr}\Big \},
\eeq
where
\beq\label{normquad}
[[u]]_{2,m,W;T}=\esssup_{t\in [0,T]} \|u(t)\|_{L^2(U)}+\Big(\int_0^T \int_U W(x,t)|\nabla u|^m dx dt\Big)^\frac 1 m.
\eeq
Consequently, under the Strict Degree Condition, when $m=2$ and $r=2-a$ we have
\beq\label{Wemb}
\|u\|_{L^p(Q_T)}
\le C  [[u]]_{2,2,W;T} \Big\{ T^{1/p}+  \esssup_{t\in[0,T]} \Big (\int_U W(x,t)^{-\frac{2-a}{a}}dx\Big )^\frac{2}{p(2-a)}\Big \},
\eeq
where $2<p<(2-a)^*$ given explicitly by
\beq\label{pkey}
p=4\Big(1-\frac 1{(2-a)^*}\Big).
\eeq

In \eqref{weight1} and \eqref{Wemb} above, we can remove $T^{1/p}$ whenever $u$ vanishes on the boundary of $U$.
\end{lemma}
\begin{proof}
 In this proof we denote $[[\cdot]] = [[\cdot]]_{2,m,W;T}$. 
By H\"older's inequality,
\beq\label{stemp}
\|\nabla u\|_{L^r(U)} 
=  \Big( \int_U W^{\frac rm} |\nabla u|^r \cdot  W^{-\frac rm} dx \Big)^\frac1r 
\le \Big (\int_U W|\nabla u|^m dx\Big )^\frac 1m \Big (\int_U W^{-\frac r{m-r}}dx\Big )^{\frac1r-\frac1m}.
\eeq
By \eqref{rfact} and the condition on $r$ we have $r^*\ge 2$. With exponent $p$ in \eqref{pm}, we write
\beq\label{manyp} 
p-2= \frac{m(r^*-2)}{r^*},
\quad p-r^*= \frac{(r^*-2)(m-r^*)}{r^*}, 
\quad p-m=\frac{2(r^*-m)}{r^*},
\eeq
then we have $2\le p\le r^*$ and $p>m$.
Let $\alpha$ and $\beta$ be defined as in \eqref{powers}. 
Then applying \eqref{e1} and \eqref{stemp} yields 
\begin{multline}\label{e4}
 \norm{u}_{L^p(Q_T) }^p \le C \delta T [[u]]^p+C [[u]]^{\alpha p}
 \Big[ \int_0^T \Big (\int_U W|\nabla u|^m dx\Big )^\frac {\beta p} m  d\tau \Big]\\
\cdot \esssup_{[0,T]} \Big (\int_U W(x,t)^{-\frac r{m-r}}dx\Big )^{\beta p(\frac1r-\frac1m)} .
\end{multline}
From \eqref{powers} and \eqref{manyp}, we have 
$\beta =\frac{r^*(p-2)}{p(r^*-2)}=\frac  m p.$ 
Therefore, we rewrite \eqref{e4} as
\begin{align*}
 \norm{u}_{L^p(Q_T) }^p 
&\le C \delta T [[u]]^p+C [[u]]^{\alpha p}  \Big (\int_0^T \int_U W|\nabla u|^m dx  d\tau \Big )^\frac {\beta p} m\esssup_{[0,T]} \Big (\int_U W(x,t)^{-\frac r{m-r}}dx\Big )^\frac{m-r}r\\
&\le C \delta T [[u]]^p+C [[u]]^{\alpha p}  [[u]]^{\beta p} \esssup_{[0,T]} \Big (\int_U W(x,t)^{-\frac r{m-r}}dx\Big )^\frac{m-r}r\\
&\le C  [[u]]^p\Big\{ \delta T + \esssup_{[0,T]}\Big (\int_U W(x,t)^{-\frac r{m-r}}dx\Big )^\frac{m-r}r\Big\}.
\end{align*}
Thus, we obtain \eqref{weight1}.
In particular, when $m=2$ and $r=2-a$ then \eqref{pm} becomes \eqref{pkey}. 
Under the Strict Degree Condition, requirements on $r$ and $m$ are met.
Then \eqref{Wemb} follows  \eqref{weight1}.     
\end{proof}

Another type of embedding will be proved in Lemma \ref{LUK}.




\myclearpage
\section{Estimates of solutions}
\label{supestimate}

Let $U$ be a bounded, open, connected subset of $\R^n$ with boundary $\Gamma$ of class $C^2$.
We consider a fluid flow in $U$ that satisfies the generalized Forchheimer 
equation \eqref{gForch} with a fixed $g(s)=g(s,\vec a)\in FP(N,\vec \alpha)$.
We study the resulting parabolic equation for the pressure $p=p(x,t)$:
\beq\label{eqorig} \frac{\partial p}{\partial t}=\nabla \cdot (K(|\nabla p|)\nabla p), \quad x\in U,\ t>0.\eeq
In addition, we assume the flux condition on the boundary:
\beqs u\cdot \vec\nu =\psi(x,t),\quad x\in \Gamma,\ t>0,\eeqs
where $\vec\nu$ is the outward normal vector on $\Gamma$ and the flux $\psi(x,t)$ is known. Hence, by \eqref{u-forma} we have
\beq\label{BC} -K(|\nabla p|)\nabla p \cdot \vec\nu=\psi \quad \hbox{on}\quad \Gamma\times (0,\infty) .\eeq
The initial data 
\beq \label{In-Cond} p(x,0) = p_0(x) \text{ is given.}\eeq
We will focus on the IBVP \eqref{eqorig}, \eqref{BC} and \eqref{In-Cond}. By integrating \eqref{eqorig} over $U$, we easily find
\beqs \ddt \int_U p(x,t)dx = \int_\Gamma K(|\nabla p|)\nabla p\cdot \vec\nu d\sigma=-\int_\Gamma \psi(x,t)d\sigma,\quad t>0. \eeqs
(Here $d\sigma$ is the surface area element.)
By the continuity of $\int_U p(x,t)dx$ and $\int_\Gamma \psi(x,t)d\sigma$ on $[0,\infty)$, we assert
\beq\label{pavg-rel} \int_U p(x,t)dx =\int_U p(x,0)dx - \int_0^t\int_\Gamma \psi(x,\tau)d\sigma d\tau,\quad t\ge 0.\eeq

Let $\bar p(x,t)=p(x,t)-\frac1{|U|}\int_U p(x,t)dx$, for $x\in U$ and $t\ge 0$, where $|U|$ denotes the volume of $U$.
Then $\bar p$ satisfies the zero average condition 
\beq\label{pbar0}
\int_U \bar{p}(x,t) dx=0 \quad \hbox{for all }t\ge 0.
\eeq
It follows from \eqref{pavg-rel} that for $t\ge 0$,
\begin{equation}
\label{pbargam}
\bar{p}(x,t)=p(x,t)-\frac{1}{|U|}\int_U p_0(x) dx+\frac{1}{|U|}\int_0^t\int_\Gamma \psi(x,\tau) d\sigma d\tau. 
\end{equation}
We call $\bar p(x,t)$ the shifted solution. 
Then $\bar{p}$ satisfies the following IBVP
\beq\label{eqgamma}
\left \{ 
\begin{array}{lll}
& \dfrac{\partial \bar{p}}{\partial t}=\nabla \cdot (K(|\nabla \bar{p}|)\nabla \bar{p}) + \frac{1}{|U|}\int_\Gamma \psi(x,t) d\sigma,& \quad x\in  U, t>0,\\
& -K(|\nabla\bar{p}|)\nabla\bar{p} \cdot \vec\nu=\psi(x,t), &\quad x\in \Gamma, t>0,\\
& \bar p(x,0)=p_0(x)- \frac{1}{|U|}\int_U p_0(x) dx, &\quad x\in U.
\end{array} \right.
\eeq
Note that even in the linear case, i.e., $K(\xi)\equiv const.$ and $\psi(x,t)$ is uniformly bounded on $\Gamma\times[0,\infty)$, the solution $p(x,t)$ can still be unbounded as $t\to\infty$.
Therefore instead of estimating $p(x,t)$ directly, we will estimate $\bar p(x,t)$. 
Thanks to the explicit relation \eqref{pbargam} between $p$ and $\bar p$, the obtained results will have clear physical interpretations and the estimates for  $p(x,t)$ can be easily retrieved from those for $\bar p(x,t)$.

The existence of weak solutions can be treated by the theory of nonlinear monotone operators (cf. \cite{BrezisMonotone,s97,z90}, see also section 3 of \cite{HIKS1} for our proof for the Dirichlet boundary condition). The regularity of weak solutions is treated in \cite{DiDegenerateBook}.
For simplicity, we always assume that the initial data $p_0$ and boundary data $\psi$ are sufficiently smooth and the solution of \eqref{eqgamma} exists  \textit{for all} $t\ge 0$. Also, sufficient regularity of the solution is assumed.
For example, when we estimate $L^\infty$-norm of $\bar p$ we require $\bar p\in C(\bar U\times [0,\infty))$;
when we estimate $L^\infty$-norm of $\bar p_t$ we require $\bar p_t\in C( U\times (0,\infty))$;
when we estimate $L^s$-norm of $\nabla p$ we require $\nabla p\in C( [0,\infty),L_{loc}^s(U))$.

In the following subsections, we derive various estimates for pressure and its derivatives. 
These estimates are important by themselves and for the next sections when we study the stability of the solutions.

%
\begin{notation}
In estimates below, constants $C$'s always depend on the dimension $n$, domain $U$ and the Forchheimer polynomial $g(s,\vec a)$. Additional dependence will be specified as needed.

We use short-hand notation $\| f \|_{L^2}$, $\| f \|_{L^\infty}$ for the norms  $\| f \|_{L^2(U)}$, $\|f\|_{L^\infty(U)}$ whenever $f$ is defined on $U$. Similarly, if a function $\phi$ is defined on $\Gamma$, we use short-hand notation 
$\| \phi \|_{L^2}$, $\|\phi \|_{L^\infty}$ for the norms  $\| \phi \|_{L^2(\Gamma)}$, $\|\phi\|_{L^\infty(\Gamma)}$.

If $f(x,t)$ is a function of two variables, we denote by $f(t)$ the function $t\to f(\cdot,t)$, therefore $\|f(t)\|_{L^2}$ means $\|f(\cdot,t)\|_{L^2}$ and $\sup_{[0,T]}\| f\|_{L^2}$ means $\sup\{ \| f(\cdot,t)\|_{L^2}:t\in[0,T]\}$. 

Throughout, we denote $Q_T=U\times (0,T)$.
Also, we use the following notation for partial derivatives: $\partial p/\partial t= p_t$ and $\partial /\partial x_i=\partial_i$.
\end{notation}

\subsection{Estimates for pressure}\label{pressuresubsec}
We begin our subsection by introducing the following result on the $L^\infty$-estimates of for the $\bar p$  of  the IBVP \eqref{eqgamma}. 
Throughout the paper, we denote 
\beq \label{newmu}
{\muex_1}=(2-a)(n+2)/n.
\eeq
\begin{proposition} \label{p-T.est} \asdc. Then, 
there is a constant $C>0$ such that any $T>0$, the following inequality holds
\beq\label{pbaronly}
\sup_{[0,T]} \|\bar p\|_{L^\infty}\le C \Big\{ \|\bar p_0\|_{L^\infty} + (1+T)^\frac2{2-a} (\sup_{[0,T]} \|\psi\|_{L^\infty}+1)^\frac{1}{1-a} \Big\}.
\eeq
\end{proposition}
\begin{proof}
We follow the celebrated De Giorgi's technique. Although, the method is standard (see, for example, \cite{LadyParaBook68}), calculations are tedious and new. Thus, for completeness, we provide its details here. For any $k\ge0$, let us denote 
\begin{align*}
& \bar p^{(k)} =\max\{\bar p-k,0\},
\quad S_{k}(t)=\{ x\in U: \bar p^{(k)}(x,t)\ge 0\}, \quad \sigma_k =\int_0^T |S_k(t)|dt, \\
& F_k    = \sup_{t\in [0,T]} \int_U |\bar p^{(k)}(x,t)|^2 dx + \int_0^T\int_U |\nabla\bar  p^{(k)}(x,t)|^{2-a} dxdt.
\end{align*}
Assume $k\ge \|\bar  p_0\|_{L^\infty}$, then $\bar p_0^{(k)}(x) =0$ a.e..
Multiplying the first equation of \eqref{eqgamma} by $\bar p^{(k)}$ and integrating 
the resultant over the domain $U$, we obtain
\begin{align*}
& \frac 12 \ddt \int_U |\bar p^{(k)}|^2 dx + \int_U K(|\nabla \bar p^{(k)}|)|\nabla \bar p^{(k)}|^2 dx
=\int_\Gamma \psi \bar p^{(k)} d\sigma  +\frac{1}{|U|}\int_\Gamma \psi(x,t) d\sigma\int_U \bar{p}^{(k)} dx.
\end{align*}
We bound $|\psi(x,t)|$ and $|\int_\Gamma \psi(x,t) d\sigma|$ by $C\|\psi(t)\|_{L^\infty}$, and apply the trace theorem to the boundary integral, and then H\"older's inequality to have
\begin{align*}
& \frac 12 \ddt \int_U |\bar p^{(k)}|^2 dx + \int_U K(|\nabla \bar p^{(k)}|)|\nabla \bar p^{(k)}|^2 dx
\le C \|\psi(t)\|_{L^\infty} \int_U \Big (|\bar p^{(k)}|+|\nabla \bar p^{(k)}| \Big ) dx \\
&\le C \|\psi(t)\|_{L^\infty} \left \{  \Big (\int_U |\bar p^{(k)}|^2dx\Big )^{1/2} |S_k(t)|^{1/2}+ \Big (\int_{S_k(t)} |\nabla \bar p^{(k)}|^{2-a}dx\Big )^\frac 1{2-a}|S_k(t)|^\frac{1-a}{2-a}\right \}.
\end{align*}
Using \eqref{Kestn} to compare $|\nabla \bar p^{(k)}|^{2-a}$ with $K(|\nabla \bar p^{(k)}|)|\nabla \bar p^{(k)}|^2  +1$, we obtain
\begin{align*}
& \frac 12 \ddt \int_U |\bar p^{(k)}|^2 dx + \int_U K(|\nabla \bar p^{(k)}|)|\nabla \bar p^{(k)}|^2 dx
\le C \|\psi(t)\|_{L^\infty} \Big \{  \Big (\int_U |\bar p^{(k)}|^2dx\Big )^{1/2} |S_k(t)|^{1/2}  \\
& \quad \quad + \Big (\int_{S_k(t)}K(|\nabla \bar p^{(k)}|)|\nabla \bar p^{(k)}|^2dx\Big )^\frac 1{2-a}|S_k(t)|^\frac{1-a}{2-a} + |S_k(t)|\Big \}.
\end{align*}
Let $\varepsilon>0$. By Young's inequality, we have
\begin{align*}
& \frac 12 \ddt \int_U |\bar p^{(k)}|^2 dx + \frac 12 \int_U K(|\nabla \bar p^{(k)}|)|\nabla \bar p^{(k)}|^2 dx
\le \varep  \int_U |\bar p^{(k)}|^2 dx  
\\
&\quad \quad+  C \Big\{  \varep^{-1} \|\psi(t)\|_{L^\infty}^2 +\|\psi(t)\|_{L^\infty}^\frac{2-a}{1-a} 
+\|\psi(t)\|_{L^\infty} \Big\}|S_k(t)|.
\end{align*}
Note that $p^{(k)}(0)=0$. For each $t \in [0,T)$, integrating the previous estimate on $(0, t)$, and the taking the supremum in $t$ yield
\begin{align*}
&\sup_{[0,T]} \int_U |\bar p^{(k)}|^2 dx + \int_0^T\int_U K(|\nabla \bar p^{(k)}|)|\nabla \bar p^{(k)}|^2 dxdt\\
&\le 4\varep \int_0^T \int_U |\bar p^{(k)}|^2 dx +C\int_0^T \Big( \varep^{-1} \|\psi(t)\|_{L^\infty}^2  
 + \|\psi(t)\|_{L^\infty}^\frac{2-a}{1-a} +\|\psi(t)\|_{L^\infty}\Big)|S_k(t)| dt\\
&\le 4\varep T \sup_{[0, T]} \int_U |\bar p^{(k)}|^2 dx +C\int_0^T \Big( \varep^{-1} \|\psi(t)\|_{L^\infty}^2  + \|\psi(t)\|_{L^\infty}^\frac{2-a}{1-a}+\|\psi(t)\|_{L^\infty}\Big)|S_k(t)| dt.
\end{align*}
By taking $\varep=1/(8T)$, we obtain
\begin{align*}
&\sup_{[0,T]} \int_U |\bar p^{(k)}|^2 dx + \int_0^T\int_U K(|\nabla \bar p^{(k)}|)|\nabla \bar p^{(k)}|^2 dxdt\\
&\le C \int_0^T\left[ T \|\psi(t)\|_{L^\infty}^2 + \|\psi(t)\|_{L^\infty}^\frac{2-a}{1-a} +  \|\psi(t)\|_{L^\infty} \right]|S_k(t)| dt\\
&\le C \int_0^T\left[ T \|\psi(t)\|_{L^\infty}^2 + \|\psi(t)\|_{L^\infty}^\frac{2-a}{1-a} +  1 \right]|S_k(t)| dt.
\end{align*}
Again, using \eqref{Kestn} to compare $K(|\nabla \bar p^{(k)}|)|\nabla \bar p^{(k)}|^2$ with $|\nabla \bar p^{(k)}|^{2-a}$ and using Young's inequality, 
one gets 
\begin{equation}\label{p-DG.cond}
F_k \eqdef \sup_{[0,T]} \int_U |\bar p^{(k)}|^2 dx + \int_0^T \int_U |\nabla\bar  p^{(k)}|^{2-a} dxdt \leq C_0 \alpha_T \sigma_k, 
\end{equation}
where $C_0$ is a constant depending on $a$, $n$ and $U$, and  $\alpha_T = T\|\psi\|_{L^\infty}^2 + \|\psi(t)\|_{L^\infty}^\frac{2-a}{1-a} +  1$. 

We  now iterate \eqref{p-DG.cond} to derive our desired estimate. Under the Strict Degree Condition \eqref{strictdeg}, ${\muex_1}>2$, hence, the parabolic Sobolev embedding in Lemma~\ref{ParSob-1} with $r=2-a$ implies 
\beq \label{Fk.co}
\| \bar p^{(k)}\|_{L^{\muex_1}(Q_T)}\le C(1+ T)^{1/{\muex_1}} (F_k^{1/2}+F_k^{1/(2-a)}).
\eeq

Next, let $M_0 \geq \|\bar p(\cdot, 0)\|_{L^\infty}$ be a fixed number which will be determined. 
For each $i \in \mathbb{N} \cup \{0\}$, let $k_i=M_0(2-2^{-i})$. 
Then $k_i$ is increasing in $i$, and therefore $S_{k_i}$,  $\sigma_{k_i}$ are decreasing. 
Also, note that $k_i \geq M_0\ge \|\bar p_0\|_{L^\infty}$ for all $i\ge 0$.  
We now define $\mathcal Q_k=\{ (x,t)\in U\times(0,T): \bar p(x,t)>k\}$. 

From \eqref{Fk.co} and the fact that
\beq \label{pk.co}
\|\bar  p^{(k_i)}\|_{L^{\muex_1}(Q_T)} \ge \|\bar  p^{(k_i)}\|_{L^{\muex_1}(\mathcal Q_{k_{i+1}})}\ge \norm{k_{i+1} -k_i }_{L^{\muex_1}(\mathcal Q_{k_{i+1}})}    \ge  (k_{i+1}-k_{i})\sigma_{k_{i+1}}^{1/{\muex_1}},
\eeq
it follows
\[
\sigma_{k_{i+1}}^{1/{\muex_1}}\le \frac{C(1+T)^{1/{\muex_1}}}{ k_{i+1}-k_{i}} \Big[F_{k_i}^{1/2}+F_{k_i}^{1/(2-a)}\Big].
\]
This together with \eqref{p-DG.cond} yield  
\[
\sigma_{k_{i+1}}^{1/{\muex_1}}\le C (1+T)^{1/{\muex_1}}\frac{ 2^i }{M_0}
\Big[ (\alpha_T \sigma_{k_i})^{1/2}+ (\alpha_T \sigma_{k_i})^{1/(2-a)}\Big].
\]
Equivalently,
\[
\sigma_{k_{i+1}}\le \frac{ C (1+T) \alpha_T^{{\muex_1}/(2-a)} }{M_0^{\muex_1}}2^{{\muex_1}i} 
 ( \sigma_{k_i}^{{\muex_1}/2}+ \sigma_{k_i}^{{\muex_1}/(2-a)}).
\]
Let us denote $Y_i = \sigma_{k_i}$, $D =  \frac{C (1+T) \alpha_T^{{\muex_1}/(2-a)} }{M_0^{\muex_1}}$, ${\theta_1} = {\muex_1}/(2-a) -1 >0$ and ${\theta_2} = {\muex_1}/2 - 1 >0 $. We now obtain   
\[  Y_{i+1} \leq D 2^{{\muex_1}i} [Y_i^{1+{\theta_1}} + Y_i^{1+{\theta_2}}].
\]
Note that $k_0=M_0\ge \|\bar p_0\|_{L^\infty}$ and 
\[ Y_0=\sigma_{M_0}\le |Q_T|= |U|T\le |U|(T+1). \]
To apply \eqref{Y0mcond} with $m=2$ in Lemma \ref{multiseq}, we choose $M_0$ sufficiently large such that  
$$Y_0 \leq C \min\{D^{-1/{\theta_2}}, D^{-1/{\theta_1}}\},\quad \text{or sufficiently,} $$
\begin{align*}
|U|(T+1) \le C \min\Big\{  ((1+T) \alpha_T^{{\muex_1}/(2-a)} M_0^{-{\muex_1}})^{-1/{\theta_2}}, ((1+T) \alpha_T^{{\muex_1}/(2-a)} M_0^{-{\muex_1}})^{-1/{\theta_1}}   \Big \}.
\end{align*}
It suffices to have
\begin{equation*} \label{M-p.cond}
\begin{split}
& M_0\ge C (1+T)^\frac{1+{\theta_2}}{{\muex_1}} \alpha_T^{1/(2-a)} = C (1+T)^\frac{1}{2} \alpha_T^\frac1{2-a}\quad  
\text{and} \\ 
& M_0\ge C (1+T)^\frac{1+{\theta_1}}{{\muex_1}} \alpha_T^{1/(2-a)} = C (1+T)^\frac{1}{2-a} \alpha_T^\frac1{2-a}.
\end{split}
\end{equation*}
These and the initial requirement $M_0 \geq \|\bar{p}_0\|_{L^\infty}$ are satisfied if 
we choose  
\[
M_0\ge C \Big\{ \|\bar p_0\|_{L^\infty} + (1+T)^\frac1{2-a} \alpha_T^\frac{1}{2-a} \Big\}.
\] 
Note that
$ (1+T)\alpha_T
\le C  (T+1)^2 (\sup_{[0,T]}\|\psi\|_{L^\infty}+1)^\frac{2-a}{1-a}$.
Therefore we finally select
\[
M_0=C_1 \Big\{ \|\bar p_0\|_{L^\infty} + (1+T)^\frac2{2-a} (\sup_{[0,T]}\|\psi\|_{L^\infty}+1)^\frac{1}{1-a} \Big\}.
\] 
Now, applying \eqref{Y0mcond} in Lemma \ref{multiseq}, we have $\sigma_{2M_0}=\displaystyle{\lim_{i \rightarrow \infty}Y_i} = 0$.
Thus, $\bar p(x,t)\le 2M_0$ in $Q_T$. Note that we used the continuity of $\bar p(x,t)$ on $\bar Q_T$. 
By using the same argument with $p$ replaced by $-p$ and $\psi$ replaced by $-\psi$, one can show that $-\bar p \le 2M_0$  in $Q_T$. 
Therefore, \eqref{pbaronly} follows and the proof is complete.
\end{proof}

We emphasize that the $L^\infty$-estimate of $\bar{p}$ in Proposition \ref{p-T.est} depends on $L^\infty$-norm of the initial data and on $T$. In the next proposition, we give a $L^\infty$-estimate which removes the former dependence. The result will be used flexibly to derive the time-uniform $L^\infty$-estimate of $\bar{p}$.

\begin{proposition} \label{local-L-infty} \asdc.  
Then, there is a constant $C>0 $ such that for any $T_0\ge 0$, $T>0$, $\delta\in (0,1]$ and $\theta\in (0,1)$, the following inequality holds
\begin{multline}\label{mainest-1}
\sup_{[T_0+\theta T,T_0+T]} \|\bar p\|_{L^\infty}\le C\Big\{ \sqrt{\alphazero}
+ (T+1)^\frac{n}{4-(n+2)a} \Big(1+\frac 1 { \delta^a \theta T}\Big)^\frac{n+2}{4-(n+2)a} \\
\cdot \Big(\|\bar p\|_{L^2(U\times (T_0,T_0+T))} + \|\bar p\|_{L^2(U\times (T_0,T_0+T))}^\frac{4}{4-(n+2)a}\Big) \Big\},
\end{multline}
where   
\beqs
\alphazero=\alphazero_{T_0,T} \eqdef \delta^{2-a}+ T\delta^{-a} \sup_{[T_0,T_0+T]}\|\psi\|_{L^\infty}^2 + \delta^{-\frac {a(2-a)}{1-a}} \sup_{[T_0,T_0+T]} \|\psi\|_{L^\infty}^\frac{2-a}{1-a}.
\eeqs
\end{proposition}
\begin{proof} 
Without loss of generality, we assume $T_0=0$. Again, we use De Giorgi's iteration technique. But, unlike in the previous one, we will iterate using the $L^2$-norm of $\bar{p}$. The proof is therefore significantly different from that of Proposition \ref{p-T.est} and we give it here in full details. 

For each $k\ge 0$, let $\bar p^{(k)}$ be as 
in the proof of the previous theorem. Also, let $\zeta=\zeta(t)$ be a non-negative cut-off function on $[0,T]$ 
with $\zeta(0)=0$. Multiplying \eqref{eqgamma} by $\bar p^{(k)} \zeta$ and integrating over $U$,  we have
\beq\label{prediff}
 \frac12 \int_U \frac{\partial |\bar p^{(k)}|^2}{\partial t} \zeta dx +\int_U K(|\nabla\bar  p^{(k)}|) |\nabla\bar  p^{(k)}|^2 \zeta dx
=\int_\Gamma \psi \bar p^{(k)} \zeta d\sigma +  \frac{1}{|U|}\int_\Gamma \psi(x,t) d\sigma\int_U \bar p^{(k)} \zeta dx.
\eeq

Denote by $\chi_k(x,t)$ the characteristic function of the set $\{ (x,t)\in \in U\times (0,T): \bar p^{(k)}(x,t) > 0\}$. 
Then, using \eqref{Km} in Lemma \ref{lem21} to estimate $K(|\nabla\bar  p^{(k)}|) |\nabla\bar  p^{(k)}|^2\ge C\delta^a |\nabla\bar  p^{(k)}|^{2-a} - C\delta^2$,  and estimating the right-hand side of \eqref{prediff} the same way as in the proof of Proposition \ref{p-T.est}, we have
\begin{align*}
&  \frac12 \ddt \int_U |\bar p^{(k)}|^2 \zeta dx  -  \frac12 \int_U |\bar p^{(k)}|^2 \zeta_t dx +C\delta^a \int_U |\nabla \bar p^{(k)}|^{2-a} \zeta dx -C\delta^2 \int_U \chi_k \zeta dx\\
& \le C \|\psi(t)\|_{L^\infty} \int_U  (|\bar p^{(k)}| \zeta +|\nabla \bar p^{(k)}|\zeta)  dx.
\end{align*}
Let $\varepsilon >0$. Then it follows from Young's inequality that 
\begin{align*}
&  \frac12 \ddt \int_U |\bar p^{(k)}|^{2} \zeta dx   +C_1\delta^a \int_U |\nabla \bar p^{(k)}|^{2-a} \zeta dx \\
& \le \frac12 \int_U |\bar p^{(k)}|^2 \zeta_t dx + C\delta^2\int_U \chi_k(t)\zeta dx 
+ \varep \int_U |\bar p^{(k)}|^2 \zeta  dx + C \varep^{-1}  \|\psi\|_{L^\infty}^2 \int_U \chi_k \zeta dx\\
&\quad + \frac{ C_1 \delta^{a} }2 \int_U |\nabla \bar p^{(k)}|^{2-a} \zeta  dx
+ C \delta^{-\frac{a}{1-a}} \|\psi(t)\|_{L^\infty}^\frac{2-a}{1-a} \int_U \chi_k \zeta dx.
\end{align*}
Now, integrating this inequality on $(0, t)$ and taking supremum of the resultant on $[0,T]$, we obtain
\begin{align*}
& \frac{1}{2} \sup_{[0,T]} \int_U |\bar p^{(k)}|^2 \zeta dx   +\frac{C_1\delta^a}2  \int_0^T\int_U |\nabla \bar p^{(k)}|^{2-a} \zeta dxdt \\
& \le  \int_0^T\int_U |\bar p^{(k)}|^2 \zeta_t dxdt
+ 2 \varep T \sup_{[0,T]} \int_U |\bar p^{(k)}|^2 \zeta  dx \\
&\quad \quad +C \Big[ \delta^2 + \varep^{-1} \|\psi\|_{L_t^\infty(0,T; L_x^\infty)}^2 +\delta^{-\frac a{1-a}} \|\psi\|_{L_t^\infty(0,T; L_x^\infty)}^\frac{2-a}{1-a}\Big] \int_0^T\int_U  \chi_k\zeta dxdt.
\end{align*}
By choosing $\varepsilon=1/(8T)$, we obtain
\beq\label{prena}
\begin{aligned}
& \sup_{[0,T]} \int_U |\bar p^{(k)}|^2 \zeta dx   + \int_0^T\int_U |\nabla \bar p^{(k)}|^{2-a} \zeta dx dt\\
&\leq   C\delta^{-a}\int_0^T\int_U |\bar p^{(k)}|^2 \zeta_t dx dt  +C\alphazero \int_0^T\int_U  \chi_k\zeta dx dt ,
\end{aligned}
\eeq
where
$\alphazero=\delta^{2-a}+ T\delta^{-a} \sup_{[0,T]} \|\psi\|_{L^\infty}^2 + \delta^{-\frac{a(2-a)}{1-a}}  \sup_{[0,T]} \|\psi\|_{L^\infty}^\frac{2-a}{1-a}.$
We will iterate this relation with $k=k_{i+1}$ and $\zeta=\zeta_i$. The choice of $k_i$ and $\zeta_i(t)$ is as follows. 
Let  $t_i =\theta T( 1- 2^{-i})$ for all $i\ge 0$, and  let $\zeta_i$ be a piecewise linear function with 
$\zeta_i(t)=0$ for $t\le t_i$, $\zeta_i(t)=1$ for $t\ge t_{i+1}$ and 
\[|\zeta_{it}(t)|\le \frac 1{t_{i+1}-t_i}= \frac{ 2^{i+1} }{\theta T} \quad\forall\ t\in [0,T].\]
Next, let $M_0$ be a fixed positive number which will be determined, and let $k_i= M_0(1-2^{-i})$, for $i\ge 0$.
Note that, because of our choices,  $k_0=0$ and $k_{i+1}-k_i = 2^{-i-1}M_0$ and  $t_0=0<t_1<\ldots<\theta T$.
We also denote
$A_{i,j} =\{(x,t):u(x,t)>k_i,\ t\in(t_j,T)\}$ and $A_i=A_{i,i}$.
Then, by applying \eqref{prena} with $k=k_{i+1}$ and $\zeta=\zeta_i$, we obtain
\beq\label{pren1}
\begin{aligned}
&  \sup_{[0,T]} \int_U |\bar p^{(k_{i+1})}|^2 \zeta_i dx   + \int_0^T\int_U |\nabla \bar p^{(k_{i+1})}|^{2-a} \zeta_i dx dt \\
& \le  \delta^{-a}\int_0^T\int_U |\bar p^{(k_{i+1})}|^2 (\zeta_{i})_t dxdt
 +C\alphazero \Big (\int_0^T \int_U \chi_{k_{i+1}} \zeta_i dxdt\Big ).
\end{aligned}
\eeq
Now, we define
$F_{i} = \sup_{[t_{i+1},T]} \int_U |\bar p^{(k_{i+1})}|^{2} \zeta_i  dx   
+ \int_{t_{i}}^T\int_U |\nabla \bar p^{(k_{i+1})}|^{2-a} \zeta_i  dx dt. $
Then it follows from \eqref{pren1} that
\beqs
F_{i} 
 \le   \delta^{-a} \int_{t_i}^{t_{i+1}} \int_U |\bar p^{(k_{i+1})}|^2 (\zeta_i)_t dxdt
+C \alphazero \Big (\int_{t_{i+1}}^T\int_U \chi_{k_{i+1}} dx dt\Big ).
\eeqs
Therefore,
\beq \label{firstFn}
F_{i}\le C  2^i (\theta T)^{-1} \delta^{-a} \|\bar{p}^{(k_{i+1})}\|_{L^2(A_{i+1,i})}^2 +C\alphazero |A_{i+1,i}|.
\eeq
Since
$ \|\bar  p^{(k_i)}\|_{L^2(A_i)}\ge  \|\bar  p^{(k_{i})}\|_{L^2(A_{i+1,i})}\ge (k_{i+1}-k_i) |A_{i+1,i}|^{1/2}$, 
we see that 
\begin{equation} \label{An.n}
  |A_{i+1,i}| \le (k_{i+1}-k_i)^{-2} \|\bar  p^{(k_i)}\|_{L^2(A_i)}^{2} = 4^{i} M_0^{-2}\|\bar  p^{(k_i)}\|_{L^2(A_i)}^{2}.
\end{equation}
From \eqref{firstFn} and \eqref{An.n}, we get
\beq\label{Fn}
F_{i} \le  C  \delta^{-a} (\theta T)^{-1} 2^i  \|\bar p^{(k_{i+1})}\|_{L^2(A_{i+1,i})}^2
 +\alphazero 4^{i} M_0^{-2}\| \bar p^{(k_i)}\|_{L^2(A_i)}^{2}
\le   C C_\delta 4^{i} \|\bar  p^{(k_i)}\|_{L^2(A_i)}^{2},
\eeq
where    
$C_\delta=\delta^{-a}(\theta T)^{-1}+\alphazero M_0^{-2}.$
Under the Strict Degree Condition, $2<{\muex_1}<(2-a)^*$, then applying Lemma~\ref{ParSob-1} with $r=2-a$ yields 
\begin{align*} 
&  \|\bar  p^{(k_{i+1})}\|_{L^{\mu_1}(A_{i+1,i+1})} =\Big(\int_{t_{i+1}}^T \int_U |\bar  p^{(k_{i+1})}|^{\muex_1} dx dt  \Big)^{1/{\muex_1}}\\
                                             &\le C(1 +(T-t_{i+1})^{1/{\muex_1}} )\Big[ \sup_{[t_{i+1}, T]} \norm{\bar  p^{(k_{i+1})} }_{L^2(U)} +  \Big(\int_{t_{i+1}}^T \int_U |\nabla \bar  p^{(k_{i+1})} |^{2-a} dx)dt\Big)^{1/(2-a)}\Big]  \\
                                            &\le C(1 +T)^{1/{\muex_1}} \Big[ \sup_{[t_{i+1}, T]} \norm{\bar  p^{(k_{i+1})}\zeta_i }_{L^2(U)} +  \Big(\int_{t_{i+1}}^T \int_U |\nabla \bar  p^{(k_{i+1})} |^{2-a} \zeta_i dxdt\Big)^{1/(2-a)}\Big].
\end{align*}
Since $t_i\le t_{i+1}$, it follows that
\begin{equation} \label{Lr-p}
 \| \bar p^{(k_{i+1})}\|_{L^{\muex_1}(A_{i+1,i+1})}
\le C (T+1)^{1/{\muex_1}} (F_{i}^{1/2}+F_{i}^{1/(2-a)}).
\end{equation} 
Then, it follows from H\"{o}lder's inequality and \eqref{Lr-p} that
\begin{align*}
& \|\bar  p^{(k_{i+1})}\|_{L^2(A_{i+1,i+1})}
 \le \|\bar  p^{(k_{i+1})}\|_{L^{\muex_1}(A_{i+1,i+1})} |A_{i+1,i+1}|^{1/2-1/{\muex_1}}\\
& \le \|\bar  p^{(k_{i+1})}\|_{L^{\muex_1}(A_{i+1,i+1})} |A_{i+1,i}|^{1/2-1/{\muex_1}}
 \le C (T+1)^{1/{\muex_1}} (F_{i}^{1/2}+F_{i}^{1/(2-a)})  |A_{i+1,i}|^{1/2-1/{\muex_1}}.
\end{align*}
This and \eqref{An.n}, \eqref{Fn} give
\begin{align*}
& \|\bar  p^{(k_{i+1})}\|_{L^2(A_{i+1,i+1})}\\
&\le 4^i C   (T+1)^{1/{\muex_1}} \Big ( C_\delta^{1/2}\|\bar  p^{(k_i)}\|_{L^2(A_i)}+C_\delta^{1/(2-a)}\|\bar  p^{(k_i)}\|_{L^2(A_i)}^{2/(2-a)}\Big ) M_0^{-1+2/{\muex_1}} \| p^{(k_i)}\|_{L^2(A_i)}^{1-2/{\muex_1}}\\
&\le 4^i C  (T+1)^{1/{\muex_1}} M_0^{-1+2/{\muex_1}}  \Big (C_\delta^{1/2} \|\bar  p^{(k_i)}\|_{L^2(A_i)}^{2-2/{\muex_1}}+ C_\delta^{1/(2-a)}\|\bar  p^{(k_i)}\|_{L^2(A_i)}^{1-2/{\muex_1}+2/(2-a)}\Big ).
\end{align*}
Now, let $Y_i=\|\bar  p^{(k_i)}\|_{L^2(A_i)}$, we get  
\begin{equation} \label{Yn-uni}
Y_{i+1}\le D 4^i \Big (Y_i^{2-2/{\muex_1}}+Y_i^{1-2/{\muex_1}+2/(2-a)} \Big )= D 4^i \Big (Y_i^{1+\muex_2}+Y_i^{1+\muex_3} \Big ),
\end{equation}
where  $D=C  (T+1)^{1/{\muex_1}}  \big( C_\delta^{1/2}  + C_\delta^{1/(2-a)} \big) M_0^{-\muex_2}$, and   
\beq\label{muzero}
\muex_2= 1-\frac 2 {\muex_1} =\frac{4-a(n+2)}{(2-a)(n+2)},\quad \muex_3 = 2\Big (\frac{1}{2-a} - \frac{1}{{\muex_1}} \Big)=\frac{4}{(2-a)(n+2)}.
\eeq  

To be able to apply  \eqref{Y0mcond} with $m=2$ in Lemma~\ref{multiseq}, we will choose $M_0$ sufficiently large such that
\beq\label{condY1} 
Y_0  \leq C\min\{D^{-1/\muex_2}, D^{-1/\muex_3} \}.
\eeq
Since $Y_0\le \| \bar{p}\|_{L^2(U\times (0,T))}$, it suffices to choose $M_0$  such that  
\[ 
M_0 \ge  C (T+1)^\frac1{{\muex_1}\muex_2} (C_\delta^\frac12+ C_\delta^\frac 1{2-a} )^\frac1{\muex_2}  \|\bar{p}\|_{L^2(U\times (0,T))}, 
M_0\ge C (T+1)^\frac1{{\muex_1}\muex_2}  (C_\delta^\frac12+ C_\delta^\frac 1{2-a} )^\frac1{\muex_2}  \|\bar{p}\|_{L^2(U\times (0,T))}^{\muex_3/\muex_2}.
\]
Observe that, if $M_0\ge \sqrt {\alphazero}$  then  $C_\delta\le 1 + \frac 1 {\delta^a \theta T }\in (1, \infty) $. Thus, \eqref{condY1} holds for
\[
M_0=  C\Big\{ \sqrt{\alphazero}+ (T+1)^\frac1{{\muex_1}\muex_2} \big(1+\frac 1 { \delta^a \theta T}\big)^\frac1{(2-a)\muex_2}\Big( \|\bar p\|_{L^2(U\times (0,T))} +   \|\bar p\|_{L^2(U\times (0,T))}^{\muex_3/\muex_2} \Big)\Big\}.
\]
With this choice of $M_0$, applying \eqref{Y0mcond} with $m=2$ in Lemma~\ref{multiseq} to \eqref{Yn-uni}, we obtain 
$\displaystyle{\lim_{i\to\infty}}Y_i=0$. Consequently,
$
\int_{\theta T}^T\int_U |\bar p^{(M_0)}|^2 dxdt=0.
$
This implies $\bar p(x,t)\le M_0$ in $U\times (\theta T,T)$.  
To prove the lower bound $\bar{p}\ge -M_0$, we replace $p$ by $-p$ and $\psi$ by $-\psi$ and 
use the same argument.  We conclude
\[
|\bar p(x,t)|\le M_0 \quad  \text{in } U\times (\theta T,T). 
\]
The proof of \eqref{mainest-1} is complete. 
\end{proof}

Our next goal is to combine Propositions \ref{p-T.est} and \ref{local-L-infty} with $L^2$-estimates of $\bar{p}$ in \cite{HI2}
to derive the uniform (in time) $L^\infty$-estimate of $\bar{p}$. For our purpose, we need to introduce some notation. 
Firstly, let us define two functions $f(t)$ and $\tilde f(t)$ for $t\ge 0$ by
\beq\label{b}
\begin{split}
f(t) & =\|\psi(t)\|_{L^\infty}^2+\|\psi(t)\|_{L^\infty}^{\frac{2-a}{1-a}} \quad \text{and}\quad 
\tilde f(t)=\|\psi_t(t)\|_{L^\infty}^2+\|\psi_t(t)\|_{L^\infty}^{\frac{2-a}{1-a}}.
\end{split}
\eeq
Assume throughout that $\psi(\cdot,t)$ and $\psi_t(\cdot,t)$ belong $C([0,\infty),L^\infty(\Gamma))$, hence $f(t)$ and $\tilde f(t)$ belong to $C([0,\infty))$.
Whenever $f'(t)$ is mentioned, we implicitly assume that $f\in C^1((0,\infty))$. 
Let $M_f(t)$ be a continuous, increasing majorant of $f(t)$ on $[0,\infty)$. 
Let
\beq\label{Abeta}  A=\limsup_{t\to\infty} f(t)\quad \text{and}\quad \beta=\limsup_{t\to\infty} [f'(t)]^-.\eeq
Again, whenever $\beta$ is used in  subsequently statements, 
it is understood that $f(t)\in C^1((0,\infty))$.

Also, for a function $u(x,t)$ defined on $U\times [0,\infty)$ we denote
\beq \label{J.def}
 J_H[u](t)=\int_U H(|\nabla u(x,t)|)dx,
\eeq
where $H$ is the function defined by \eqref{Hxi}. Note from \eqref{Hcompare} that
\beq\label{JHcom}
c_1\int_U|\nabla u(x,t)|^{2-a} dx - c_3 \le J_H[u](t)\le c_2\int_U|\nabla u(x,t)|^{2-a} dx,\quad c_3>0.
\eeq

We now recall relevant estimates from \cite{HI2}.

\begin{theorem}[cf. \cite{HI2}, Theorem 4.3]\label{p-estimate} The following  estimates hold 
\begin{itemize}
\item[\textup{(i)}] For $t\ge 0$,
\beq\label{p-bar-bound1}
\norm{\bar p(t)}_{L^2}^2 +\int_0^t J_H[p](\tau)d\tau\le \norm{\bar p_0}_{L^2}^2 +C\int_0^t f(\tau)d\tau.
\eeq
\item[\textup{(ii)}] Assume that the Degree Condition holds,  then 
\beq\label{p-bar-bound}
\norm{\bar p(t)}_{L^2}^2 \le  \norm{\bar p_0}_{L^2}^2 +C(1+ M_f(t)^\frac{2}{2-a})\quad \text{for all } t\ge 0.
\eeq
Moreover, if $A<\infty$ then
\beq\label{nonzerobeta} \limsup_{t\to\infty} \norm{\bar p(t)}_{L^2}^2 \le C (A+A^\frac2{2-a}),\eeq
and if  $\beta<\infty$ then there is $T>0$ such that 
\beq\label{unboundp} \norm{\bar p(t)}_{L^2}^2 \le C\big(1+\beta^\frac{1}{1-a}+f(t)^\frac{2}{2-a}\big) \quad \text{for all }  t>T.\eeq
\end{itemize}
\end{theorem}



%

We finalize our main estimates for $L^\infty$-norm in the following theorem.
The exponent $\muex_2$ appearing below is already defined in \eqref{muzero}, and we also denote
\beq\label{mu1}
\muex_4= \frac{4}{(2-a)(4-a(n+2))}.
\eeq

\begin{theorem}\label{thm-1}
\asdc. Then there is a constant $C>0$ such that the following statements hold true. 
\begin{itemize}
\item[\textup{(i)}] For $t>0$,
\beq\label{p1}
\|\bar p(t)\|_{L^\infty} \le C\Big(1+t^{-\frac 1{\muex_2(2-a)}}\Big) \Big\{ 1+ \|\bar p_0\|_{L^2}^{\muex_4(2-a)} +M_f(t)^{\muex_4} \Big\},
\eeq
\beq \label{pmax}
\|\bar p(t)\|_{L^\infty} \le C\Big( 1+ \|\bar p_0\|_{L^\infty}+\|\bar p_0\|_{L^2}^{\muex_4(2-a)} +M_f(t)^{\muex_4} \Big).
\eeq

\item[\textup{(ii)}] If $A<\infty$ then
\beq\label{p-2} 
\limsup_{t\to\infty} \| \bar p(t)\|_{L^\infty} \le C \big(  1+  A^{\muex_4}  \big).
\eeq

\item[\textup{(iii)}] If  $\beta<\infty$ then there is $T>0$ such that 
\beq\label{p-3} \| \bar p(t)\|_{L^\infty} \le C\Big\{ 1   +\beta^\frac{\muex_4(2-a)}{2(1-a)}  + \sup_{[t-2,t]} \|\psi\|_{L^\infty}^\frac{\muex_4(2-a)}{1-a} \Big\}   \quad \hbox{for all }  t>T.\eeq
\end{itemize}
\end{theorem}
\begin{proof} 
(i) We prove \eqref{p1} first.
For $t\in(0,1)$, applying \eqref{mainest-1} to $T_0=0$, $T=t$, $\delta=1$ and $\theta =1/2$ we obtain
\begin{align*}
\|\bar p(t)\|_{L^\infty}
&\le C\Big\{ (t+1)^{\frac 1{2}}(1+ \sup_{[0,t]}\|\psi\|_{L^\infty})^\frac{2-a}{2(1-a)})\\
&\quad+ (t+1)^\frac1{\muex_1 \muex_2} \big(1+ 2 t^{-1}\big)^\frac1{(2-a)\muex_2} \Big(\|\bar p\|_{L^2(U\times (0,t))} + \|\bar p\|_{L^2(U\times (0,t))}^{\frac{4}{4-a(n+2)}}\Big) \Big\}\\
&\le C (1+ t^{-1})^\frac1{(2-a)\muex_2} \Big(1+\sup_{[0,t]}\|\psi\|_{L^\infty}^\frac{2-a}{2(1-a)} + \sup_{[0,t]}\|\bar p\|_{L^2}^{\frac{4}{4-a(n+2)}}\Big) \Big\}.
\end{align*}
By estimate \eqref{p-bar-bound},
\begin{align*}
\|\bar p(t)\|_{L^\infty}&\le C t^{-\frac 1{\muex_2(2-a)}}\Big\{1+ M_f(t)^\frac12+   \Big(1+ \norm{\bar p_0}_{L^2} + M_f(t)^{\frac 1{2-a}}\Big)^{\frac{4}{4-a(n+2)}} \Big\}\\
&\le C t^{-\frac 1{\muex_2(2-a)}}\Big\{ 1+ \norm{\bar p_0}_{L^2}^\frac{4}{4-a(n+2)} +M_f(t)^\frac{4}{(2-a)(4-a(n+2))} \Big\}.
\end{align*}
Therefore, we obtain \eqref{p1} for $t\in(0,1)$.
For $t\in[1,\infty)$, applying \eqref{mainest-1} to $T_0=t-1$, $T=1$, $\delta=1$, and $\theta=1/2$, we have
\beqs
\|\bar p(t)\|_{L^\infty}\le C\Big\{  1+\sup_{[t-1,t]} \|\psi\|_{L^\infty}^\frac{2-a}{2(1-a)} + \sup_{[t-1,t]} \norm{\bar p}_{L^2} + \sup_{[t-1,t]} \norm{\bar p(t)}_{L^2}^{\muex_4(2-a)} \Big\}.
\eeqs
Note that $\muex_4(2-a)> 1$, then H\"older's inequality gives
\beq\label{opp}
\|\bar p(t)\|_{L^\infty}\le C\Big\{  1+\sup_{[t-1,t]} \|\psi\|_{L^\infty}^\frac{2-a}{2(1-a)} + \sup_{[t-1,t]} \norm{\bar p}_{L^2}^{\muex_4(2-a)} \Big\}.
\eeq
Then, again, combining \eqref{opp} with \eqref{p-bar-bound} and using H\"older's inequality yield \eqref{p1} for $t\in[1,\infty)$.

To prove \eqref{pmax}, we note that in the case $t\ge 1$, \eqref{p1} implies \eqref{pmax}.
Now, for $t\in (0,1),$ we have  from inequality \eqref{pbaronly} that
\beq\label{pp}
\|\bar p(t)\|_{L^\infty}\le C \big( 1+\| \bar p_0\|_{L^\infty} + \sup_{[0,t]} \|\psi\|_{L^\infty}^\frac{1}{1-a} \big).
\eeq
Since $\|\psi(t)\|_{L^\infty}\le M_f(t)^\frac{1-a}{2-a}\le M_f(t)^{\mu_4}$, then \eqref{pmax} follows  \eqref{pp}.

(ii) Since $A<\infty$, we see that
\beq\label{limF}
\limsup_{t\to\infty} (\sup_{[t-1,t]} \|\psi\|_{L^\infty})^\frac{2-a}{2(1-a)}\le \limsup_{t\to\infty}  f(t)^\frac12 = A^\frac12. 
\eeq  
It follows \eqref{opp},  \eqref{nonzerobeta} and \eqref{limF} that
\beqs 
\limsup_{t\to\infty} \| \bar p(t)\|_{L^\infty} \le C \Big\{  1+A^\frac12+(A^\frac12+ A^{\frac 1{2-a}})^{\muex_4(2-a)}  \Big\},
\eeqs
which implies \eqref{p-2}. 

(iii) Because  $\beta<\infty$, using \eqref{unboundp} in \eqref{opp} there is $T>0$ such that  for all $t>T$,
\beqs\label{p-30} 
\| \bar p(t)\|_{L^\infty} \le C\Big\{ 1+ \sup_{[t-1,t]} \|\psi\|_{L^\infty}^\frac{2-a}{2(1-a)} 
      +\big[1+\beta^\frac{1}{2(1-a)}+\sup_{[t-1,t]} \|\psi\|_{L^\infty}^\frac{1}{1-a}\big]^{\muex_4(2-a)}    \Big\},\eeqs
and \eqref{p-3} follows.
\end{proof}

Concerning $\bar p(t)$ being small, as $t\to\infty$, rather than just being bounded,  we have the following result.

\begin{theorem}\label{thmsplit}
\asdc. For any $\varep>0$, there is $\delta_0>0$ such that
\beq\label{limsup:pbar0}
\text{if}\quad \limsup_{t\to\infty}\|\psi(t)\|_{L^\infty}<\delta_0 \quad\text{then}\quad \limsup_{t\to\infty} \|\bar p(t)\|_{L^\infty} < \varep.
\eeq
Consequently,  
\beq\label{limsup:pbar}
\text{if}\quad \lim_{t\to\infty}\|\psi(t)\|_{L^\infty}=0\quad\text{then}\quad \lim_{t\to\infty} \|\bar p(t)\|_{L^\infty} =0.
\eeq
\end{theorem}
\begin{proof} 
Applying \eqref{mainest-1} to $T_0=t-1$, $T=1$ and $\theta=1/2$, we have 
\beqs
\| \bar p(t)\|_{L^\infty} \le C\Big\{ \sqrt{\alphazero(t)}+ (1+ \delta^{-a})^{\frac 1{(2-a)\muex_2}} \Big(\|\bar p\|_{L^2(U\times (t-1,t))} + \|\bar p\|_{L^2(U\times (t-1,t))}^\frac{4}{4-a(n+2)}\Big) \Big\},
\eeqs
where
$\alphazero(t) = \delta^{2-a}+\delta^{-a}\sup_{[t-1,t]}\|\psi\|_{L^\infty}^2+ \delta^{-\frac {a(2-a)}{1-a}} \sup_{[t-1,t]}\|\psi\|_{L^\infty}^\frac{2-a}{1-a}.$
Then, it follows that
\beq\label{prelim0}
\|\bar p(t)\|_{L^\infty}\le C\sqrt{\alphazero(t)} + C (1+\delta^{-a})^\frac1{(2-a)\muex_2} \Big\{\sup_{[t-1, t]} \norm{\bar p}_{L^2} +\sup_{[t-1, t]} \norm{\bar p}_{L^2}^{\frac{4}{4-a(n+2)}}\Big\}.
\eeq
Let $\delta_0\in(0,1)$, then 
$A\le \delta_0^2+\delta_0^\frac{2-a}{1-a}\le 2\delta_0^2$,
and therefore, by \eqref{nonzerobeta},
\beqs
\limsup_{t\to\infty} (\sup_{[t-1, t]} \norm{\bar p}_{L^2} )\le C(A^\frac12+A^\frac1{2-a})\le C \delta_0.
\eeqs
Hence, we obtain
\begin{align*}
&\limsup_{t\to\infty}(1+\delta^{-a})^\frac1{(2-a)\muex_2} \Big\{\sup_{[t-1, t]} \norm{\bar p}_{L^2} +\sup_{[t-1, t]} \norm{\bar p}_{L^2}^{\frac{4}{4-a(n+2)}}\Big\}
\le C (1+\delta^{-a})^\frac1{(2-a)\muex_2}\delta_0.
\end{align*}
Also,
\beqs
\limsup_{t\to\infty}  \alphazero(t)\le  \delta^{2-a} +\delta^{-a}\delta_0^2 + \delta^{-\frac {a(2-a)}{1-a}} \delta_0^\frac{2-a}{1-a} \le \delta^{2-a} + 2\delta^{-\frac {a(2-a)}{1-a}} \delta_0^2. 
\eeqs
Therefore, for any $\delta>0$, we have from \eqref{prelim0} that
\beq\label{limwithdelta}
\limsup_{t\to\infty} \|\bar p(t)\|_{L^\infty} \le C\big(\delta^{1-\frac a 2} +[ \delta^{-\frac {a(2-a)}{2(1-a)}} +(1+\delta^{-a})^\frac1{(2-a)\muex_2} ]\delta_0\big).
\eeq 
Now, choose $\delta$ sufficiently small such that $2C\delta^{1-a/2} <\varepsilon$. Then we can choose $\delta_0>0$ even sufficiently smaller so that
$
 [\delta^{-\frac {a(2-a)}{2(1-a)}} +(1+\delta^{-a})^\frac1{(2-a)\muex_2}]\delta_0 < \delta^{1-a/2}.
$
From this, the desired estimates \eqref{limsup:pbar0} follows \eqref{limwithdelta}.
The statement \eqref{limsup:pbar} obviously is a consequence of \eqref{limsup:pbar0}. The proof is complete.
\end{proof}

\subsection{Estimates for pressure's gradient}
\label{Lsgrad-sec}

In this subsection, we establish the interior $L^s$-estimate of $\nabla p$, for all $s>0$. 
We follow the approach in \cite{LadyParaBook68}. 
First is a basic estimate for $\nabla p$ which prepares for our iteration later.

\begin{lemma} \label{s.it-L} For each $s\geq 0$, there is a constant $C>0$ depending on $s$ such that  for any $T>0$ and smooth cut-off function 
$\zeta(x) \in C_c^\infty(U)$, the following estimate holds
\begin{equation}\label{preLs2}
\begin{aligned}
 & \sup_{[0,T]} \int_U |\nabla p(x, t)|^{2s+2} \zeta^2 dx
+C  \int_0^T\int_U K(|\nabla p|) |\nabla^2 p|^2   |\nabla p|^{2s}  \zeta^2 dx dt \\
&\le  \int_U |\nabla p_0(x)|^{2s+2} \zeta^2 dx
+  C \int_0^T \int_U K(|\nabla p|)  |\nabla p|^{2s+2}|\nabla \zeta|^2 dx dt.
\end{aligned} 
\end{equation}
\end{lemma}
\begin{proof} 
In this proof, we use Einstein's summation convention, that is, when an index variable appears twice in a single term it implies summation of that term over all the values of the index. 
Multiplying the equation \eqref{eqorig} by 
$-\nabla \cdot (|\nabla p|^{2s}\nabla p \zeta^2)$ and integrating the resultant over $U$, we obtain
\begin{align*}
 \frac{1}{2s+2} \ddt \int_U |\nabla p|^{2s+2} \zeta^2 dx
= -\int_U \partial_i (K(|\nabla p|)\partial_i p) \, \partial_j (|\nabla p|^{2s} \partial_j p \zeta^2) dx.
\end{align*}
This equality and the integration by parts yield  
\begin{align*}
& \frac{1}{2s+2} \ddt \int_U |\nabla p|^{2s+2} \zeta^2 dx
= -\int_U \partial_j [K(|\nabla p|)\partial_i p] \, \partial_i [|\nabla p|^{2s} \partial_j p \zeta^2] dx\\
&= -\int_U \partial_j [K(|\nabla p|)\partial_i p] \, \partial_i \partial_j p \,  |\nabla p|^{2s}  \zeta^2 dx
-\int_U \partial_j [K(|\nabla p|)\partial_i p]\,  \partial_j p \,  |\nabla p|^{2s} \partial_i [\zeta^2] dx\\
&\quad -\int_U \partial_j [K(|\nabla p|)\partial_i p] \,\partial_j p   \, [2s |\nabla p|^{2s-2} \partial_i\partial_m p\partial_m p]\,   \zeta^2 dx.
\end{align*}	
We rewrite it in the following form
\begin{align*}
& \frac{1}{2s+2} \ddt \int_U |\nabla p|^{2s+2} \zeta^2 dx
= -\int_U \Big[\partial_{y_l} (K(|y|) y_i)\Big|_{y=\nabla p} \partial_j\partial_l p\Big]\partial_j \partial_i p  |\nabla p|^{2s}  \zeta^2 dx\\
&\quad -2\int_U\Big[ \partial_{y_l} (K(|y|) y_i)\Big|_{y=\nabla p} \partial_j\partial_l p \Big] \partial_j p \, |\nabla p|^{2s} \, \zeta \partial_i  \zeta dx\\
&\quad -2s\int_U \Big[\partial_{y_l} (K(|y|) y_i)\Big|_{y=\nabla p} \partial_j\partial_l p \Big]\partial_j p   \, (|\nabla p|^{2s-2} \partial_i\partial_m p\partial_m p)\,     \zeta^2 dx.
\end{align*}	
We denote the three terms  on the right-hand side by $I_1$, $I_2$ and $I_3$, and estimate each of them.
By \eqref{K-est-2} one can easily prove for any $y,z\in \R^n$ that
\begin{equation*}
 z^T \nabla (K(|y|)y) z\ge (1-a)K(|y|)|z|^2\quad \text{and} \quad  |\nabla (K(|y|)y)|\le (1+a)K(|y|).
\end{equation*}
It follows that 
\begin{align*}
 I_1
&=  -\int_U \partial_{y_l} [K(|y|) y_i]\Big|_{y=\nabla p} \partial_l(\partial_j p) \partial_i (\partial_j p) \,  |\nabla p|^{2s}  \zeta^2 dx\\
&\le -(1-a)\sum_j \int_U K(|\nabla p|) |\nabla(\partial_j p)|^2   |\nabla p|^{2s}  \zeta^2 dx\\
& = -(1-a)  \int_U K(|\nabla p|) |\nabla^2 p|^2 \,  |\nabla p|^{2s}  \zeta^2 dx.
\end{align*}
Moreover, we have 
\begin{equation*}
 |I_2|\le 2(1+a)\int_U K(|\nabla p|) |\nabla^2 p| |\nabla p|^{2s+1} \zeta  |\nabla \zeta| dx, 
\end{equation*}
and
\begin{align*}
 I_3
&= -2s\int_U \partial_{y_l} [K(|y|) y_i]\Big|_{y=\nabla p} (\partial_l\partial_j p\partial_j p)  (\partial_i \partial_m p \partial_m p)  \,  |\nabla p|^{2s-2}   \zeta^2 dx\\
&= -2s\int_U \partial_{y_l} [K(|y|) y_i]\Big|_{y=\nabla p} \Big(\partial_l \frac12|\nabla p|^2\Big)  \Big(\partial_i \frac12|\nabla p|^2\Big)\,   |\nabla p|^{2s-2}   \zeta^2 dx\\
&\le  -2(1-a)s\int_U K(| \nabla p|)\Big | \nabla \Big(\frac12|\nabla p|^2\Big)\Big |^2  |\nabla p|^{2s-2}   \zeta^2 dx\le 0.
\end{align*}
Combining these estimates together with Young's inequality, we see that 
\begin{align*}
 & \frac{1}{2s+2} \ddt \int_U |\nabla p|^{2s+2} \zeta^2 dx
+(1-a)  \int_U K(|\nabla p|) |\nabla^2 p|^2   |\nabla p|^{2s}  \zeta^2 dx\\
&\le 2(1+a)\int_U K(|\nabla p|) |\nabla^2 p| |\nabla p|^{2s+1} \zeta|\nabla \zeta| dx\\
& \le \frac{1-a}{2} \int_{U'} K(|\nabla p|) |\nabla^2 p|^2   |\nabla p|^{2s}  \zeta^2 dx + C \int_U K(|\nabla p|) |\nabla p|^{2s+2}|\nabla \zeta|^2 dx.
\end{align*}
Thus,
\beq\label{gradid}
\begin{aligned}
 & \frac{1}{2s+2} \ddt \int_U |\nabla p|^{2s+2} \zeta^2 dx
+\frac{1-a}{2}  \int_U K(|\nabla p|) |\nabla^2 p|^2   |\nabla p|^{2s}  \zeta^2 dx\\
&\le  C \int_U K(|\nabla p|) |\nabla p|^{2s+2}|\nabla \zeta|^2 dx.
\end{aligned}
\eeq
Integrating this inequality in time yields \eqref{preLs2}.
\end{proof}

In order to iterate \eqref{preLs2} in $s$ we need an embedding similar to Lemma 5.4 on page 93 in \cite{LadyParaBook68}.
For our degenerate equation, the following version has the key weight function $K(|\nabla w|)$. 

\begin{lemma}\label{LUK} For each $s\geq1$, there exists a constant $C>0$ depending on $s$ such that  for each smooth cut-off function 
$\zeta(x) \in C_c^\infty(U)$, the following inequality holds 
\begin{equation*}
\begin{aligned}
\int_U K(|\nabla w|) |\nabla w|^{2s+2} \zeta^2 dx
 \le  C \sup_{\text{\rm supp} \zeta } |w|^2  & \left \{ \int_U K(|\nabla w|)|\nabla w|^{2s-2} |\nabla^2 w|^2 \zeta^2 dx \right. \\
& \quad \quad \quad + \left. \int_U K(|\nabla w|)|\nabla w|^{2s}  |\nabla \zeta|^2 dx \right\},
\end{aligned} 
\end{equation*}
for every sufficiently regular function $w(x)$ such that the right hand side is well-defined.
\end{lemma}
\begin{proof} 
Again, the Einstein's summation convention is used in this proof. 

Let $I =  \int_U K(|\nabla w|) |\nabla w|^{2s+2} \zeta^2 dx$.
From direct calculations, we see that 
\begin{align*}
I 
&=\int_U K(|\nabla w|) |\nabla w|^{2s} \partial_i w\partial_i w \zeta^2 dx
 = - \int_U \partial_i (K(|\nabla w|) |\nabla w|^{2s} \partial_i w \zeta^2)\cdot  w dx\\
& =  - \int_U\Big( K'(|\nabla w|)\frac{\partial_i\partial_j w\partial_j w}{|\nabla w|} \Big) |\nabla w|^{2s} \partial_i w\cdot  \zeta^2 \cdot w dx - \int_U K(|\nabla w|)|\nabla w|^{2s} \Delta w \cdot \zeta^2  \cdot w dx\\
&\quad - 2s \int_U K(|\nabla w|)\Big(|\nabla w|^{2s-2} \partial_i\partial_j w \partial_j w\Big)\cdot  \partial_i w \cdot \zeta^2  \cdot w dx
-2 \int_U K(|\nabla w|)|\nabla w|^{2s} \partial_i w \zeta \partial_i  \zeta \cdot w dx.
\end{align*}
From this and \eqref{K-est-2}, it follows that 
\begin{align*}
I&\le a \int_U K(|\nabla w|)|\nabla^2 w|  |\nabla w|^{2s}  \zeta^2 |w| dx  
+ \int_U K(|\nabla w|)|\nabla w|^{2s} |\Delta w| \zeta^2  |w| dx\\
&\quad + 2 s \int_U K(|\nabla w|)|\nabla w|^{2s-2} |\nabla^2  w|  |\nabla w|^2 \zeta^2  |w| dx
+2 \int_U K(|\nabla w|)|\nabla w|^{2s} |\nabla w| \zeta |\nabla \zeta| |w|dx\\
& \le  C \int_U K(|\nabla w|)|\nabla w|^{2s} |\nabla^2 w| \zeta^2  |w| dx
+ C\int_U K(|\nabla w|)|\nabla w|^{2s+1}  \zeta  |\nabla \zeta| |w| dx.
\end{align*}
This last inequality and the Young's inequality imply that 
\beqs
I \le  \frac{1}{2} I
  +C  \Big \{ \int_U K(|\nabla w|)|\nabla w|^{2s-2} |\nabla^2 w|^2 \zeta^2 |w|^2 dx
+  \int_U K(|\nabla w|)|\nabla w|^{2s}  |\nabla \zeta|^2 |w|^2dx \Big \}.
\eeqs
Therefore, we obtain
\begin{equation*} 
 I  \le  C \sup_{{\rm supp} \zeta} |w|^2   \Big \{ \int_U K(|\nabla w|)|\nabla w|^{2s-2} |\nabla^2 w|^2 \zeta^2 dx
 + \int_U K(|\nabla w|)|\nabla w|^{2s}  |\nabla \zeta|^2 dx \Big\}.
\end{equation*}
This completes the proof.
\end{proof}

We combine Lemmas \ref{s.it-L} and \ref{LUK} now. By applying Lemma \ref{LUK} to $\bar p$ with $s+1$ in place of $s$, we have 
\begin{equation*} 
\begin{aligned}
\int_0^T \int_U K(|\nabla p|) |\nabla p|^{2s+4} \zeta^2 dx dt 
\le  C \sup_{ \rm{supp}\zeta } |\bar p|^2 & \left \{ \int_0^T\int_U K(|\nabla p|)|\nabla p|^{2s} |\nabla^2 p|^2 \zeta^2 dx dt \right.\\
& + \left. \int_0^T \int_U K(|\nabla p|)|\nabla p|^{2s+2}  |\nabla \zeta|^2 dx dt \right \}.
\end{aligned} 
\end{equation*}
This last inequality and \eqref{preLs2} imply for $s\ge 0$ that
\beq\label{Kgrad1}
\begin{aligned}
\int_0^T \int_U K(|\nabla p|) |\nabla p|^{2s+4} \zeta^2 dx dt
& \le  C \sup_{{\rm supp}\zeta} |\bar p|^2  \left[  \int_U |\nabla p_0(x)|^{2s+2} \zeta^2 dx \right. \\
&  \quad +  \left. \int_0^T \int_U K(|\nabla p|) |\nabla p|^{2s+2}|\nabla \zeta|^2 dx dt \right].
\end{aligned}
\eeq 
Also, from \eqref{preLs2}, we have for $s\ge 2$ that
\begin{equation}\label{preLs3}
\begin{aligned}
 & \sup_{[0,T]} \int_U |\nabla p(x, t)|^{s} \zeta^2 dx
\le  \int_U |\nabla p_0(x)|^{s} \zeta^2 dx
+  C \int_0^T \int_U K(|\nabla p|)  |\nabla p|^{s}|\nabla \zeta|^2 dx dt.
\end{aligned} 
\end{equation}

We are ready to iterate the relation \eqref{Kgrad1}. Hereafterward, we consider $U'\Subset U$, that is, $U'$ is an open set compactly contained in $U$.

\begin{proposition}\label{interp}
For $U'\Subset \setV \subset U$ and $s\ge 2$ there exists a constant $C>0$ depending on $U'$, $V$ and $s$ such that for any $T>0$, we have
\begin{equation}\label{Kgrad50}
\begin{split}
&\int_0^T \int_{U'} K(|\nabla p|) |\nabla p|^{s} dx dt 
\le C  \Big(1+\sup_{ [0,T]} \|\bar p\|_{L^\infty(\setV)}\Big)^{s-2}  \\
&\quad \cdot \Big\{ \int_{U} |\nabla p_0(x)|^{2} dx +\int_{U} |\nabla p_0(x)|^{s-2} dx + \int_0^T \int_{U} K(|\nabla p|) |\nabla p|^2 dx dt\Big\}.
\end{split}
\end{equation}
\end{proposition}
\begin{proof} Let $m \in \mathbb{N}, m \geq 2$ be a fixed number.  Then, let $\{U_k\}_{k=0}^m$ be a family of open, smooth domains in $U$ such that  
 $U'=U_m\Subset U_{m-1}\Subset U_{m-2}\Subset \ldots \Subset U_1 \Subset U_0=\setV\subset U$. For each $k =1,2,\cdots, m$, let $\zeta_k(x)$ be a smooth cut-off function 
 which is equal to one on $U_k$ and zero on $U\setminus U_{k-1}$. There is a positive constant $C>0$ depending on $\zeta_k$, $k=1,2,\ldots,m$, such that 
$ |\nabla \zeta_k| \leq C$, for all $k = 1, 2,\cdots, m.$
 Also, for each integer $k\ge 1$ and $s_0 \geq 0$,  we define 
$$X_k=\int_0^T \int_{U_k} K(|\nabla p|) |\nabla p|^{2k+s_0}  dxdt, \quad A_k=\int_{U_k} |\nabla p_0(x)|^{2k+s_0} dx.$$
Then, by applying \eqref{Kgrad1} with $\zeta =\zeta_k$ and $2s+2=2k+s_0$, we see that 
\beq\label{Xk1}
X_{k+1}\le C_{k,m,s_0} N_0 (A_k + X_k), \quad k = 1,2,\cdots m-1,
\eeq
where $N_0=\sup_{\setV\times[0,T]} |\bar p|^2=\sup_{[0,T]}\|\bar p\|_{L^\infty(V)}^2$. 
Letting $C_m=\max\{C_{k,m,s_0}:k=1,2,\ldots , m-1\}$, we have from \eqref{Xk1} that
\beqs
X_{k+1}\le C_m N_0 (A_k + X_k), \quad k = 1,2,\cdots m-1,
\eeqs
This inequality particularly yields  
\beqs
\begin{split}
X_m&\le (C_m N_0) A_{m-1}+(C_m N_0)^2 A_{m-2}+\cdots+(C_m N_0)^{m-1} A_1 + (C_m N_0)^{m-1}  X_1\\
&\le C[( N_0 A_{m-1}+N_0^2 A_{m-2}+\cdots+N_0^{m-1} A_1) + N_0^{m-1} X_1], 
\end{split}
\eeqs
with $C = C(m)$ depending on all sets $U_k$ and functions $\zeta_k$, for $k=1,2,\ldots,m-1$.
In other words, we have proved that for each integer $m \geq 2$ and real number $s\ge 0$, there is $C = C(m,s_0)$ such that 
\begin{equation*}\label{grad1}
\begin{split}
\int_0^T \int_{U'} K(|\nabla p|) |\nabla p|^{2m+s_0} dx dt & \le C \sum_{i=1}^{m-1} N_0^i\int_{U_{m-i}} |\nabla p_0(x)|^{2(m-i)+s_0} dx \\
& \quad +C N_0^{m-1}\int_0^T \int_{U_1} K(|\nabla p|) |\nabla p|^{2+s_0} dx dt.
\end{split}
\end{equation*}
By using Young's inequality, one can rewrite this inequality as 
\begin{multline}\label{grad15}
\int_0^T \int_{U'} K(|\nabla p|) |\nabla p|^{2m+s_0} dx dt \le C (N_0+N_0^{m-1}) \\
\cdot \Big\{ \int_{U} |\nabla p_0(x)|^{2+s_0} dx +\int_{U} |\nabla p_0(x)|^{2m+s_0-2} dx
+\int_0^T \int_{U_1} K(|\nabla p|) |\nabla p|^{2+s_0} dx dt\Big\}.
\end{multline}
In particular,  with $m=2$ and $s_0=0$, \eqref{grad15} becomes 
\begin{align}\label{Kgrad3}
 \int_0^T \int_{U'} K(|\nabla p|) |\nabla p|^4 dx
&\le C N_0 \Big\{ \int_{U} |\nabla p_0(x)|^2 dx
+   \int_0^T \int_{U_1} K(|\nabla p|) |\nabla p|^2 dx dt\Big\}.
\end{align}
This implies \eqref{Kgrad50} when $s =4$. In case $ s \in (2,4)$, let $\alpha$ and $\beta$ be two positive numbers  such that 
\beq\label{salbe}
\frac 1 s =\frac \alpha{2}+\frac\beta {4}\quad\text{and} \quad \alpha+\beta=1. 
\eeq
Then, using interpolation inequality, we get 
\begin{align*}
\Big(\int_0^T\int_{U'} K(|\nabla p|)|\nabla p|^s dx dt\Big)^{\frac 1 s} &\le \Big(\int_0^T\int_{U_1} K(|\nabla p|) |\nabla p|^{2} dx dt\Big)^{\frac \alpha {2}} \Big(\int_0^T\int_{U'} K(|\nabla p|) |\nabla p|^{4} dx dt\Big)^{\frac \beta {4}}
\end{align*}
From this and \eqref{Kgrad3}, it follows that  
\begin{align*}
\int_0^T\int_{U'} K(|\nabla p|)|\nabla p|^s dx dt
&\le C N_0^{\frac {\beta s} {4}}  \Big\{ \int_{U} |\nabla p_0(x)|^2 dx
+   \int_0^T \int_{U_1} K(|\nabla p|) |\nabla p|^2 dx\Big\}.
\end{align*}
From \eqref{salbe} follows $\beta s/4=s/2-1$. Thus, we have 
\beqs\label{Kgrad4}
\int_0^T\int_{U'} K(|\nabla p|)|\nabla p|^s dx dt
\le C N_0^{\frac {s} {2}-1}  \Big\{ \int_{U} |\nabla p_0(x)|^2 dx
+   \int_0^T \int_{U} K(|\nabla p|) |\nabla p|^2 dx\Big\}.
\eeqs
This implies \eqref{Kgrad50} for $s\in(2,4)$. Therefore, we have proved \eqref{Kgrad50} with $s \in (2, 4]$. Now, for $s>4$, there are a number $s_0\in (0, 2]$ and an integer $m\ge 2$ such that $s=s_0+2m$. From \eqref{grad15}, we have 
\begin{align*}
&\int_0^T \int_{U'} K(|\nabla p|) |\nabla p|^{s} dx dt 
\le C (N_0+N_0^{m-1}) \Big\{ \int_{U} |\nabla p_0(x)|^{2+s_0} dx +\int_{U} |\nabla p_0(x)|^{2m+s_0-2} dx\\
&\quad  +\int_0^T \int_{U_1} K(|\nabla p|) |\nabla p|^{2+s_0} dx dt\Big\}\\
&\le C (N_0+N_0^\frac{s-s_0-2}{2}) \Big\{ \int_{U} |\nabla p_0(x)|^{2+s_0} dx +\int_{U} |\nabla p_0(x)|^{s-2} dx
+\int_0^T \int_{U_1} K(|\nabla p|) |\nabla p|^{2+s_0} dx dt\Big\}.
\end{align*}
Since $s_0+2\in(2,4]$, it follows from the last inequality and estimate \eqref{Kgrad50} already proved for the case $s \in (2,4]$ with $U_1$ in place of $U'$ and $s_0+2$ in place of $s$ that 
\begin{align*}
&\int_0^T \int_{U'} K(|\nabla p|) |\nabla p|^{s} dx dt 
\le C (N_0+N_0^\frac{s-s_0-2}{2}) \Big\{ \int_{U} |\nabla p_0(x)|^{2+s_0} dx +\int_{U} |\nabla p_0(x)|^{s-2} dx\\
&\quad + N_0^{\frac {s_0} {2}}  \Big( \int_{U} |\nabla p_0(x)|^2 dx
+   \int_0^T \int_{U} K(|\nabla p|) |\nabla p|^2 dx dt \Big)\Big\}\\
&\le C (1+N_0^\frac{s-2}{2}) \Big\{ \int_{U} |\nabla p_0(x)|^{2} dx +\int_{U} |\nabla p_0(x)|^{s-2} dx+ \int_0^T \int_{U} K(|\nabla p|) |\nabla p|^2 dxdt \Big\}.
\end{align*}
From this, the desired estimate \eqref{Kgrad50} follows and the proof is complete.
\end{proof}

Now, we give specific  estimates for the $L^s$-norm of $\nabla p$ in terms of initial and boundary data.
We denote
\beq\label{Nf1}
 \Nf_{1,T} = 1+\sup_{[0,T]}\|\psi\|_{L^\infty}.
\eeq
\begin{theorem}  \label{thm312}
\asdc. Then for $s\ge 2$, there is a positive constant $C$ depending on $U'$ and $s$  such that for any $T >0$
\begin{equation}\label{grad20}
\int_0^T \int_{U'} K(|\nabla p|) |\nabla p|^{s} dx dt 
\le C L_1(s)  (T+1)^{\frac{2(s-2)}{2-a}+1} \Nf_{1,T}^{\frac{s-a}{1-a}},
\end{equation}
\begin{equation}\label{grad20pw}
\sup_{[0,T]}\int_{U'} |\nabla p(x,t)|^{s} dx
\le C L_2(s)  (T+1)^{\frac{2(s-2)}{2-a}+1} \Nf_{1,T}^{\frac{s-a}{1-a}},
\end{equation}
where
\begin{align*}
L_1(s)&=L_1(s;[p_0])\eqdef (1+ \|\bar p_0\|_{L^\infty})^{s-2}  \Big ( 1+\int_U |\nabla p_0(x)|^{\max\{2,s-2\}} dx \Big),\\
L_2(s)&=L_2(s;[p_0])\eqdef (1+ \|\bar p_0\|_{L^\infty})^{s-2}  \Big ( 1+\int_U |\nabla p_0(x)|^{s} dx \Big).
\end{align*}
\end{theorem}
\begin{proof} From the inequality \eqref{pbaronly} in Proposition \ref{p-T.est}, it follows that 
\beqs\label{Lamdbound}
\sup_{[0,T]} \|\bar p\|_{L^\infty} \le C \Big\{ \|\bar p_0\|_{L^\infty} + (T+1)^\frac2{2-a}\Nf_{1,T}^\frac1{1-a} \Big\}.
\eeqs
Thanks to \eqref{p-bar-bound1} of Theorem \ref{p-estimate}, we have 
\beq\label{JH2-a}
\int_0^T \int_U K(|\nabla p|) |\nabla p|^2 dxdt\le \norm{\bar p_0}_{L^2}^2 +C\int_0^T f(\tau)d\tau.
\eeq
These inequalities together with \eqref{Kgrad50} of Proposition \ref{interp} yield
\begin{multline*}
\int_0^T \int_{U'} K(|\nabla p|) |\nabla p|^{s} dx dt 
\le C \Big ( \|\bar p_0\|_{L^\infty}  + (T+1)^\frac2{2-a}\Nf_{1,T}^\frac1{1-a}  \Big )^{s-2} \\
\cdot\Big (\|\bar p_0\|_{L^2}^2+ \int_{U} \big(|\nabla p_0(x)|^{2}+|\nabla p_0(x)|^{s-2}\big) dx+\int_0^T f(\tau)d\tau\Big).
\end{multline*}
Applying Poincar\'e's inequality to $\bar p_0$, and using Young's inequality, we obtain 
\begin{multline}\label{gpr1}
\int_0^T \int_{U'} K(|\nabla p|) |\nabla p|^{s} dx dt 
\le C \Big ( \|\bar p_0\|_{L^\infty}  + (T+1)^\frac2{2-a}\Nf_{1,T}^\frac1{1-a}  \Big )^{s-2} \\
\cdot\Big ( 1+\int_U |\nabla p_0(x)|^{\max\{2,s-2\}} dx+\int_0^T f(\tau)d\tau\Big),
\end{multline}
Then \eqref{grad20} follows.
It follows from \eqref{preLs3} and \eqref{grad20} that
\begin{multline*}
\sup_{[0,T]}\int_{U'} |\nabla p(x,t)|^{s} dx
\le \int_{U'} |\nabla p_0(x)|^{s} dx + C \Big ( \|\bar p_0\|_{L^\infty}  + (T+1)^\frac2{2-a}\Nf_{1,T}^\frac1{1-a}  \Big )^{s-2} \\
\cdot\Big (1+\int_{U} \big(|\nabla p_0(x)|^{\max\{2,s-2\}} dx+\int_0^T f(\tau)d\tau\Big),
\end{multline*}
Then by Young's inequality we obtain 
\begin{multline}\label{gpr2}
\sup_{[0,T]}\int_{U'} |\nabla p(x,t)|^{s} dx
\le C \Big ( \|\bar p_0\|_{L^\infty}  + (T+1)^\frac2{2-a}\Nf_{1,T}^\frac1{1-a}  \Big )^{s-2} \\
\cdot\Big (1+  \int_{U} |\nabla p_0(x)|^{s} dx+\int_0^T f(\tau)d\tau\Big).
\end{multline}
Then \eqref{grad20pw} follows.
\end{proof}

Recall that $f$ is defined in \eqref{b} and $M_f$ is a continuous increasing majorant of $f$. Our next result is similar to Theorem \ref{thm312} but the estimates do not contain the power growth in $T$. 

\begin{theorem}\label{cor311}  
\asdc. For $s\ge 2$, there is a positive constant $C$ depending on $U'$ and $s$  such that for any $T >1$ we have
\beq\label{Kgrad55}
\int_0^T \int_{U'} K(|\nabla p|) |\nabla p|^{s} dx dt 
\le C L_3(s)  \Big(1+M_f(T) \Big)^{\muex_4(s-2)}\Big\{1+ \int_0^T f(t)dt\Big\},
\eeq
\beq\label{Kgrad55pw}
\sup_{[0,T]}\int_{U'} |\nabla p(x,t)|^{s} dx
\le C  L_4(s)\Big(1+M_f(T) \Big)^{\muex_4(s-2)}\Big\{1+ \int_0^T f(t)dt\Big\},
\eeq
where
\begin{align*}
L_3(s)&= L_3(s;[p_0])\eqdef  \Big( 1+ \|\bar p_0\|_{L^\infty}+\|\bar p_0\|_{L^2}^{\muex_4(2-a)}\Big)^{s-2} \Big\{ 1+\int_{U} |\nabla p_0(x)|^{\max\{2,s-2\}} dx\Big\},\\
L_4(s)&=L_4(s;[p_0])\eqdef \Big( 1+ \|\bar p_0\|_{L^\infty}+\|\bar p_0\|_{L^2}^{\muex_4(2-a)}\Big)^{s-2} \Big\{ 1+\int_{U} |\nabla p_0(x)|^{s} dx\Big\}. 
\end{align*}

\end{theorem}
\begin{proof}
The proof is the same as in Theorem \ref{thm312} with the use of estimate \eqref{pmax} in place of \eqref{pbaronly}.
Instead of \eqref{grad20} and \eqref{grad20pw}, we have, respectively,
\beq\label{gpr3}
\begin{aligned}
\int_0^T \int_{U'} K(|\nabla p|) |\nabla p|^{s} dx dt 
&\le C  \Big( 1+ \|\bar p_0\|_{L^\infty}+\|\bar p_0\|_{L^2}^{\muex_4(2-a)} +M_f(T)^{\muex_4} \Big)^{s-2}\\
&\quad \cdot \Big\{ 1+\int_{U} |\nabla p_0(x)|^{\max\{2,s-2\}} dx+ \int_0^T f(t)dt\Big\},
\end{aligned}
\eeq
\beq\label{gpr4}
\begin{aligned}
\sup_{[0,T]}\int_{U'} |\nabla p(x,t)|^{s} dx
&\le  C  \Big( 1+ \|\bar p_0\|_{L^\infty}+\|\bar p_0\|_{L^2}^{\muex_4(2-a)} +M_f(T)^{\muex_4} \Big)^{s-2}\\
&\quad \cdot \Big\{ 1+\int_{U} |\nabla p_0(x)|^{s} dx+ \int_0^T f(t)dt\Big\}.
\end{aligned}
\eeq
Then \eqref{Kgrad55} and \eqref{Kgrad55pw} follow, respectively.
\end{proof}

In \eqref{Kgrad55} and \eqref{Kgrad55pw}, letting $T\to\infty$, we obtain the following.

\begin{corollary}\label{corMM}
Under the Strict Degree Condition, if 
\beq\label{xgood}
\Upsilon_1 \eqdef 1+\sup_{[0,\infty)}f <\infty \quad \text{and}\quad \Upsilon_2 \eqdef 1+\int_0^\infty f(t)dt<\infty,
\eeq 
then for any $s\ge 2$, there is a constant $C=C(U',s)>0$ such that 
\beq \label{Kgrad60}
\int_0^\infty \int_{U'} K(|\nabla p|) |\nabla p|^{s} dx dt
\le C  L_3(s) \Upsilon_1^{\muex_4(s-2)} \Upsilon_2,
\eeq
\beq \label{Kgrad60pw}
\sup_{[0,\infty)}\int_{U'} |\nabla p(x,t)|^{s} dx
\le  C  L_4(s) \Upsilon_1^{\muex_4(s-2)} \Upsilon_2.
\eeq
\end{corollary}

Using the preceding theorems and the properties of function $K(\cdot)$, we can derive the following direct estimates for $\int_0^T\int_{U'}|\nabla p|^s dxdt$. 

\begin{theorem}  \label{cor312} 
\asdc. For $s\ge 2-a$, there is 
a constant $C>0$ such that for any $T>0$ we have
\beq\label{grad39}
\int_0^T \int_{U'} |\nabla p|^{s} dx dt \le C \le C L_1(s+a)  (T+1)^{\frac{2s}{2-a}-1} \Nf_{1,T}^{\frac{s}{1-a}},
\eeq
and, alternatively, 
\beq\label{Kgrad70}
\int_0^T \int_{U'} |\nabla p|^{s} dx dt 
\le CT + C L_3(s+a)  \Big(1+M_f(T) \Big)^{\muex_4(s+a-2)}\Big\{1+ \int_0^T f(t)dt\Big\},
\eeq
where the positive numbers $L_1(\cdot)$ and $L_3(\cdot)$ are defined in Theorems \ref{thm312} and \ref{cor311}.
\end{theorem}
\begin{proof} 
By relation \eqref{Km},
\beq\label{KgradsRel}
\int_0^T \int_{U'} |\nabla p|^{s} dx dt \le C T + \int_0^T\int_{U'} K(|\nabla p|)|\nabla p|^{s+a} dx dt.
\eeq
The last integral is estimated by applying \eqref{grad20} with $s+a$ replacing $s$.
As a result, we obtain  \eqref{grad39}.
Now, using relation \eqref{KgradsRel} and applying \eqref{Kgrad55} for $s+a$ in place of $s$, we obtain \eqref{Kgrad70}. 
\end{proof}

\subsection{Estimates for pressure's time derivative}
\label{Lpt-sec}
In this subsection, we derive the interior $L^\infty$-norm of $\bar p_t$. 
Throughout this subsection $U'\Subset U$.
Under the Strict Degree Condition, let 
\beq \label{sdef}
{\muex_5}=4\Big(1-\frac 1 {(2-a)^*}\Big),\quad
\muex_6=1-\frac2 {\muex_5},\quad  
\muex_7=\frac 2 {{\muex_5}(2-a)}.
\eeq
Then ${\muex_5}\in (2,(2-a)^*)$ and $\muex_6,\muex_7\in(0,1).$

\begin{proposition} \label{Lqbar} \asdc. 
There is a constant $C =C(U')>0$ such that for any $T_0\ge 0$, $T>0$ and $\theta\in(0,1)$, we have
\begin{multline}\label{qbarbound}
\sup_{[T_0+\theta T,T_0+T]}\|\bar p_t\|_{L^\infty(U')}\le C\Big\{\big[\lambda (1+(\theta T)^{-1/2} )\big]^\frac{1}{\muex_6} +\big[ \lambda T^{1/2} \sup_{[T_0,T_0+T]}\|\psi_t\|_{L^\infty} \big]^\frac1{\muex_6+1}\Big\}\\
\cdot \Big( \|\bar p_t\|_{L^2(U\times (T_0,T_0+T))} + \|\bar p_t\|_{L^2(U\times (T_0,T_0+T))}^\frac{\muex_6}{\muex_6+1}\Big),
\end{multline}
where
$\lambda=   1 + \sup_{[T_0,T_0+T]}\left(\int_U |\nabla p(x,t)|^{2-a}dx  \right)^{\muex_7},$ 
and
\beq\label{pt:bound}
\sup_{[T_0+\theta T,T_0+T]}\|p_t\|_{L^\infty(U')}\le C\Big(1+\frac 1{\theta T} \Big)^\frac{1}{\muex_6} \Big( 1+\sup_{[T_0,T_0+T]}\norm{\nabla p}_{L^{2-a}(U)}^{\frac 2{{\muex_5}-2}}\Big) \norm{ p_t}_{L^2(U\times(T_0, T_0+T))}.
\eeq 
\end{proposition}

\begin{proof} Without loss of generality, we assume $T_0=0$.
We prove \eqref{qbarbound} first.
Let $ q=p_t$ and
\beqs
\bar q=q-\frac1{|U|}\int_U qdx =p_t -\frac1{|U|} \frac d{dt}\int_U pdx=p_t + \frac1{|U|} \int_\Gamma \psi(x,t) d\sigma=\frac{\partial}{\partial t} \bar p=\bar p_t.
\eeqs
Then it follows from \eqref{eqorig} that $\bar q$ solves
\beq\label{eqt}
\frac{\partial  \bar q}{\partial t}=\nabla \cdot \big(K(|\nabla p|)\nabla p\big)_t  + \frac1{|U|} \int_\Gamma \psi_t d\sigma.
\eeq 

For $k\ge 0$, let $\bar q^{(k)}=\max\{\bar q-k,0\}$
and $\chi_k(x,t)$  be the characteristic function of set $\{(x,t) \in U\times(0,T): \bar q(x,t)>k\}$.
On $S_k(t)$,  we have $(\nabla p)_t=(\nabla\bar p)_t=\nabla\bar q=\nabla\bar q^{(k)}$. 

Let $\zeta=\zeta(x,t)$ be the cut-off function on $U\times [0,T]$ satisfying $\zeta(\cdot,0)=0$ and  $\zeta(\cdot,t)$ having compact support in $U$. 
We will use test function $\bar q^{(k)}\zeta^2$, noting that $\nabla (\bar q^{(k)}\zeta^2) =\zeta [\nabla (\bar q^{(k)}\zeta)+\bar q^{(k)}\nabla \zeta]$.

Multiplying \eqref{eqt} by $\bar q^{(k)}\zeta^2 $ and integrating the resultant on $U$, we get 
\begin{align*}
\frac 12\ddt \int_U |\bar q^{(k)}\zeta|^2 dx & 
= \int_U |\bar q^{(k)}|^2 \zeta \zeta_t  dx - \int_U (K(|\nabla p|))_t\nabla p \cdot[ \nabla (\bar q^{(k)}\zeta)+ \bar q^{(k)}\nabla \zeta]\zeta  dx \\
& \quad 
- \int_U K(|\nabla p|) (\nabla p)_t \cdot [ \nabla (\bar q^{(k)}\zeta)+ \bar q^{(k)}\zeta\nabla]\zeta  dx  +\frac1{|U|}\int_\Gamma \psi_t  d\sigma \int_U \bar q^{(k)}\zeta^2 dx. 
\end{align*}
Put $z=\zeta[\nabla (\bar q^{(k)}\zeta)+ \bar q^{(k)}\nabla \zeta]$.
We simplify the third term on the right-hand side of the last inequality as
\begin{multline*}
(\nabla p)_t \cdot z 
=\zeta \nabla \bar q^{(k)}\cdot [\nabla (\bar q^{(k)}\zeta)+ \bar q^{(k)}\nabla \zeta]\\
=[\nabla (\bar q^{(k)}\zeta)- \bar q^{(k)}\nabla \zeta]\cdot [\nabla (\bar q^{(k)}\zeta)+ \bar q^{(k)}\nabla \zeta] 
= |\nabla (\bar q^{(k)}\zeta)|^2 - |\bar q^{(k)} \nabla \zeta|^2 .
\end{multline*}
For the second term of right-hand side, using \eqref{K-est-2} we have  
\begin{align*}
|(K(|\nabla p|))_t\nabla p \cdot z|\
&=|K'(|\nabla p|)| \frac{|\nabla p\cdot \nabla p_t|}{|\nabla p|} |\nabla p \cdot z| 
\le a K(|\nabla p|) |\nabla \bar q||z|.
\end{align*}
Moreover,
\begin{align*}
|\nabla \bar q| |z|
&=|\zeta \nabla\bar  q^{(k)}| | \nabla (\bar q^{(k)}\zeta)+ q^{(k)}\nabla \zeta |
\le (|\nabla (\bar q^{(k)}\zeta)| + |\bar q^{(k)}||\nabla \zeta|)^2 \\
&= |\nabla (\bar q^{(k)}\zeta)|^2 + 2|\bar q^{(k)}| |\nabla \zeta| |\nabla (\bar q^{(k)}\zeta)| + |\bar q^{(k)} \nabla \zeta|^2.
\end{align*}
It follows that
\begin{align*}
& \frac 12\ddt \int_U |\bar q^{(k)} \zeta|^2 dx + 
(1-a) \int_U K(|\nabla p|) |\nabla (\bar q^{(k)} \zeta)|^2 dx \\
& \le \int_U |\bar q^{(k)}|^2 \zeta  |\zeta_t| dx + (1+a) \int_U K(|\nabla p|)|\bar q^{(k)} \nabla \zeta|^2 dx\\
& \quad+2a\int_U   K(|\nabla p|)| \bar q^{(k)} \nabla \zeta|  |\nabla(\bar q^{(k)}\zeta )|dx 
+C\|\psi_t(t)\|_{L^\infty} \int_U |\bar q^{(k)}|\zeta^2 dx.
\end{align*}
Let $\varepsilon>0$. By Cauchy's inequality,
\begin{align*}
2aK(|\nabla p|)| \bar q^{(k)} \nabla \zeta|  |\nabla(\bar q^{(k)}\zeta )|
&\le \frac {1-a}{2} K(|\nabla p|)|\nabla(\bar q^{(k)}\zeta )|^2 + \frac{2a^2}{1-a}  K(|\nabla p|)|\bar  q^{(k)} \nabla \zeta|^2,\\
C \|\psi_t(t)\|_{L^\infty} |\bar q^{(k)}|\zeta^2 
&\le \varepsilon |\bar q^{(k)}\zeta|^2 + C\varepsilon^{-1}\|\psi_t(t)\|_{L^\infty}^2 \zeta^2.
\end{align*}
Therefore,
\begin{align*}
& \frac 12\ddt \int_U |\bar q^{(k)} \zeta|^2 dx + \frac{1-a}2 \int_U K(|\nabla p|) |\nabla (\bar q^{(k)} \zeta)|^2 dx 
  \le \int_U |\bar q^{(k)}|^2 \zeta  |\zeta_t| dx\\
&\quad  + C \int_U K(|\nabla p|)|\bar q^{(k)} \nabla \zeta|^2 dx + \varep \int_U |\bar q^{(k)}\zeta|^2 dx + C\varep^{-1}\|\psi_t(t)\|_{L^\infty}^2 \int_U \chi_k \zeta^2 dx.
\end{align*}
Now, integrating this inequality in time from $0$ to $t$ and taking supremum on $[0,T]$,  we obtain 
\begin{multline*}
 \sup_{[0,T]} \int_U |\bar q^{(k)}\zeta |^2 dx + \int_0^T \int_U K(|\nabla p|) |\nabla (\bar q^{(k)}\zeta) |^2 dx dt\\
 \le C \Big [\int_0^T \int_U |\bar q^{(k)}|^2 \zeta |\zeta_t| dx dt 
 + \int_0^T\int_U K(|\nabla p|) |\bar q^{(k)} \nabla \zeta|^2 dxdt   \Big]\\
 + C_1\varep T \sup_{[0,T]}\int_U |\bar q^{(k)}\zeta|^2 dx + C\varep^{-1}\sup_{[0,T]}\|\psi_t(t)\|_{L^\infty}^2 \int_0^T \int_U \chi_k \zeta^2 dxdt.
 \end{multline*}
  Using the fact that the functions $K$ and $\zeta$ are bounded and taking $\varep=1/(2C_1 T)$, we obtain  
\begin{equation} \label{m-Lp-t}
\begin{split}
& \sup_{[0,T]} \int_U |\bar q^{(k)} \zeta|^2 dx 
+  \int_0^T \int_U K(|\nabla p|)|\nabla (\bar q^{(k)} \zeta)|^2 dx dt\\
&\quad \le C\left [ \int_0^T  \int_U |\bar q^{(k)}|^2 (\zeta |\zeta_t|  +|\nabla \zeta|^2) dxdt \right]
+ C T \sup_{[0,T]}\|\psi_t(t)\|_{L^\infty}^2 \int_0^T \int_U \chi_k \zeta^2 dxdt.
\end{split}
 \end{equation}
 Applying inequality \eqref{Wemb}  in Lemma~\ref{ParaSob-3} to $\bar q^{(k)} \zeta$ with the weight $W(x,t)=K(|\nabla p(x,t)|)$, we have 
  \begin{align*}
&   \norm{\bar q^{(k)}\zeta}_{L^{\muex_5}(Q_T)}
 \le C   \left \{  \sup_{[0,T]}\Big(\int_U K(|\nabla p|)^{-\frac {2-a} a} dx \Big)^{\frac 2 {s(2-a)} } \right \}\\
  &\quad  \cdot \left\{  \sup_{[0,T]} \Big(\int_U |\bar q^{(k)}\zeta|^2 dx \Big)^{1/2}+\Big(\int_0^T \int_U K(|\nabla p|)|\nabla (\bar q^{(k)}\zeta)|^2 dx dt\Big)^{1/2} \right\}\\
& \le C \sup_{[0,T]}\Big(\int_U (1+|\nabla p|)^{2-a} dx  \Big)^{\muex_7} 
  \cdot \left\{\sup_{[0,T]} \int_U |\bar q^{(k)}\zeta |^2 dx + \int_0^T \int_U K(|\nabla p|)|\nabla (\bar q^{(k)}\zeta)|^2 dx dt \right\}^{1/2}.
 \end{align*} 
Therefore, by \eqref{m-Lp-t} 
  \beq\label{q:bound}
  \norm{\bar q^{(k)}\zeta}_{L^{\muex_5}(Q_T)} \le C\lambda \Big(\int_0^T  \int_U  |\bar q^{(k)}|^2 (|\zeta_t|\zeta + |\nabla \zeta|^2)dx dt + T \sup_{[0,T]}\|\psi_t(t)\|_{L^\infty}^2 \int_0^T \int_U \chi_k \zeta^2 dxdt\Big)^{1/2}.
   \eeq
    
Let $x_0$ be any given point in $U$. Denote  $\rho={\rm dist}(x_0,\partial U)>0$. 
Let $M_0 >0$ be fixed which will be determined later.
For $i\ge 0$, define 
\[
k_i= M_0(1-2^{-i}),\quad
t_i =\theta T( 1- 2^{-i}),\quad 
\rho_i= \frac 1 4\rho(1+ 2^{-i}).
\] 
Then 
$t_0=0<t_1<\ldots<\theta T$ and $\rho_0=\rho/2>\rho_1>\ldots >\rho/4 >0.$
Note that 
\beqs 
\lim_{i\to\infty}t_i=\theta T\quad\text{and}\quad 
\lim_{i\to\infty}\rho_i=\rho/4.
\eeqs
Let $U_i =\{x : \norm{x-x_0}< \rho_i\} $ then $ U_{i+1} \Subset U_i$ for $i=0,1,2,\ldots$. 
For $i,j\ge  0$, we denote 
\beq\label{defQnAnm} 
\begin{aligned}
\mathcal Q_i &=\{(x,t):  x\in U_i,\ t\in(t_i,T)    \}, \quad 
  A_{i,j} &=\{(x,t)\in\mathcal Q_i:  \bar q(x,t)>k_i,\ t\in(t_j, T)    \}.
\end{aligned}
\eeq
For each $\mathcal Q_i$, we use a cut-off function $\zeta_i(x,t)$ which is piecewise linear in $t$ and satisfies $\zeta_i\equiv 1$ on $\mathcal Q_{i+1}$ and $\zeta_i\equiv 0$ on $Q_T\setminus\mathcal Q_{i}$.
Then there is $C>0$ such that  
\beq\label{der:cutoff}
|(\zeta_i)_t|\le \frac C{t_{i+1}-t_i} = \frac {C2^{i+1}}{\theta T} \quad \text { and } \quad
|\nabla \zeta_i | \le \frac {C}{\rho_{i}-\rho_{i+1}} = \frac {C 2^{i+1}}{4\rho}\quad\text{for all } i\ge 0.\eeq
Define 
$F_{i} = \norm{\bar q^{(k_{i+1})}\zeta_i }_{L^{\muex_5}(A_{i+1, i})}$. 
Applying \eqref{q:bound} with  $k=k_{i+1}$  and $\zeta=\zeta_i$ gives
\beq\label{qre2}
F_i \le C\lambda  \Big\{ \int_0^T \int_U |\bar q^{(k_{i+1})}|^2 \Big(\zeta_i |(\zeta_i)_t| +|\nabla \zeta_i|^2\Big) dxdt +T \sup_{[0,T]}\|\psi_t(t)\|_{L^\infty}^2  |A_{i+1,i}|\Big\}^{ 1/2}.
\eeq
Using derivative estimates \eqref{der:cutoff} for $\zeta$,  we obtain  
\beq \label{Fnbound}
\begin{aligned}
F_{i}
&\le C\lambda  \Big( \big [2^i(\theta T)^{-1/2}+ 2^i\big]  \|\bar q^{(k_{i+1})}\|_{L^2(A_{i+1,i})} + T^{1/2} \sup_{[0,T]}\|\psi_t(t)\|_{L^\infty} |A_{i+1,i}|^{1/2}\Big)\\
&\le  C\lambda  2^i(1+\frac 1 {\theta T})^{1/2} \|\bar q^{(k_i)} \|_{L^2(A_i)} + C\lambda T^{1/2} \sup_{[0,T]}\|\psi_t(t)\|_{L^\infty} |A_{i+1,i}|^{1/2}.
\end{aligned}
\eeq
Then, it follows from H\"{o}lder's inequality, \eqref{qre2} 
and \eqref{Fnbound} that
\beq\label{qki+1}
\begin{aligned}
 \|\bar q^{(k_{i+1})} \|_{L^2(A_{i+1, i+1})}
& \le \|\bar q^{(k_{i+1})}\|_{L^{\muex_5}(A_{i+1,i+1})} |A_{i+1,i+1}|^{1/2-1/{\muex_5}}\\
& \le \|\bar q^{(k_{i+1})}\zeta_i \|_{L^{\muex_5}(A_{i+1,i+1})} |A_{i+1,i}|^{1/2-1/{\muex_5}}
 \le C  F_{i}  |A_{i+1,i}|^{1/2-1/{\muex_5}}.
\end{aligned}
\eeq
Note that 
$ \|\bar q^{(k_i)}\|_{L^2(A_i)}\ge  \|\bar q^{(k_{i})}\|_{L^2(A_{i+1,i})}\ge (k_{i+1}-k_i) |A_{i+1,i}|^{1/2}$. 
Thus, 
\beq\label{Aii}
  |A_{i+1,i}| \le (k_{i+1}-k_i)^{-2} \|\bar q^{(k_i)}\|_{L^2(A_i)}^{2} \le C 4^{i} M_0^{-2}\|\bar q^{(k_i)}\|_{L^2(A_i)}^{2}.
\eeq
Then it follows \eqref{qki+1}, \eqref{Fnbound} and \eqref{Aii} that
\begin{align*}
\|\bar q^{(k_{i+1})}\|_{L^2(A_{i+1,i+1})}&\le C\lambda 
\Big\{2^i (1+\frac 1{\theta T} )^{1/2}\|\bar q^{(k_i)}\|_{L^2(A_i)}+ T^{1/2} \sup_{[0,T]}\|\psi_t(t)\|_{L^\infty}2^{i}M_0^{-1} \| q^{(k_i)}  \|_{L^2(A_i)}\Big\}\\
&\quad \cdot 2^{i-\frac {2i} {\muex_5} }M_0^{-1+2/{\muex_5}} \| \bar q^{(k_i)}  \|_{L^2(A_i)}^{1-2/{\muex_5}}\\
&\le C 4^i \lambda \Big\{ (1+\frac 1{\theta T} ) M_0^{-1+2/{\muex_5}} +   T^{1/2} \sup_{[0,T]}\|\psi_t(t)\|_{L^\infty} M_0^{-2+2/{\muex_5}}\Big\}\norm{\bar q^{(k_i)}}_{L^2(A_i)}^{2-2/{\muex_5}}.
\end{align*}
Let $Y_i=\| \bar q^{(k_i)} \|_{L^2(A_i)}$, $B =4$ and
\begin{align*}
D_1&=C\lambda (1+\frac 1{\theta T} )^{1/2} M_0^{-1+2/{\muex_5}}=C\lambda(1+\frac 1{\theta T} )^{1/2}M_0^{-\muex_6},\\
D_2&=C\lambda T^{1/2} \sup_{[0,T]}\|\psi_t(t)\|_{L^\infty} M_0^{-2+2/{\muex_5}}=C\lambda T^{1/2} \sup_{[0,T]}\|\psi_t(t)\|_{L^\infty} M_0^{-1-\muex_6}.
\end{align*}
We obtain 
$Y_{i+1}\le  B^i(D_1 Y_i^{1+\muex_6}+D_1 Y_i^{1+\muex_6})$ for all $i\ge 0$.
We now determine $M_0$ so that 
\beq\label{Y0qcond}
 Y_0 \leq (2D_1)^{-1/\muex_6}B^{-1/{\muex_6^2}}, \quad  Y_0 \leq (2D_2)^{-1/\muex_6}B^{-1/{\muex_6^2}}.
\eeq
This condition is met if
\[ M_0\ge C \big[ \lambda (1+(\theta T)^{-1/2} ) \big]^{1/\muex_6} Y_0,
\quad  M_0\ge C \big[ \lambda T^{1/2} \sup_{[0,T]}\|\psi_t(t)\|_{L^\infty} \big]^\frac1{\muex_6+1} Y_0^\frac{\muex_6}{\muex_6+1}.\]
Since $Y_0\le \| \bar q\|_{L^2(U\times (0,T))}$, it suffices to choose $M_0$ as
\beq\label{Mzero} 
M_0 = C\Big\{\big[\lambda (1+(\theta T)^{-1/2} )\big]^{1/\muex_6} +\big[ \lambda T^{1/2} \sup_{[0,T]}\|\psi_t(t)\|_{L^\infty} \big]^\frac1{\muex_6+1}\Big\}
\Big( \|\bar q\|_{L^2(U\times (0,T))}^\frac{\muex_6}{\muex_6+1}+\|\bar q\|_{L^2(U\times (0,T))}\Big).
\eeq
Then by condition \eqref{Y0qcond}, applying \eqref{Y0mcond} in Lemma \ref{multiseq} for $m=2$ gives 
$\displaystyle{\lim_{i\to\infty}}Y_i=0$. (Alternatively, Lemma \ref{oriseq} can be used in this case.)
Hence, 
\beq\label{intq}
\int_{\theta T}^T\int_{B(x_0,\rho/4)} |\bar  q^{(M_0)}|^2 dxdt=0.
\eeq
Since $\bar q(x,t)\in C(U\times (0,\infty)$, it follows from \eqref{intq} that  $\bar q(x,t)\le M_0$ in $B(x_0,\rho/4)\times (\theta T,T)$.
Replace $q$ by $-q$ and $\psi$ by $-\psi$ and 
use the same argument we obtain $|\bar q(x,t)|\le M_0$ in $B(x_0,\rho/4)\times (\theta T,T)$.
Now by covering $U'$ by finitely many such balls $B(x_0,\rho/4)$, we come to conclusion
\beq\label{ptM0}
|\bar q(x,t)|\le M_0 \quad  \text{in } U'\times (\theta T,T). 
\eeq
By the choice of $M_0$, we obtain \eqref{qbarbound} from \eqref{ptM0}.

In the above proof of \eqref{qbarbound}, we can work with $q$ instead of $\bar q$, with
\beq\label{ebarqt}
\frac{\partial  q}{\partial t}=\nabla \cdot \big(K(|\nabla p|)\nabla p\big)_t,
\eeq 
instead of equation \eqref{eqt}, then the term $\|\psi_t\|_{L^\infty}$ can be removed and we obtain the desired estimate \eqref{pt:bound} for $p_t$.
\end{proof}

\begin{remark}
The main difference between estimate of $\bar p_t$ in \eqref{qbarbound} and the estimate of $p_t$ in \eqref{pt:bound} is the involvement of $\psi_t$.
Though we cannot derive \eqref{pt:bound} from \eqref{qbarbound}, in the following development we will focus on $\bar p_t$ only.
\end{remark}

 The following estimates can be easily derived (by the mean of Young's inequality) from corresponding ones in \cite{HI2}. They will be used in finding $L^\infty$-estimates for $\bar p_t$ in terms of initial data and boundary data.   
We use the following notation:
\begin{align}
&m_1(t)=1+ \norm{\bar p_0}_{L^2}^2  + M_f(t)^\frac2{2-a}+\int_{t-1}^t \tilde f(\tau)d\tau,
&&m_2(t)=1 +  A^\frac2{2-a}+\int_{t-1}^t \tilde f(\tau)d\tau,\\
\label{Atildef}
&m_3(t)=1+\beta^\frac1{1-a}
+\sup_{[t-1,t]} f^\frac{2}{2-a} +\int_{t-1}^t \tilde f(\tau)d\tau,
&&\Ulim_1=A+A^\frac2{2-a}+\limsup_{t\to\infty}\int_{t-1}^t \tilde f(\tau)d\tau.
\end{align}

\begin{theorem}[cf. \cite{HI2} Theorems 4.4  and 4.5]\label{H-bound-lemma}
{\rm (i)} One has for all $t\ge 0$ that
\beq
\label{H-bound-0}
J_H[p](t)+\norm{\bar p_t}_{L^2(U\times(0,t))}^2
\le C\Big(  \norm{\bar p_0}_{L^2}^2 +J_H[p](0)+  (t+1) \sup_{[0,t]} f +\int_{0}^t\tilde f(\tau)d\tau\Big),
\eeq
and for all $t\ge 1$ that
\beq \label{H-bound-6}
J_H[p](t)+\norm{\bar p_t}_{L^2(U\times(t-1/2,t))}^2
\le C\Big( \norm{\bar p(t-1)}_{L^2}^2 + \sup_{[t-1,t]}f+ \int_{t-1}^t  \tilde f(\tau)d\tau\Big).
\eeq

Now, assume that the Degree Condition holds. 

{\rm (ii)} For $0<t_0< 1$ and $t\ge t_0$, one has
\beq\label{Jpt-boundA0} 
 \norm{\bar p_t(t)}_{L^2}^2  \le C\, t_0^{-1} L_5(t_0)+ C \int_0^t f(\tau)+ \tilde f(\tau)d\tau,
\eeq
 where $L_5(t_0)=L_5(t_0;[p_0,\psi])\eqdef 1+\norm{\bar p_0}_{L^2}^2 +\|\nabla p_0\|_{L^{2-a}}^{2-a} + M_f(t_0)^\frac2{2-a}+\int_0^{t_0}\tilde f(\tau)d\tau$.

For all $t\ge 1$, one has 
\beq\label{H-bound-b}
J_H[p](t),\ \norm{\bar p_t(t)}_{L^2}^2 \le C m_1(t) \quad \text{for all }t\ge 1.
\eeq

{\rm (iii)} If $A<\infty$ then there is $T>1$ such that
\beq\label{boundedlarge} 
J_H[p](t),\ \norm{\bar p_t(t)}_{L^2}^2  \le  C m_2(t) \quad \text{for all } t>T, 
\eeq
\beq\label{limsupH1} 
\limsup_{t\to\infty} J_H[p](t),\ \limsup_{t\to\infty}  \norm{\bar p_t(t)}_{L^2}^2 \le  C \Ulim_1.\eeq

{\rm (iv)} If $\beta<\infty$ then  there is $T>1$ such that
\beq\label{unbounJ} 
J_H[p](t),\ \norm{\bar p_t(t)}_{L^2}^2  \le C m_3(t) \quad \text{for all } t>T.\eeq
\end{theorem}  

%
%
%
%

We now state our main estimates of $\bar p_t$. 
In addition to $\Nf_{1,t}$ defined by \eqref{Nf1}, we will also use 
\beq\label{Nf2}
\Nf_{2,t}=1+\int_{0}^t  \tilde f(\tau)d\tau.
\eeq

\begin{theorem}\label{ptbar} \asdc. 
For $t>0$, one has
\beq\label{tpos}
\begin{aligned}
\norm{\bar p_t(t)}_{L^\infty(U')} 
&\le C  \Big\{ \norm{\bar p_0}_{L^2}^2+\|\nabla p_0\|_{L^{2-a}}^{2-a}+(1+t)\Nf_{1,t}^\frac{2-a}{1-a} +\Nf_{2,t}\Big\}^{\frac{\muex_7}{\muex_6}+\frac 1 2} \\
&\quad \cdot\Big\{1+ t^{-\frac 1{2\muex_6}} + t^{\frac 1 {2(\muex_6+1)}}  \sup_{[0,t]}\|\psi_t\|_{L^\infty}^\frac1{\muex_6+1}\Big\}.
\end{aligned}
\eeq
For $t\ge 3/2$, one has 
\beq\label{tlarge}
\begin{aligned}
\norm{\bar p_t(t)}_{L^\infty(U')} 
\le C\Big( 1+\norm{\bar p_0}_{L^2}^2 +M_f(t)^{\frac 2{2-a} }+ \int_{t-3/2}^t \tilde f(\tau)d\tau \Big)^{\frac{\muex_7}{\muex_6} +\frac 1 2} \Big(1 +\sup_{[t-1/2,t]}\|\psi_t\|_{L^\infty}^\frac1{\muex_6+1}\Big).
\end{aligned}
\eeq
Above, the positive constant $C$ depends on $U'$.
\end{theorem}
\begin{proof}
Applying \eqref{qbarbound} with $T_0=0$, $T=t$ and $\theta=1/ 2$, an then applying Young's inequality to the term 
 $\|\bar p_t\|_{L^2(U\times (0,t))}^{\muex_6/(\muex_6+1)}$, we have 
\begin{align*}
\norm{\bar p_t(t)}_{L^\infty(U')} 
&\le C (1+\sup_{[0,t]}\|\nabla p\|_{L^{2-a}}^{2-a})^\frac{\muex_7}{\muex_6} \\
&\quad \cdot  \Big(\Big[1+\sqrt{2/t\,}  \Big]^\frac{1}{\muex_6} +\big[ t^{1/2} \sup_{[0,t]}\|\psi_t\|_{L^\infty} \big]^\frac1{\muex_6+1}\Big) \Big( 1+\|\bar p_t\|_{L^2(U\times (0,t))}\Big).
\end{align*}
Using \eqref{H-bound-0} to estimate $\sup_{[0,t]}\|\nabla p\|_{L^{2-a}}^{2-a}$ and $\norm{\bar p_t}_{L^2(U\times(0,t))}$ in the previous inequality, we find that
\begin{align*}
&\norm{\bar p_t(t)}_{L^\infty(U')} 
\le C \Big\{ 1+\norm{\bar p_0}_{L^2}^2 +\|\nabla p_0\|_{L^{2-a}}^{2-a} +(1+t)\Nf_{1,t}^\frac{2-a}{1-a} +\Nf_{2,t}\Big\}^{ \frac{\muex_7}{\muex_6}} \\
 &\quad \cdot\Big\{1+t^{-\frac 1{2\muex_6}} + t^{\frac 1 {2(\muex_6+1)}} \sup_{[0,t]}\|\psi_t\|_{L^\infty}^\frac1{\muex_6+1}\Big\}
\Big(1+\norm{\bar p_0}_{L^2}^2 +\|\nabla p_0\|_{L^{2-a}}^{2-a} + (1+t)\Nf_{1,t}^\frac{2-a}{1-a} +\Nf_{2,t}\Big)^\frac12 .
\end{align*}
Then inequality \eqref{tpos} follows.

Now, consider $t\ge 3/2$. 
Applying estimate \eqref{qbarbound} to the interval $[t-1/2,t]$, i.e., $T_0=t-1/2$ and $T=1/2$,  with $\theta=1/2$, we obtain
\begin{align*}
\norm{\bar p_t(t)}_{L^\infty(U')}
&\le C\Big( 1+\sup_{[t-1/2,t]}\norm{\nabla p}_{L^{2-a}(U)}^{2-a}\Big)^{\frac{\muex_7}{\muex_6}} \Big( 1 +\sup_{[t-1/2,t]}\|\psi_t(t)\|_{L^\infty} ^\frac1{\muex_6+1}\Big)\\
&\quad\cdot\Big( \|\bar p_t\|_{L^2(U\times (t-1/2,t))}^\frac{\muex_6}{\muex_6+1}+\|\bar p_t\|_{L^2(U\times (t-1/2,t))}\Big).
\end{align*}
Let $m(t)=\sup_{[t-3/2,t]}f+ \int_{t-3/2}^t \tilde f(\tau)d\tau$. 
Utilizing estimate \eqref{H-bound-6}, we have
\begin{multline}\label{t32}
\norm{\bar p_t(t)}_{L^\infty(U')} 
\le C\Big( 1+\sup_{\tau\in[t-3/2,t-1]} \norm{\bar p(\tau)}_{L^2}^2 + m(t)\Big)^{\frac{\muex_7}{\muex_6}}
 \Big(1 +\sup_{[t-1/2,t]}\|\psi_t(t)\|_{L^\infty}^\frac1{\muex_6+1}\Big)\\
\cdot  \Big\{ \Big( \norm{\bar p(t-1)}_{L^2}^2 + m(t) \Big)^\frac{\muex_6}{\muex_6+1}+\Big( \norm{\bar p(t-1)}_{L^2}^2 + m(t) \Big) \Big\}^{\frac12}.
\end{multline}
By Young's inequality, 
\begin{align*}
|\bar p_t(x,t)|&\le C\Big\{ 1+\sup_{\tau\in[t-3/2,t-1]} \norm{\bar p(\tau)}_{L^2}^2 + m(t) \Big\}^{\frac{\muex_7}{\muex_6} +\frac 1 2}
 \Big\{1 +\sup_{[t-1/2,t]}\|\psi_t(t)\|_{L^\infty}^\frac1{\muex_6+1}\Big\}.
\end{align*}
Then inequality \eqref{tlarge} is obtained by using \eqref{p-bar-bound} to estimate $\| \bar p(\tau)\|_{L^2}$ for $\tau\le t-1$.    
\end{proof}

Let 
\beq\label{Ulim2}
\Ulim_2=\limsup_{t\to\infty} m_2(t)=1+A^\frac{2}{2-a}+\limsup_{t\to\infty}\int_{t-1}^t\tilde f(\tau)d\tau,
\eeq
\beq\label{A3def}
\Ulim_3=1+A^{\frac2{2-a}}+ \limsup_{t\to\infty} \|\psi_t(t)\|^\frac{2-a}{1-a}.
\eeq
We have from \eqref{Atildef}, \eqref{Ulim2} and \eqref{A3def} the relation:
\beq\label{Urel}
\Ulim_1\le C\Ulim_2\le C'\Ulim_3.
\eeq


\begin{theorem}\label{ptbar2}
\asdc.

 {\rm (i)} If $A<\infty$ then
\beq\label{limsup-pt} 
\begin{aligned}
\limsup_{t\to\infty} \norm{\bar p_t(t)}_{L^\infty(U')}
\le C\big(\Ulim_1 +\Ulim_1^\frac{\muex_6}{\muex_6+1}\big)^{1/2} 
\Ulim_2^{\frac{\muex_7}{\muex_6}}  \Big(1 +\limsup_{t\to\infty}\|\psi_t(t)\|_{L^\infty}^\frac1{\muex_6+1}\Big).
\end{aligned}
 \eeq
Consequently, denoting $\muex_8=\frac{\muex_7}{\muex_6}+\frac 12+\frac{1-a}{(2-a)(\muex_6+1)}$, one has
\beq\label{limsup-pt10} 
\limsup_{t\to\infty} \norm{\bar p_t(t)}_{L^\infty(U')}\le C \Ulim_3^{\muex_8}.
\eeq

{\rm (ii)} If  $\beta<\infty$, then there is $T>0$ such that for all $t>T$,
\beq\label{ptbeta}
 \norm{\bar p_t(t)}_{L^\infty(U')}
\le C\Big(1+\beta^\frac1{1-a}+\sup_{\tau\in[t-3/2,t]} f(\tau)^\frac2{2-a}+\int_{t-3/2}^t \tilde f(\tau)d\tau\Big)^{\frac {\muex_7}{\muex_6} +\frac 1 2}\Big(1 +\sup_{[t-1/2,t]}\|\psi_t(t)\|_{L^\infty}^\frac1{\muex_6+1}\Big) .
 \eeq

Above, the positive constant $C$ depends on $U'$.
\end{theorem}
\begin{proof}
 (i) We follow the proof in Theorem \ref{ptbar}. 
Observe that
\beq\label{limint}
\limsup_{t\to\infty} \int_{t-3/2}^t \tilde f(\tau)d\tau
\le \limsup_{t\to\infty} \int_{t-2}^{t-1} \tilde f(\tau)d\tau + \limsup_{t\to\infty} \int_{t-1}^t \tilde f(\tau)d\tau
\le 2 \limsup_{t\to\infty} \int_{t-1}^t \tilde f(\tau)d\tau.
\eeq
The limit estimates \eqref{nonzerobeta} and \eqref{limint} yield 
\beqs
\limsup_{t\to\infty} \Big(\sup_{[t-3/2,t-1]} \|\bar p\|_{L^2}+m(t)\Big)\le C \Ulim_1. 
\eeqs
Combining this with \eqref{t32}, we obtain
\beqs
\limsup_{t\to\infty} \norm{\bar p_t(t)}_{L^\infty(U')}\le C\big( 1+\Ulim_1\big)^\frac{\muex_7}{\muex_6}  \Big(1 +\limsup_{t\to\infty}\big[\sup_{[t-1/2,t]}\|\psi_t\|_{L^\infty}^\frac1{\muex_6+1}\big]\Big) \big(\Ulim_1^\frac{\muex_6}{\muex_6+1} +\Ulim_1 \big)^{\frac12}.
 \eeqs
Then \eqref{limsup-pt} follows this and relation \eqref{Urel}. 
Elementary calculations will give \eqref{limsup-pt10} from \eqref{limsup-pt}.

(ii) We use \eqref{t32} and Young's inequality again, this time, with estimate \eqref{unboundp}; it results in \eqref{ptbeta}.
\end{proof}

While \eqref{limsup-pt} gives an asymptotic estimate for $\bar p_t$ in the case $A<\infty$, the next result covers the case $A=\infty$.

\begin{theorem}\label{smallqbar}
\asdc. 
Let $L_6=\norm{\bar p_0}_{L^2}^2 $ in general case, and $L_6=\beta^\frac1{1-a}$ in case $\beta<\infty$, and define
\beq\label{hdef}
N(t)=1+ L_6 +M_f(t)^{\frac 2 {2-a}}+\int_{t-2}^t \tilde f(\tau) d\tau,\quad
h(t)= N(t)^{\frac{\muex_7}{\muex_6}} \Big(1 +\sup_{[t-1,t]}\|\psi_t\|_{L^\infty}^\frac1{\muex_6+1}\Big). 
\eeq
If 
\beq\label{hNcond}
 \lim_{t\to\infty} h(t)^\frac{2(\muex_6+1)}{\muex_6}  e^{-d_1\int_2^t N(\tau)^{-b}d\tau} =0 
\quad \text{and} \quad  
\lim_{t\to \infty} \frac {h'(t)}{h(t)}N^b(t) =0 ,
\eeq
where $b$ is defined in \eqref{ab} and $d_1>0$ appears in \eqref{recallpt} below,
then
\beq\label{limqbar}
\limsup_{t\to\infty} \norm{\bar p_t(t)}_{L^\infty(U')}
\le C\limsup_{t\to\infty}
\Big\{ h(t) \Big(   N(t)^{b}\norm {\psi_t(t)}_{L^\infty}  +\Big[ N(t)^{b} \norm { \psi_t(t)}_{L^\infty} \Big]^\frac{\muex_6}{\muex_6+1}\Big)\Big\},  
\eeq
where $C>0$ depends on $U'$.
\end{theorem}
\begin{proof}
Applying \eqref{qbarbound} to the interval $[t-1,t]$ with $\theta=\frac 1 2$, we have
\begin{multline}\label{ptinterval}
\norm{\bar p_t(t)}_{L^\infty(U')} \le C \Big[ \Big( 1+\sup_{[t-1,t]} J_H[p](\tau)\Big)^\frac{\muex_7}{\muex_6} 
\Big(1+\sup_{[0,T]} \|\psi_t(t)\|_{L^\infty}^\frac1{\muex_6+1}\Big)\Big]\\
\cdot \Big( \sup_{[t-1,t]} \|\bar p_t\|_{L^2(U)}^\frac{\muex_6}{\muex_6+1}+\sup_{[t-1,t]} \|\bar p_t\|_{L^2(U)}\Big).
\end{multline}
In the following $T>2$ is sufficiently large.  By \eqref{H-bound-b} in the general case,  and by \eqref{unbounJ} in case $\beta<\infty$, we have
\beqs
J_H[p](\tau) \le C\Big(1+ L_6 +M_f(t)^{\frac 2 {2-a}}+\int_{t-1}^t \tilde f(\tau) d\tau\Big), \quad \tau>T. 
\eeqs
Therefore,
\beq\label{nablaP1b} 
\sup_{[t-1,t]} J_H[p](\tau) \le C N(t). 
\eeq
Recall inequality (5.13) in \cite{HI2}: there is $d_1>0$ such that for $t>$ we have
    \beq\label{recallpt}
      \frac {d} {dt} \norm{\bar p_t(t)}^2 \le -d_1N(t)^{-b} \norm{\bar p_t(t)}^2 +C \norm { \psi_t(t)}^2_{L^\infty} N(t)^b.
    \eeq
Therefore, for $t'\in (T,\infty)$
\beqs
 \norm{\bar p_t(t')}^2 \le   e^{-d_1\int_T^{t'} N(\tau)^{-b}d\tau} \norm{\bar p_t(T)}^2+ \int_T^{t'}  e^{-d_1 \int_\tau^{t'} N(s)^{-b}ds} \norm { \psi_t(\tau)}^2_{L^\infty} N(\tau)^b d\tau. 
\eeqs  
Then taking supremum in $t'$ over the interval $[t-1,t]$ yields
\begin{align*}
 \sup_{[t-1,t]} \norm{\bar p_t}^2 
\le e^{-d_1\int_T^{t-1} N^{-b}(\tau)d\tau} \norm{\bar p_t(T)}^2 + e^{-d_1\int_T^{t-1} N^{-b}(\tau)d\tau} \int_T^t e^{d_1 \int_T^\tau N^{-b}(\theta)d\theta} \norm{\psi_t(\tau)}_{L^\infty}^2 N(\tau)^{b} d\tau .
\end{align*}
Thanks to the fact $N(t)\ge 1$, we have 
\beqs
e^{-d_1\int_T^{t-1} N(\tau)^{-b}d\tau}= e^{-d_1\int_T^{t} N(\tau)^{-b}d\tau} e^{d_1\int_{t-1}^t N(\tau)^{-b}d\tau} \le e^{d_1} e^{-d_1\int_T^{t} N(\tau)^{-b}d\tau}.
\eeqs  
Hence,
\begin{multline}\label{supptint}
 \sup_{[t-1,t]} \norm{\bar p_t}^2\\
\le C \Big(e^{-d_1\int_T^{t} N^{-b}(\tau)d\tau} \norm{\bar p_t(T)}^2 + e^{-d_1 \int_T^{t} N^{-b}(\tau)d\tau}  \int_T^t e^{d_1 \int_T^\tau N^{-b}(\theta)d\theta} \norm{\psi_t(\tau)}_{L^\infty}^2 N(\tau)^b d\tau\Big)\\
=C\Big(e^{-d_1\int_T^{t} N^{-b}(\tau)d\tau} \norm{\bar p_t(T)}^2 +  \int_T^t e^{-d_1\int_\tau^t N^{-b}(\theta)d\theta} \norm{\psi_t(\tau)}_{L^\infty}^2 N(\tau)^b d\tau \Big).
\end{multline}
By \eqref{ptinterval}, \eqref{nablaP1b} and \eqref{supptint}, we obtain 
\beqs
\begin{aligned}
\norm{\bar p_t(t)}_{L^\infty(U')}
& \le C h(t)
\Big\{  \Big(e^{-d_1\int_T^{t} N(\tau)^{-b}d\tau} \norm{\bar p_t(T)}^2+ \int_T^t  e^{-d_1 \int_\tau^{t} N(s)^{-b} ds} \norm { \psi_t(\tau)}^2_{L^\infty} N(\tau)^b d\tau \Big)^{\frac{\muex_6}{2(\muex_6+1)}}\\
&\quad+\Big( e^{-d_1\int_T^{t} N(\tau)^{-b}d\tau} \norm{\bar p_t(T)}^2+ \int_T^t  e^{-d_1 \int_\tau^{t} N(s)^{-b} ds} \norm { \psi_t(\tau)}^2_{L^\infty} N(\tau)^b d\tau \Big)^\frac12\Big\}. 
\end{aligned}
\eeqs   
Thus,
\begin{align*}
&\norm{\bar p_t(t)}_{L^\infty(U')}
 \le C\Big(  h(t)^{\frac{2(\muex_6+1)}{\muex_6} }e^{-d_1 \int_T^t N(\tau)^{-b}d\tau} \norm{\bar p_t(T)}^2\\
&\quad + h(t)^{\frac{2(\muex_6+1)}{\muex_6} }\int_T^t  e^{-d_1 \int_\tau^t N(s)^{-b} ds} \norm { \psi_t(\tau)}^2_{L^\infty} N(\tau)^b d\tau\Big)^{\frac{\muex_6}{2(\muex_6+1)}} \\
&\quad+C\Big(h(t)^2 e^{- d_1 \int_T^t N(\tau)^{-b}d\tau} \norm{\bar p_t(T)}^2+ h^2(t)\int_T^t  e^{-d_1 \int_\tau^t N(s)^{-b} ds} \norm { \psi_t(\tau)}^2_{L^\infty} N(\tau)^b d\tau\Big)^{\frac 12}   .
\end{align*}
Under condition \eqref{hNcond}, applying Lemma~\ref{difflem2} we obtain
\beqs
\begin{aligned}
\limsup_{t\to\infty} \norm{\bar p_t(t)}_{L^\infty } &\le C\Big\{\limsup_{t\to\infty}\Big[ \norm { \psi_t(\tau)}^2_{L^\infty} h(t)^{\frac{2(\muex_6+1)}{\muex_6} } N(t)^{2b}\Big]\Big\}^{\frac{\muex_6} {2(\muex_6+1)} }\\
&\quad+C\Big\{\limsup_{t\to\infty} \Big[\norm { \psi_t(\tau)}^2_{L^\infty} h(t)^2  N(t)^{2b}\Big]\Big\}^{\frac 1 2}.  
\end{aligned}
\eeqs
Therefore, we obtain \eqref{limqbar}.  
\end{proof}

Note from \eqref{limsup-pt} that $\lim_{t\to\infty} \|\bar p_t(t)\|_{L^\infty(U')}= 0$  provided 
$$\lim_{t\to\infty}\|\psi(t)\|_{L^\infty}= 0\quad \text{and}\quad \lim_{t\to\infty}\|\psi_t(t)\|_{L^\infty}= 0.$$ 
In the following, we can drop the first limit condition.

\begin{corollary}\label{ZeroCor}
Under the Strict Degree Condition, if $\|\psi(t)\|_{L^\infty}$ is uniformly bounded on $[0,\infty)$ and $\|\psi_t(t)\|_{L^\infty}\to 0$ as $t\to\infty$, then
\beq\label{zerolim}
\|\bar p_t(t)\|_{L^\infty(U')}\to 0 \quad \text{as } t\to\infty.
\eeq
\end{corollary}
\begin{proof}
 In this case, $N(t)$ and $h(t)$ are uniformly bounded on $[2,\infty)$ by a constant $C_0$. In the Theorem \ref{smallqbar}, we can replace $N(t)$ and $h(t)$ by this $C_0$. Then conditions in \eqref{hNcond} are met and \eqref{zerolim} follows \eqref{limqbar}. 
\end{proof}

\subsection{Estimates for pressure's second derivatives}
\label{seconderiv}

We estimate the Hessian $\nabla^2 p=\big(\frac{\partial^2 p}{\partial x_i \partial x_j}\big)_{i,j=1,2,\ldots,n}$.
Throughout this subsection $U'\Subset U$.

\begin{proposition}\label{hessprop}
Let $U'\Subset \setV \Subset U$ and $\delta\in(0,a]$.
Then for $t>0$,
\beq\label{delunite}
\int_{U'} |\nabla^2 p(x,t)|^{2-\delta} dx \le C \Big(1+ \int_{U} |\nabla p(x,t)|^{(2-\delta)a/\delta}  dx\Big) \Big(1+  \|\bar p_t(t)\|_{L^\infty( \setV )}^2\Big),
\eeq
where the positive constant $C$ depends on $U'$, $V$ and $\delta$.
\end{proposition}
\begin{proof}
Young's inequality gives   
\begin{multline}\label{del3}
\int_{U'} |\nabla^2 p|^{2-\delta} dx \le \int_{U'} K(|\nabla p|) |\nabla^2 p|^2 dx + \int_{U'} K(|\nabla p|)^{-(2-\delta)/\delta} dx\\
\le \int_{U'} K(|\nabla p|) |\nabla^2 p|^2 dx + C\int_{U'} \big[1+|\nabla p|^{(2-\delta)a/\delta}\big] dx.
\end{multline}
Note that for $s\ge 0$, we have
\beq\label{silly}
\frac{1}{2s+2} \ddt \int_U |\nabla\bar p|^{2s+2} \zeta^2 dx= - \int_U \bar p_t  \nabla \cdot (|\nabla\bar p|^{2s}\nabla\bar p \zeta^2)  dx.
\eeq
Since $\nabla \bar p =\nabla p$, replacing $p$ by $\bar p$ in \eqref{gradid} and using \eqref{silly} with $s=0$, we have for any $\zeta(x)\in C_c^\infty(U)$ that 
\begin{align*}
 \int_U K(|\nabla p|) |\nabla^2 p|^2   \zeta^2 dx\le  C \int_U K(|\nabla p|) |\nabla p|^{2}|\nabla \zeta|^2 dx
+ C   \int_U \bar p_t \nabla \cdot (\nabla p \zeta^2) dx.
\end{align*}
Let $\varepsilon>0$. For the last integral,
\begin{align*}
& \Big|  \int_U \bar p_t \nabla \cdot (\nabla p \zeta^2) dx \Big|
\le C \int_U |\bar p_t| |\nabla ^2  p| \zeta^2 + |\bar p_t||\nabla p| \zeta |\nabla \zeta|  dx \\
&\le \varep \int_U K(|\nabla p|) |\nabla^2 p|^2     \zeta^2 dx
+ C\varep^{-1}\int_U |\bar p_t|^2  K(|\nabla p|)^{-1} \zeta^2 dx
+C \int_U  |\bar p_t||\nabla p| \zeta |\nabla \zeta|  dx.
\end{align*}
Therefore,
\beq\label{gradineq1}
\begin{aligned}
& (1-\varep) \int_U K(|\nabla p|) |\nabla^2 p|^2   \zeta^2 dx
\le   C \int_U K(|\nabla p|) |\nabla p|^{2}|\nabla \zeta|^2 dx\\
&\quad\quad\quad\quad+C\varep^{-1}\int_U |\bar p_t|^2 (1+|\nabla   p|)^a \zeta^2 dx
+C \int_U  |\bar p_t||\nabla p| \zeta |\nabla \zeta|  dx.
\end{aligned}
\eeq
Constructing appropriate $\zeta$ with $\zeta\equiv 1$ on $U'$ and supp $\zeta \subset \setV$, we obtain from \eqref{gradineq1} that
\begin{align*}
&  \int_{U'} K(|\nabla p|) |\nabla^2 p|^2 dx
\le C\int_{ \setV }|\nabla p|^{2-a}dx+    C \|\bar p_t(t)\|_{L^\infty( \setV )}^2 \int_{ \setV } (1+|\nabla   p|)^a  dx
\\
&+ C \|\bar p_t(t)\|_{L^\infty( \setV )} \int_{ \setV } |\nabla   p| dx \le C\int_{ \setV }|\nabla p|^{2-a}dx+  C \Big(1+ \int_{ \setV } |\nabla p|   dx\Big) (1+  \|\bar p_t\|_{L^\infty(\setV)}^2).
\end{align*}
Thus, by Young's inequality, we have
\beq\label{gradinqe2}
  \int_{U'} K(|\nabla p|) |\nabla^2 p|^2 dx
\le C \Big(1+ \int_{ \setV } |\nabla p|^{2-a}   dx\Big) (1+  \|\bar p_t\|_{L^\infty(\setV)}^2).
\eeq
Combining \eqref{del3} with \eqref{gradinqe2}, we obtain
\begin{multline*}
\int_{U'} |\nabla^2 p(x,t)|^{2-\delta} dx \le C \Big(1+ \int_{U} |\nabla p(x,t)|^{2-a}  dx\Big) \Big(1+  \|\bar p_t(t)\|_{L^\infty( \setV )}^2\Big)\\
+ C\int_{U'} |\nabla p(x,t)|^{(2-\delta)a/\delta} dx.
\end{multline*}
Then \eqref{delunite} follows.
\end{proof}

We consider two  cases $\delta=a$ and $\delta<a$ separately.
We define two exponents
\beq 
\muex_9=\frac{2\muex_7}{\muex_6}+2+\frac 1 {\muex_6+1}\quad\text{and}\quad \muex_{10}=\frac{2-a}{1-a}\big(\frac{2\muex_7}{\muex_6}+2\big) + \frac2{\muex_6+1}.
\eeq

\begin{theorem}\label{hessest}
\asdc.

{\rm (i)} If $t>0$ then 
\begin{multline}\label{hess1b}
 \int_{U'} |\nabla^2 p(x,t)|^{2-a} dx 
 \le C L_7  (1+ t^{-\frac 1{\muex_6}}) 
(t+1)^{\muex_9} 
\Nf_{1,t}^{\frac{2-a}{1-a}\big(\frac{2\muex_7}{\muex_6}+2\big)} \Big\{ 1 +  \sup_{[0,t]}\|\psi_t\|_{L^\infty}\Big\}^{\muex_{10}},
\end{multline}
where $L_7=( 1+\norm{\bar p_0}_{L^2}^2 +\|\nabla p_0\|_{L^{2-a}}^{2-a} )^{\frac{2\muex_7}{\muex_6}+ 2}$  and $\Nf_{1,t}$ is defined in \eqref{Nf1}.

If $t\ge 3/2$ then 
\beq\label{h2bcomp}
 \begin{aligned}
 \int_{U'} |\nabla^2 p(x,t)|^{2-a} dx 
 \le  C L_8 \{ 1 +M_f(t)\}^{\frac2{2-a}\big(\frac{2\muex_7}{\muex_6}+ 2\big)}
 \cdot\Big\{1+\sup_{[t-3/2, t]} \norm{\psi_t}_{L^\infty}\Big\}^{\muex_{10}},
\end{aligned}
\eeq
where $L_8=( 1+\norm{\bar p_0}_{L^2}^2 )^{\frac{2\muex_7}{\muex_6}+ 2}$.

{\rm (ii)} If $A<\infty$ then 
\beq\label{limsupHessb}
\limsup_{t\to\infty}\int_{U'} |\nabla^2 p(x,t)|^{2-a} dx
\le  C \Ulim_3^{2\muex_8+1},
\eeq
where $\Ulim_3$ is defined by \eqref{A3def}.

{\rm (iii)} If  $\beta<\infty$, then there is $T>0$ such that for all $t>T$,
\beq\label{Hess3b}
 \begin{split}
 \int_{U'} |\nabla^2 p(x,t)|^{2-a} dx &\le C\Big(1+\beta^\frac1{1-a}+\sup_{[t-3/2,t]} f^\frac2{2-a}\Big)^{\frac{2\muex_7}{\muex_6} + 2}
 \big(1 +\sup_{[t-3/2,t]}\|\psi_t\|_{L^\infty}\big)^{\muex_{10}}.
\end{split}
\eeq

Above, the positive constant $C$ depends on $U'$.
\end{theorem}
\begin{proof}
(i) For $t>0$, we obtain  from \eqref{delunite} with $\delta=a$ that 
\beq\label{Hesspb}
\int_{U'} |\nabla^2 p(x,t)|^{2-a} dx \le C(1+ \norm{\nabla p}^{2-a}_{L^{2-a}})(1 +  \|\bar p_t\|_{L^\infty( \setV )}^{2}).
\eeq
Then using estimates \eqref{H-bound-0} and \eqref{tpos} in \eqref{Hesspb}, we obtain
\begin{multline*}
 \int_{U'} |\nabla^2 p(x,t)|^{2-a} dx 
 \le C \Big\{ 1+\norm{\bar p_0}_{L^2}^2+\|\nabla p_0\|_{L^{2-a}}^{2-a} +(1+t)\sup_{[0,t]} f +\int_{0}^t  \tilde f(\tau)d\tau\Big\}^{\frac{2\muex_7}{\muex_6}+2} \\
\quad \cdot\Big\{1+ t^{-\frac 1{2\muex_6}} + t^{\frac 1 {2(\muex_6+1)}}  \sup_{[0,t]}\|\psi_t\|_{L^\infty}^\frac1{\muex_6+1}\Big\}^2.
\end{multline*}
Then \eqref{hess1b} follows.
If $t\ge 3/2$ then using \eqref{H-bound-b} and \eqref{tlarge} in \eqref{Hesspb} we obtain 
 \begin{align*}
 \int_{U'} |\nabla^2 p(x,t)|^{2-a} dx 
 \le  C\Big\{ 1+\norm{\bar p_0}_{L^2}^2 +M_f(t)^\frac2{2-a}+ \int_{t-3/2}^t \tilde f(\tau)d\tau \Big\}^{\frac{2\muex_7}{\muex_6}+ 2}\\
 \cdot\Big\{1+\sup_{[t-1/2, t]} \norm{\psi_t}_{L^\infty}^{\frac 1{\muex_6+1}}\Big\}^2.
\end{align*}
Then \eqref{h2bcomp} follows.

(ii) If $A<\infty$ then using \eqref{limsupH1} and \eqref{limsup-pt10} in \eqref{Hesspb} we obtain
\beqs
\limsup_{t\to\infty}\int_{U'} |\nabla^2 p(x,t) |^{2-a} dx \le C \Ulim_2 \Ulim_3^{2\muex_8}.
\eeqs
This yields \eqref{limsupHessb}.

(iii) If $\beta<\infty$ then using \eqref{unbounJ} and \eqref{ptbeta} in \eqref{Hesspb} we obtain \eqref{Hess3b}.
\end{proof}

Next, we treat the case $\delta\in(0,a)$ for which we define the following exponents
\begin{align}
&\nuex_1= \max\{2,\frac{(2-\delta)a}\delta\},
&& 
\nuex_2=2+\frac{2\muex_7}{\muex_6}+\frac 1 {\muex_6+1}+\frac{2(\nuex_1-2)}{2-a},\\
&\nuex_3=\frac{2-a}{1-a}\Big(\frac{2\muex_7}{\muex_6}+1\Big)+\frac{\nuex_1-a}{1-a},
&&
\nuex_4=\frac 2{2-a}\Big(\frac{2\muex_7}{\muex_6} + 1\Big)+\muex_4(\nuex_1-2),
\\
&\muex_{11}=\frac{2-a}{1-a}\Big(\frac{2\muex_7}{\muex_6}+1\Big)+\frac2{\muex_6+1},
&&\muex_{12}=\frac{2\muex_7}{\muex_6} + 1 +\frac{1-a}{2-a}\frac 2{\muex_6+1}.
\end{align}

\begin{theorem}\label{HessThm3}
 \asdc.  Let $\delta$ be any number in $(0,a)$.

 For $t>0$, we have
\beq\label{hess1c}
 \int_{U'} |\nabla^2 p(x,t)|^{2-\delta} dx 
 \le  C L_9 (1+ t^{-\frac 1{\muex_6}}) (t+1)^{\nuex_2}  \Nf_{1,t}^{\nuex_3} ( \sup_{[0,t]}\|\psi_t\|_{L^\infty}+1)^{\muex_{11}},
\eeq
where
$L_9=(1+ \norm{\bar p_0}_{L^2}^2+\|\nabla p_0\|_{L^{2-a}}^{2-a})  L_2(\nuex_1).$

For $t\ge 3/2$, we have
\beq\label{hess2c}
 \int_{U'} |\nabla^2 p(x,t)|^{2-\delta} dx 
\le C L_{10} (1+M_f(t))^{\nuex_4} \Big( 1+\sup_{[t-3/2,t]} \tilde f\Big)^{\muex_{12}}\Big\{ 1+ \int_0^t f(\tau)d\tau\Big\},
\eeq
where
$L_{10}=( 1+\norm{\bar p_0}_{L^2}^2)L_4(\nuex_1).$

Above, the positive constant $C$ depends on $U'$ and $\delta$, and $L_2(\cdot)$, $L_4(\cdot)$ are defined in Theorems \ref{thm312} and \ref{cor311}. 
\end{theorem}
\begin{proof}
We use estimate \eqref{delunite} and utilizing \eqref{grad20pw} with $s=\nuex_1$ and \eqref{tpos} to obtain
\begin{multline*}
 \int_{U'} |\nabla^2 p(x,t)|^{2-\delta} dx 
 \le C \Big\{ 1+\norm{\bar p_0}_{L^2}^2+\|\nabla p_0\|_{L^{2-a}}^{2-a}+(1+t)\Nf_{1,t}^\frac{2-a}{1-a} +\Nf_{2,t}\Big\}^{\frac{2\muex_7}{\muex_6}+1} \\
\quad \cdot\Big\{1+ t^{-\frac 1{2\muex_6}} + t^{\frac 1 {2(\muex_6+1)}}  \sup_{[0,t]}\|\psi_t\|_{L^\infty}^\frac1{\muex_6+1}\Big\}^2 
 L_2(\nuex_1)  (t+1)^{\frac{2(\nuex_1-2)}{2-a}+1} \Nf_{1,t}^{\frac{\nuex_1-a}{1-a}}.
\end{multline*}
Then \eqref{hess1c} follows. For $t\ge 3/2$, using \eqref{Kgrad55pw} for $s=\nuex_1$ instead of \eqref{grad20pw} and using \eqref{tlarge} instead of \eqref{tpos}, we obtain
\begin{multline*}
 \int_{U'} |\nabla^2 p(x,t)|^{2-\delta} dx 
 \le C \Big( 1+\norm{\bar p_0}_{L^2}^2 +M_f(t)^{\frac 2{2-a} }+ \int_{t-3/2}^t \tilde f(\tau)d\tau \Big)^{\frac{2\muex_7}{\muex_6} + 1 } \Big(1 +\sup_{[t-1/2,t]}\|\psi_t\|_{L^\infty}^\frac1{\muex_6+1}\Big)^2\\
\cdot L_4(\nuex_1)\Big(1+M_f(t) \Big)^{\muex_4(\nuex_1-2)}\Big\{1+ \int_0^t f(\tau)d\tau\Big\}.
\end{multline*}
Note that  $\|\psi_t(t)\|_{L^\infty}\le C(1+\tilde f(t))^\frac{1-a}{2-a}$.
Then we obtain \eqref{hess2c}.
\end{proof}

For asymptotic estimates, we have the following.

\begin{theorem}\label{HessLim}
\asdc.  
Suppose $\Upsilon_1$ and $\Upsilon_2$ defined in Corollary \ref{corMM} are finite numbers.
Then for any $\delta\in (0,a)$, we have
\beq\label{limhessU}
 \limsup_{t\to\infty} \int_{U'} |\nabla p(x,t)|^{2-\delta} dx
\le C  L_4(\nuex_1) \Upsilon_1^{\muex_4(\nuex_1-2)} \Upsilon_2
 \Ulim_3^{2\muex_8},
\eeq
where $C>0$ depends on $U'$ and $\delta$.
\end{theorem}
\begin{proof}
Taking limit superior of \eqref{delunite} as $t\to\infty$ and using \eqref{Kgrad60pw} for $s=\nuex_1$, and using \eqref{limsup-pt10}, we obtain
\eqref{limhessU}.
\end{proof}

\myclearpage
\section{Dependence on initial and boundary data}
\label{datacont}

In this section, we prove the continuous dependence of solutions  $p(x,t)$ of the IBVP \eqref{eqorig}, \eqref{BC} and \eqref{In-Cond} with respect to the $L^\infty$-norm on the initial data and  boundary data. The results are established for either finite time intervals or at time infinity.  

Let $p_1(x,t)$ and $p_2(x,t)$ be two solutions of the IBVP \eqref{eqgamma} having fluxes $\psi_1$ and $\psi_2$, and  initial data $p_1(x,0)$ and $p_2(x,0)$, respectively. 
Let $\Psi=\psi_1-\psi_2$, $P=p_1-p_2$, and $\bar P=P-|U|^{-1}\int_U P dx$. Then by \eqref{pbargam},
\beqs
\bar P(x,t)
=\bar p_1(x,t)-\bar p_2(x,t)
=P(x,t)- \frac{1}{|U|}\int_U P(x,0) dx + \frac{1}{|U|}\int_0^t\int_\Gamma \Psi(x,\tau) d\sigma d\tau.
\eeqs
We will estimate $\norm{\bar P}_{L^\infty(U')}$ where the subset $U'$ satisfies $U'\Subset U$ throughout the section.
From \eqref{eqgamma} follows
\beq\label{eq-1} \frac{\partial\bar P}{\partial t}=\nabla \cdot \big (K(|\nabla \bar p_1|)\nabla \bar p_1  \big ) - \nabla \cdot \big  (K(|\nabla \bar p_2|)\nabla \bar p_2 \big )
+\frac{1}{|U|}\int_\Gamma \Psi(x,t) d\sigma. \eeq

The following quantity will be used throughout this section:
\beqs
\Lambda(t) =1 + \int_U \Big(|\nabla p_1(x,t)|^{2-a} + |\nabla p_2(x,t)|^{2-a}\Big) dx, 
\eeqs

\subsection{Results for pressure}
\label{sec41}

First we establish $L^\infty$-estimates for $\bar P$ in terms of its $L^2$- and $W^{1,s}$- norms.
We recall that the exponents $\muex_5$, $\muex_6$ and $\muex_7$ are defined in \eqref{sdef}.

 \begin{proposition}\label{cont:strictcond}  \asdc.
Let $ U'\Subset \setV \Subset U$. 
Let $\mu> \muex_6^{-1}=\frac {\muex_5} {{\muex_5}-2}$ and denote 
\beq \label{etadef}
{\gamex_1}=\muex_6 -\frac1\mu =1- \frac 2 {\muex_5} -\frac 1 \mu\in(0,1).
\eeq
Then we have for any $T_0\ge 0$, $T>0$ and $\theta\in(0,1)$  that 
 \beq\label{Pinfty1}
\sup_{[T_0+\theta T,T_0+T]} \|\bar P\|_{L^\infty(U')}\le C\, {\mathcal C}_{T_0,T,\theta}\Big( \norm{\bar P}_{L^2(U\times (T_0,T_0+T) )}^{\frac{{\gamex_1}}{{\gamex_1}+1}}+ \norm{\bar P}_{L^2(U\times (T_0,T_0+T) )}\Big),
 \eeq
where the positive constant $C$ is independent of $T_0$, $T$ and $\theta$, 
 \beq\label{Cthe3}
     \mathcal C_{T_0,T,\theta}= (\lambda T^{1/2} \vartheta)^{ \frac{1}{{\gamex_1}+1}} +  \Big(\lambda\Big[1+ \frac1{\theta T}\Big]^{1/2}\Big)^\frac1{\muex_6}  + \Big(\lambda T^{1/2}\sup_{[T_0,T_0+T]}\norm{\Psi}\Big)^{ \frac{1}{\muex_6+1}},
 \eeq
with
\begin{align}
\label{lamdef2}
&\lambda = \lambda_{T_0,T} \eqdef  \sup_{t\in [T_0,T_0+T]} \Lambda(t) ^{\muex_7},\\
\label{updef}
& \vartheta =\vartheta_{T_0,T} \eqdef \Big[ \int_{T_0}^{T_0+T}\int_{\setV} \Big(|\nabla \bar p_1|^{2\mu(1-a)} + |\nabla \bar p_2|^{2\mu(1-a)} \Big)dx dt\Big]^\frac{1}{2\mu}. 
\end{align}
\end{proposition}
\begin{proof} 
Without loss of generality, we assume $T_0=0$.
Let $\zeta(x,t)=\phi(x)\varphi(t)$ be a cut-off function with $\varphi(0)=0$ and supp\,$\phi \subset \setV$. 
Same as in Proposition \ref{p-T.est} we define
\begin{align*}
 \bar P^{(k)} =\max\{\bar P-k,0\},
\quad S_{k}(t)=\{ x\in U: \bar P^{(k)}(x,t)\ge 0\},
\end{align*}
and denote by $\chi_k(x,t)$ the characteristic function of $S_k(t)$.

Multiplying equation \eqref{eq-1}  by $\bar P^{(k)} \zeta^2 $, integrating it over $U$ and using integration by parts, we have
\begin{align*}
&\frac 1 2\ddt \int_U |\bar  P^{(k)}\zeta|^2 dx 
=\int_U |\bar P^{(k)}|^2\zeta \zeta_t dx  \\
&\quad -\int_U \Big(K(\nabla \bar p_1|)\nabla \bar p_1-K(|\nabla \bar p_2|)\nabla \bar p_2 \Big ) \cdot \nabla (\bar P^{(k)} \zeta^2)  dx
+\frac{1}{|U|}\int_\Gamma \Psi(x,t) d\sigma\int_U\bar  P^{(k)}\zeta^2  dx.
\end{align*}
Elementary calculations yield
\begin{align*}
&\frac 1 2\ddt \int_U |\bar  P^{(k)}\zeta|^2 dx 
\le \int_U |\bar P^{(k)}|^2\zeta |\zeta_t| dx  -\int_U \Big(K(\nabla \bar p_1|)\nabla \bar p_1-K(|\nabla \bar p_2|)\nabla \bar p_2 \Big ) \cdot \nabla  \bar P^{(k)} \zeta^2  dx\\
&\quad  -2 \int_U \Big(K(\nabla \bar p_1|)\nabla \bar p_1-K(|\nabla \bar p_2|)\nabla \bar p_2 \Big ) \cdot   \bar P^{(k)} \zeta \nabla \zeta   dx +\frac{1}{|U|}\int_\Gamma \Psi(x,t) d\sigma\int_U\bar  P^{(k)}\zeta^2  dx.
\end{align*}
Denoting the last four summands by $I_i$, $i=1,2,3,4$, respectively, we rewrite the above as
\begin{equation}\label{EqPprime}
\frac 1 2\ddt \int_U |\bar  P^{(k)}\zeta|^2 dx \le I_1+I_2+I_3+I_4.
\end{equation}
Let $\xi(x,t)=|\nabla \bar p_1|\vee |\nabla \bar p_2|$.  
We consider $I_2$. Let $J(x,t)=\Big(K(\nabla \bar p_1|)\nabla \bar p_1-K(|\nabla \bar p_2|)\nabla \bar p_2 \Big ) \cdot \nabla  \bar P^{(k)} $.
On the set $U\setminus S_k(t)$, since $\nabla \bar P^{(k)}=0$ a.e., we have $J(x,t)=0$ a.e..
On the set $S_k(t)$, we have $\nabla \bar P^{(k)}=\nabla \bar P$ a.e., and, by the monotonicity \eqref{mono1}, we have for almost all $x\in S_k(t)$ that
\begin{align*}
J(x,t)
& = \big( K(|\nabla \bar p_1|)\nabla \bar p_1-K(\nabla \bar p_2|)\nabla \bar p_2 \big)\cdot (\nabla \bar  p_1 -\nabla \bar p_2)\\
&\ge (1-a)K(|\nabla \bar p_1|\vee |\nabla \bar p_2|) |\nabla \bar P|^2
= (1-a)K(\xi) |\nabla P^{(k)}|^2.
\end{align*}
Therefore, 
\beq\label{oriI2}
I_2\le -  (1-a)  \int_U  K(\xi)|\nabla \bar P^{(k)}|^2 \zeta^2dx.
\eeq
Note that 
\beqs
|\zeta\nabla \bar P^{(k)}|^2 = |\nabla(\bar P^{(k)} \zeta)- \bar P^{(k)}\nabla  \zeta |^2 = |\nabla (\bar P^{(k)} \zeta) |^2 - 2\nabla (\bar P^{(k)} \zeta) \cdot (\bar P^{(k)} \nabla \zeta ) + |\bar P^{(k)} \nabla \zeta|^2 .
\eeqs
Thus, \eqref{oriI2} gives
\begin{align*}
I_2
&\le -  (1-a) \int_U  K(\xi)|\nabla (\bar P^{(k)} \zeta) |^2dx +2(1-a)\int_U  K(\xi)|\nabla (\bar P^{(k)} \zeta)| | P^{(k)}\nabla \zeta |dx\\
&\quad -(1-a)\int_U  K(\xi)|\bar P^{(k)} \nabla\zeta|^2dx .
\end{align*}
Using Cauchy's inequality for the second term on the right-hand side of the previous inequality gives 
\beq\label{I2}
\begin{aligned}
I_2
&\le -  \frac {1-a} 2 \int_U  K(\xi)|\nabla (\bar P^{(k)} \zeta) |^2dx + (1-a)\int_U  K(\xi)|\bar P^{(k)} \nabla\zeta|^2dx \\
&\le -  \frac {1-a} 2 \int_U  K(\xi)|\nabla (\bar P^{(k)} \zeta) |^2dx + C\int_U  |\bar P^{(k)} \nabla\zeta|^2dx .
\end{aligned}
\eeq
For the last inequality, we used the fact that $K(\xi)$ is bounded above.
For $I_3$ and $I_4$,  we have   for any $\varepsilon>0$ that
\begin{align}
\label{I3} I_3
&\le C \int_U |\bar  P^{(k)}  |  (|\nabla \bar p_1|^{1-a} + |\nabla \bar p_2|^{1-a}) \zeta |\nabla \zeta| dx \nonumber \\
&\le \varep \int_U |\bar  P^{(k)} \zeta |^2 dx + C \varep^{-1}\int_U (|\nabla \bar p_1|^{1-a} + |\nabla \bar p_2|^{1-a})^2 \chi_k |\nabla \zeta|^2 dx,\\ 
\label{I4} I_4&\le\varep \int_U |\bar  P^{(k)} \zeta |^2 dx + C\varep^{-1}\|\Psi(t)\|_{L^\infty}^2\int_U  \chi_k \zeta^2 dx.
\end{align}
Combining \eqref{EqPprime}, \eqref{I2}, \eqref{I3} and \eqref{I4} yields 
\begin{align*}
& \frac 1 2\ddt \int_U |\bar  P^{(k)}\zeta |^2  dx
+  \frac{1-a}2 \int_U K(\xi) |\nabla \bar P^{(k)}\zeta|^{2-a}dx
\le C \int_U | \bar P^{(k)} |^2 (|\zeta_t|\zeta   +  |\nabla \zeta|^2)  dx \\
&\quad+2\varep  \int_U |\bar P^{(k)}\zeta|^2 dx +  C\varep^{-1} \|\Psi(t)\|_{L^\infty}^2 \int_U  \chi_k  \zeta^2 dx 
+C\varep^{-1}\int_U (|\nabla \bar p_1|^{1-a} + |\nabla \bar p_2|^{1-a})^2  \chi_k |\nabla \zeta|^2 dx.
\end{align*}
Integrating in time, taking the supremum in $t$ over $[0,T]$, and selecting $\varep =1/(16T)$ we find that
\begin{multline*}
\sup_{[0,T]}\norm{\bar  P^{(k)}\zeta}_{L^2(U)}^2
+  \int_0^T \int_U K(\xi)  |\nabla \bar P^{(k)} \zeta |^2dx dt\le C\int_0^T \int_U | \bar P^{(k)}|^2 (|\zeta_t|\zeta  +|\nabla \zeta|^2 )dxdt\\
\quad +CT\sup_{[0,T]}\|\Psi\|_{L^\infty}^2 \int_0^T \int_U  \chi_k \zeta^2 dx dt+CT \int_0^T\int_U (|\nabla \bar p_1|^{1-a} + |\nabla \bar p_2|^{1-a})^2   \chi_k |\nabla \zeta|^{2} dx dt.
\end{multline*}
Applying H\"older's inequality to the last double integral yields
\beq\label{Pkpre}
\begin{aligned}
&\sup_{[0,T]}\norm{\bar  P^{(k)}\zeta}_{L^2(U)}^2
+  \int_0^T \int_U K(\xi)  |\nabla \bar P^{(k)} \zeta |^2dx dt\\
&\le C\int_0^T \int_U | \bar P^{(k)}|^2 (|\zeta_t|\zeta  +|\nabla \zeta|^2 )dx
 +CT\sup_{[0,T]}\|\Psi\|_{L^\infty}^2 \int_0^T \int_U \chi_k \zeta^2 dxdt\\
&\quad  + CT\Big( \int_0^T\int_{\setV} (|\nabla \bar p_1|^{1-a} + |\nabla \bar p_2|^{1-a})^{2\mu}  |\nabla \zeta|^{2\mu} dx dt\Big)^{1/\mu} 
\Big(\iint_{Q_T\cap {\rm supp}\zeta}\chi_k  dxdt\Big)^{1-1/\mu}.
\end{aligned}
\eeq
Under the Strict Degree Condition, by applying  Sobolev inequality \eqref{Wemb} with $W(x,t)=K(\xi(x,t))$, we have 
\beqs
\|\bar  P^{(k)}\zeta\|_{L^{\muex_5}(Q_T)}  
\le C\lambda \Big(\sup_{[0,T]}\norm{\bar  P^{(k)}\zeta}_{L^2(U)}^2
+  \int_0^T \int_U K(\xi)  |\nabla \bar P^{(k)} \zeta |^2dx dt\Big)^{1/2},
\eeqs
where $\lambda$ is defined by \eqref{lamdef2}. Combining with \eqref{Pkpre}, we have
\begin{multline}\label{Pkzeta}
\|\bar  P^{(k)}\zeta\|_{L^{\muex_5}(Q_T)}  
\le C\lambda \Big[ \Big(\int_0^T \int_U | \bar P^{(k)}|^2 (|\zeta_t|+ |\nabla \zeta|^2)dxdt\Big)^{1/2}\\
+ T^{1/2} \sup_{[0,T]} \|\Psi\|_{L^\infty}\Big(\iint_{Q_T\cap{\rm supp}\zeta} \chi_k  dxdt\Big)^{1/2}
+T^{1/2}\vartheta \Big(\iint_{Q_T\cap{\rm supp}\zeta} \chi_k  dxdt\Big)^{\frac 12(1-1/\mu)}\Big],
\end{multline}
where $\vartheta$ is defined by \eqref{updef}.

Let $M_0 >0$ be fixed  which we will determine later.
For integers $i\ge 0$, let $k_i= M_0(1-2^{-i})$ and where $\zeta_i$ be defined as in \eqref{der:cutoff} and the sets  $\mathcal Q_i$, $A_{i,j}$ be  defined similar to \eqref{defQnAnm} replacing $\bar p_t$ by $\bar P$. 

Define $F_{i} = \norm{\bar P ^{(k_{i+1})}\zeta_i}_{L^{\muex_5}(A_{i+1,i}) }$.
Applying \eqref{Pkzeta} with $k=k_{i+1}$ and $\zeta=\zeta_i$ we have
 \beq\label{more1}
  \begin{aligned}
  F_{i}
&\le C\lambda \Big[ \big(2^{i/2} (\theta T)^{-1/2} + 2^i\big) \norm{ \bar P^{(k_{i+1})}}_{L^2(A_{i+1, i})}\\ 
&\quad +2^iT^{1/2}\vartheta |A_{i+1,i}|^{\frac 12(1-1/\mu)}  + T^{1/2}\sup_{[0,T]}\|\Psi\|_{L^\infty} |A_{i+1,i}|^{1/2} \Big].
  \end{aligned}
 \eeq
Estimating the same way as in \eqref{qki+1}, we have
\beq\label{more2}
 \|\bar P^{(k_{i+1})}\zeta_i \|_{L^2(A_{i+1,i+1})} \le C F_{i}  |A_{i+1,i}|^{1/2-1/{\muex_5}}.
\eeq
Also, similar to \eqref{Aii}, we have 
\beq\label{more3}
  |A_{i+1,i}| \le C 4^{i} M_0^{-2}\|\bar P^{(k_i)}\|_{L^2(A_i)}^{2}.  
\eeq
From estimates \eqref{more1}, \eqref{more2}, and the boundedness of $\zeta_t$, we obtain
\begin{align*}
&\|\bar P^{(k_{i+1})}\|_{L^2(A_{i+1})} 
=\|\bar P^{(k_{i+1})}\zeta_i\|_{L^2(A_{i+1})} 
\le C\lambda \Big\{ 2^{i}( (\theta T)^{-1/2} + 1) \norm{ \bar P^{(k_{i+1})}}_{L^2(A_{i+1, i})}\\ 
&\quad +2^iT^{1/2}\vartheta |A_{i+1,i}|^{\frac 12(1-1/\mu)}  + T^{1/2}\sup_{[0,T]}\|\Psi\|_{L^\infty}|A_{i+1,i}|^{1/2} \Big\} |A_{i+1,i}|^{1/2-1/{\muex_5}}.
\end{align*}
Using \eqref{more3} gives
\begin{align*}
\|\bar P^{(k_{i+1})}\|_{L^2(A_{i+1})} 
&\le C\lambda \Big\{ 2^{i}( (\theta T)^{-1/2} + 1) \norm{ \bar P^{(k_{i+1})}}_{L^2(A_{i+1, i})} 
 + 4^iT^{1/2}\vartheta  M_0^{-1+1/\mu} \|\bar P^{(k_i)}\|_{L^2(A_i)}^{1-1/\mu}\\
&\quad   + T^{1/2}2^iM_0^{-1}\sup_{[0,T]}\|\Psi\|_{L^\infty} \|\bar P^{(k_i)}\|_{L^2(A_i)} \Big\} 2^{i} M_0^{-(1-\frac 2 {\muex_5})}\|\bar P^{(k_i)}\|_{L^2(A_i)}^{1-\frac 2 {\muex_5}}.
\end{align*}

Put $Y_i=\| \bar P^{(k_i)} \|_{L^2(A_i)}$, the previous inequality gives        
\begin{align*}
&Y_{i+1}\le C\lambda  M_0^{-1+\frac 2 {\muex_5}}8^i
\Big\{\vartheta  T^\frac12 M_0^{-1+\frac1\mu} Y_i^{1-\frac1\mu} + \Big[(1+ \frac1{\theta T})^\frac12+T^\frac12\sup_{[0,T]}\|\Psi\|_{L^\infty} M_0^{-1}\Big]Y_i   \Big\} Y_i^{1-\frac 2 s}\\
&\le C  8^i
\Big\{\lambda\vartheta  T^\frac12 M_0^{-1-{\gamex_1}} Y_i^{1+{\gamex_1}} 
+ \lambda\Big[(1+ \frac1{\theta T})^\frac12 M_0^{-\muex_6}\Big]Y_i^{1+\muex_6}
+\lambda T^\frac12\sup_{[0,T]}\|\Psi\|_{L^\infty}M_0^{-1-\muex_6}Y_i^{1+\muex_6}   \Big\}.
\end{align*}
Since $Y_0\le \norm{\bar P}_{L^2(U\times (0,T) )} $, aiming at applying Lemma~\ref{multiseq} with $m=3$ we choose $M_0$ sufficiently large such that 
  \[
  \begin{split}
  \norm{\bar P}_{L^2(U\times (0,T) )}
\le C\min\Big\{
&(\lambda T^{1/2} \vartheta)^{- 1/{\gamex_1}}M_0^\frac{1+{\gamex_1}}{{\gamex_1}},  (\lambda(1+ \frac1{\theta T})^{1/2})^{-1/\muex_6}M_0, \\
& (\lambda T^{1/2}\sup_{[0,T]}\|\Psi\|_{L^\infty})^{- 1/\muex_6} M_0^\frac{1+\muex_6}{\muex_6}  \Big\}.
  \end{split}
  \]
Thus, we require    
 \begin{align*}
  M_0&\ge C(\lambda T^{1/2}\vartheta)^{ \frac{1}{{\gamex_1}+1}} \norm{\bar P}_{L^2(U\times (0,T) )}^{\frac{{\gamex_1}}{{\gamex_1}+1}},\quad
   M_0\ge C (\lambda(1+ \frac1{\theta T})^{1/2})^{1/\muex_6} \norm{\bar P}_{L^2(U\times (0,T) )}, \\
   M_0&\ge C(\lambda T^{1/2} \sup_{[0,T]}\|\Psi\|_{L^\infty})^{ \frac{1}{\muex_6+1}} \norm{\bar P}_{L^2(U\times (0,T) )}^{\frac{\muex_6}{\muex_6+1}}. 
  \end{align*}
 We choose 
 \begin{align}\label{M0}
  M_0 \ge C \Big((\lambda T^{1/2} \vartheta)^{ \frac{1}{{\gamex_1}+1}} + (\lambda(1+ \frac1{\theta T})^{1/2})^{1/\muex_6}  + (\lambda T^{1/2} \sup_{[0,T]}\norm{\Psi}_{L^\infty})^{ \frac{1}{\muex_6+1}} \Big)\nonumber\\ 
\cdot\Big( \norm{\bar P}_{L^2(U\times (0,T) )}^{\frac{{\gamex_1}}{{\gamex_1}+1}}+ \norm{\bar P}_{L^2(U\times (0,T) )} + \norm{\bar P}_{L^2(U\times (0,T) )}^{\frac{\muex_6}{\muex_6+1}}\Big), 
 \end{align}
 for some constant $C>0$.  
Since $\lambda\ge 1$ and ${\gamex_1}<\muex_6<1$, we select
\beq\label{firstM}
\begin{aligned}
 M_0 &= C \Big((\lambda T^{1/2} \vartheta)^{ \frac{1}{{\gamex_1}+1}} +  (\lambda(1+ \frac1{\theta T})^{1/2})^{1/\muex_6}  + ( \lambda T^{1/2}\sup_{[0,T]}\norm{\Psi})^{ \frac{1}{\muex_6+1}} \Big) \\
&\quad \cdot \Big( \norm{\bar P}_{L^2(U\times (0,T) )}^{\frac{{\gamex_1}}{{\gamex_1}+1}}+ \norm{\bar P}_{L^2(U\times (0,T) )}\Big). 
\end{aligned}
\eeq
Applying Lemma \ref{multiseq} and using the same arguments as in Proposition \ref{local-L-infty}, we obtain \eqref{Pinfty1}.
 \end{proof}

We now derive a variation of Proposition \ref{cont:strictcond}. Let $T_0=0$.
In the proof of Proposition \ref{cont:strictcond}, let $\zeta =\zeta(x) $ be the cut-off function such that $\zeta$ vanishes in neighborhood of the boundary $\Gamma$. 
If $k\ge \|\bar P_0\|_{L^\infty}$, then $\bar P^{(k)}(0)=0$ and, hence, \eqref{Pkzeta} holds true with $\zeta_t=0$. 
In the iteration process, we choose $M_0 \geq \|\bar P(\cdot, 0)\|_{L^\infty}$ and $k_i=M_0(2-2^{-i})$ for $i\ge 0$. This way, we have $k_i \geq M_0\ge \|\bar P_0\|_{L^\infty}$, thus, $\bar P^{(k_i)}(0)=0$.
Therefore, we can replace $\frac1{\theta T}$ by $0$ in \eqref{M0}, which becomes     
  \begin{align*}
 M_0 &\ge C \Big((\lambda T^{1/2} \vartheta)^{ \frac{1}{{\gamex_1}+1}} + \lambda^{1/\muex_6}  + (\lambda T^{1/2} \sup_{[0,T]}\norm{\Psi})^{ \frac{1}{\muex_6+1}} \Big) \\
&\quad \cdot\Big( \norm{\bar P}_{L^2(U\times (0,T) )}^{\frac{{\gamex_1}}{{\gamex_1}+1}}+ \norm{\bar P}_{L^2(U\times (0,T) )} + \norm{\bar P}_{L^2(U\times (0,T) )}^{\frac{\muex_6}{\muex_6+1}}\Big).
 \end{align*}
Therefore, instead of \eqref{firstM}, we can select 
\begin{align*}
 M_0 
&= \|\bar P(0)\|_{L^\infty} \\
&\quad +C \Big((\lambda T^{1/2} \vartheta)^{ \frac{1}{{\gamex_1}+1}}  + \lambda^{1/\muex_6}+ (\lambda T^{1/2} \sup_{[0,T]}\norm{\Psi})^{ \frac{1}{\muex_6+1}} \Big) 
\Big( \norm{\bar P}_{L^2(U\times (0,T) )}^{\frac{{\gamex_1}}{{\gamex_1}+1}}+ \norm{\bar P}_{L^2(U\times (0,T) )}\Big).
 \end{align*}
 Applying Lemma \ref{multiseq}, we have 
$\lim_{i \rightarrow \infty}Y_i = 0$
 and $\bar p(x,t)\le 2M_0$ a.e. in $B(x_0,\rho/4)\times(\theta T,T)$. Proceeding the proof as in Proposition \ref{local-L-infty}, we obtain the following result.  

 \begin{proposition}\label{newcont}  Assume the same as in Proposition \ref{cont:strictcond}. Then we have for any $T>0$ that
 \beq\label{Pinfty2}
 \sup_{[0,T]} \|\bar P\|_{L^\infty(U')}\le 2 \|\bar P(0)\|_{L^\infty}+C\, \mathcal C_{T}\Big( \norm{\bar P}_{L^2(U\times (0,T) )}^{\frac{{\gamex_1}}{{\gamex_1}+1}}+ \norm{\bar P}_{L^2(U\times (0,T) )}\Big),
 \eeq
where
 \beq\label{Cthe10}
    \mathcal  C_{T}= (\lambda_T T^{1/2} \vartheta_T)^{ \frac{1}{{\gamex_1}+1}}  + \lambda_T^\frac{1}{\muex_6}+\Big(\lambda_T T^{1/2}\sup_{[0,T]}\norm{\Psi}_{L^\infty}\Big)^{ \frac{1}{\muex_6+1}},
 \eeq
with $\lambda_T=\lambda_{0,T}$ and $\vartheta_T=\vartheta_{0,T}$ defined in Theorem \ref{cont:strictcond} for $T_0=0$.
 \end{proposition}

We will use Propositions \ref{cont:strictcond} and \ref{newcont} to obtain specific $L^\infty$-estimates of $\bar P$ in terms of initial and boundary data.
First, we introduce some quantities and parameters. 

Same as \eqref{b}, we define for $i=1,2$, 
\beqs
f_i(t)=\|\psi_i(t)\|_{L^\infty}^2+\|\psi_i(t)\|_{L^\infty}^{\frac{2-a}{1-a}} \quad\text{and}\quad
\tilde f_i(t)=\|\psi_{it}(t)\|_{L^\infty}^2+\|\psi_{it}(t)\|_{L^\infty}^{\frac{2-a}{1-a}}.
\eeqs

For $i=1,2$, we assume $f_i(t),\tilde f_i(t)\in C([0,\infty))$ and when needed $f_i(t)\in C^1((0,\infty))$.
Let 
\beqs 
A_i=\limsup_{t\to\infty} f_i(t)\quad \text{and}\quad \beta_i=\limsup_{t\to\infty} [f_i'(t)]^-,
\eeqs 
\beqs
\bar A=A_1+A_2\quad \text{and}\quad   \bar \beta=\beta_1+\beta_2.
\eeqs

Let $M_{f_i}(t)$, $i=1,2$, be a continuous increasing majorant  of $f_i(t)$ on $[0,\infty)$.  
Set
\beqs
\myB(t)= f_1(t)+ f_2(t),\quad   \ M_\myB(t)=M_{f_1}(t)+M_{f_2}(t)\quad\text{and}\quad
\tilde \myB(t)=\tilde f_1(t)+\tilde f_2(t).
\eeqs 

For initial data, set 
\beq\label{AB0}
A_0=\|\bar p_1(0)\|_{L^2}^2 +\|\bar p_2(0)\|_{L^2}^2  \quad \text{and} \quad B_0=J_H[p_1](0)+J_H[p_2](0).
\eeq 

We recall results from \cite{HI2}. 


\begin{theorem}[cf. \cite{HI2}, Lemma 5.5 and Theorem 5.6]\label{pgrowth}
{\rm (i)} For $t\ge 0$,
\begin{align}\label{pcont}
\norm{\bar P(t)}_{L^2}^2 \le \norm{\bar P(0)}_{L^2}^2 
+C\int_0^t\norm{\Psi(\tau)}_{L^\infty}^2\Lambda(\tau)^b d\tau,
\end{align}
where $b$ is defined in \eqref{ab}.
Consequently, for any $T>0$,
\begin{align}\label{pLip}
\sup_{[0,T]}\norm{\bar P}_{L^2}^2  \le \norm{\bar P(0)}_{L^2}^2 
+C\cdot  M_{1,T} \sup_{[0,T]}\norm{\Psi(t)}_{L^\infty}^2 ,
\end{align}
where $M_{1,T}=A_0+T+\int_0^T\big[ f_1(\tau)+f_2(\tau)\big] d\tau$.

{\rm (ii)} Assume the Degree Condition. Then for $t\ge 0$
\beq\label{psubt} 
\norm{\bar P(t)}_{L^2}^2 \le  e^{-d_2\int_0^t \Lambda^{-b}(\tau)d\tau} \norm{\bar P(0)}_{L^2}^2 
+\int_0^t e^{-d_2\int_\tau^t \Lambda^{-b}(\theta)d\theta}\norm{\Psi(\tau)}_{L^\infty}^2\Lambda(\tau)^{b} d\tau. \eeq
%
%
%

Moreover, if $\bar A<\infty$ and   
$\int_1^\infty(1+\int_{\tau-1}^\tau \tilde \myB(s)ds )^{-b} d\tau=\infty$
then
 \beq\label{psub-bounded} 
 \limsup_{t\to\infty} \norm{\bar P(t)}_{L^2}^2
\le C \limsup_{t\to\infty}\Big\{\norm{\Psi(t)}_{L^\infty}^2\Big(1+{\bar A}^\frac2{2-a}+\int_{t-1}^t \tilde \myB(\tau)d\tau\Big)^{2b}\Big\}.\eeq
\end{theorem}

We now state the continuous dependence in interior $L^\infty$-norm.
We use similar quantities to $\Nf_{1,t}$ in \eqref{Nf1} and $\Nf_{2,t}$ \eqref{Nf2}, namely,
\beq\label{Nfbar} 
\bar\Nf_{1,T}= 1+\sum_{i=1,2} \sup_{[0,T]}\|\psi_i\|_{L^\infty} \quad\text{and}\quad
\bar\Nf_{2,T}= 1 + \int_0^T \tilde F(\tau)d\tau.
\eeq

\begin{theorem}\label{thm46}
 \asdc. 
Let $\mu$ be any number satisfying 
\beq
\mu>\frac {\muex_5}{{\muex_5}-2}\quad\text{and}\quad \mu\ge \frac{2-a}{2(1-a)}.
\eeq
For $T>0$, we have
\beq\label{contfinite10}
\begin{aligned}
&\sup_{[0,T]}\|\bar P\|_{L^\infty(U')} 
\le 2 \|\bar P(0)\|_{L^\infty}\\
&\quad + C L_{11}  M_{2,T} \Big(\|\bar P(0)\|_{L^2}  + \sup_{[0,T]}\norm{\Psi(t)}_{L^\infty}
+\big[ \|\bar P(0)\|_{L^2} +\sup_{[0,T]}\norm{\Psi(t)}_{L^\infty}\big]^{\frac { {\gamex_1}}{{\gamex_1}+1}}\Big),
\end{aligned}
\eeq
where ${\gamex_1}$ is defined by \eqref{etadef}, number $L_{11}>0$ depends on the initial data and is defined by \eqref{L13} below, number $M_{2,T}>0$ depends on the boundary data and is defined by \eqref{NT} below.
\end{theorem}
\begin{proof}
Many exponents will be needed in our proof and are defined here: 
\begin{align}
\label{mu34}
&\gamex_2=2(1-a)\mu\ge 2-a,
&&\gamex_3=\frac{\gamex_2}{2\mu(1-a)}=1,
\\
\label{mu56}
&\gamex_4=\frac1{2\mu}\Big(\frac{2\gamex_2}{2-a}-1\Big)=\frac{2(1-a)}{2-a}-\frac1{2\mu},
&&\gamex_5=\frac{\muex_7}{\muex_6} + \frac{2\gamex_4+1}{2({\gamex_1}+1)},\\
\label{mu89}
& \gamex_6=\frac{\muex_7(2-a)}{\muex_6(1-a)}+\frac{\gamex_3}{{\gamex_1}+1}.
\end{align}
We will use notation $\gamex_3$ in calculations below instead of its explicit value for the sake of generality which will be needed in section \ref{polycont}.
We prove \eqref{contfinite10} with $M_{2,T}$ explicitly given by
\beq\label{NT}
M_{2,T}= T^{\mu_{20}}  \bar\Nf_{1,T}^{\gamex_6+\frac{2-a}{2(1-a)}} \bar\Nf_{2,T}^\frac{\muex_7}{\muex_6},
\eeq
where the exponent $\mu_{20}$ is $\frac{{\gamex_1}}{2({\gamex_1}+1)}$ in case $T\le 1$ and is $\gamex_5+1$ in case $T> 1$.

We will apply Proposition \ref{cont:strictcond} for $T_0=0$ and Proposition  \ref{newcont}. 
Fix a subset $\setV$ of $U$ such that $U'\Subset V\Subset U$.
First, for $M_{1,T}$ in \eqref{pLip} we note that 
\beq\label{M1plus}
1+M_{1,T} \le C \Big\{ 1+\sum_{i=1}^2 \norm{\bar p_i(0)}_{L^2}^2 + (T+1) \Big[ (\sum_{i=1}^2 \sup_{[0,T]} \|\psi_i\|)^\frac{2-a}{1-a}+1\Big] \Big\}
\le C \ell_0 (T+1)\bar\Nf_{1,T}^\frac{2-a}{1-a},
\eeq
where $\ell_0=1+A_0$.
By \eqref{pLip} and \eqref{M1plus}, we have
 \begin{multline}\label{st0}
\norm{\bar P}_{L^2(U\times (0,T) )}^{\frac {\gamex_1} {{\gamex_1}+1}}  +\norm{\bar P}_{L^2(U\times (0,T) )}
\le C  \ell_0^{1/2} (T+1)^{1/2}\bar\Nf_{1,T}^\frac{2-a}{2(1-a)} \\
\cdot \Big (
[T^{1/2}(\|\bar P(0)\|_{L^2} + \sup_{[0,T]}\norm{\Psi(t)}_{L^\infty})]^{\frac { {\gamex_1}}{{\gamex_1}+1}}
+T^{1/2}(\|\bar P(0)\|_{L^2} + \sup_{[0,T]}\norm{\Psi(t)}_{L^\infty})
\Big).
\end{multline}
Second, we will estimate ${\mathcal C}_{T_0,T,\theta}$ in \eqref{Cthe3} with $T_0=0$, and $\mathcal C_T$ in \eqref{Cthe10}. 
To simplify our calculations, we will replace ${\mathcal C}_{0,T,\theta}$ in \eqref{Cthe3} by the following upper bound
 \beq\label{Cthesimple}
      {\mathcal C}_{T,\theta}= \lambda_T^\frac1{\muex_6} \Big\{ (T^{1/2} \vartheta_T)^{ \frac{1}{{\gamex_1}+1} } +  (1+ \frac1{(\theta T)^{1/2}})^\frac1{\muex_6}  + (T^{1/2}\sup_{[0,T]}\norm{\Psi}_{L^\infty})^{ \frac{1}{\muex_6+1}} \Big\},
 \eeq
and replace ${\mathcal C}_T$ in \eqref{Cthe10} by
 \beq\label{CTsimple}
      {\mathcal C}_{T}= \lambda_T^\frac{1}{\muex_6} \Big\{ 1+ (T^{1/2} \vartheta_T)^{ \frac{1}{{\gamex_1}+1}}  + (T^{1/2}\sup_{[0,T]}\norm{\Psi}_{L^\infty})^{ \frac{1}{\muex_6+1}} \Big\}.
 \eeq
We now find bounds for involved quantities in \eqref{Cthesimple} and \eqref{CTsimple}.
We have  from estimate \eqref{H-bound-0} that
\begin{align*}
\lambda_T=\sup_{t\in [0,T]} \Lambda(t) ^{\muex_7}
&\le  C\sum_{i=1}^2\Big\{ 1 +  \norm{\bar p_i(0)}_{L^2}^2 +  J_H[\bar p_i](0) + (T+1) \sup_{[0,T]} \|\psi_i\|^\frac{2-a}{1-a} +\int_0^T \tilde f_i(\tau)d\tau \Big \}^{\muex_7}.
\end{align*}
Therefore,
\beq\label{lamd}
\lambda_T\le C \ell_1 (T+1)^{\muex_7} \bar\Nf_{1,T}^\frac{\muex_7(2-a)}{1-a} \bar\Nf_{2,T}^{\muex_7},
\eeq
where  $\ell_1=(1 +A_0+B_0)^{2-a}$.
Applying  \eqref{grad39} to $s=\gamex_2$ and $\setV$ replacing $U'$, we have
\beq\label{gradInt20}
\vartheta_T
= \Big(\int_0^T \int_{\setV}\Big[ 1+ |\nabla \bar p_1|^{\gamex_2} + |\nabla \bar p_2|^{\gamex_2}\Big]dx dt\Big)^{\frac 1{2\mu}} 
 \le C \ell_2 (T+1)^{\gamex_4}\bar\Nf_{1,T}^{\gamex_3},
\eeq
where $\ell_2=\big\{\sum_{i=1,2} L_1\big(\gamex_2+a;[p_i(0)]\big)\big\}^\frac1{2\mu}$.
Also,  
\beq\label{Psisim}
\sup_{[0,T]} \norm{\Psi}_{L^\infty} \le \Big(\sum_{i=1}^2 \sup_{[0,T]} \norm{\psi_i}_{L^\infty}\Big) \le \bar\Nf_{1,T}.
\eeq
We denote
$D_T=\|\bar P(0)\|_{L^2}^{\frac { {\gamex_1}}{{\gamex_1}+1}} +\|\bar P(0)\|_{L^2} 
+ \sup_{[0,T]}\norm{\Psi(t)}_{L^\infty}^{\frac { {\gamex_1}}{{\gamex_1}+1}} + \sup_{[0,T]}\norm{\Psi(t)}_{L^\infty}.$

We consider $0<T\le 1$ first. By \eqref{st0} and \eqref{gradInt20},  we respectively have
 \beq\label{st10}
\norm{\bar P}_{L^2(U\times (0,T) )}^{\frac {\gamex_1} {{\gamex_1}+1}}  +\norm{\bar P}_{L^2(U\times (0,T) )}
\le C \ell_0^\frac12 T^\frac{{\gamex_1}}{2({\gamex_1}+1)}  \bar\Nf_{1,T}^\frac{2-a}{2(1-a)} D_T,
\eeq
\beq\label{gradInt}
\vartheta_T \le C \ell_2 \bar\Nf_{1,T}^{\gamex_3}.
\eeq
By \eqref{CTsimple}, \eqref{lamd}, \eqref{gradInt} and \eqref{Psisim}, we have
\beq\label{st30}
\begin{aligned}
 {\mathcal C}_{T}
&\le C  \ell_1^\frac1{\muex_6} (T+1)^\frac{\muex_7}{\muex_6} \bar\Nf_{1,T}^{\frac{\muex_7(2-a)}{\muex_6(1-a)}} \bar\Nf_{2,T}^\frac{\muex_7}{\muex_6}  
\Big\{ 1+ \ell_2^\frac{1}{{\gamex_1}+1} \bar\Nf_{1,T}^\frac{{\gamex_3}}{{\gamex_1}+1} + \ell_2^\frac{1}{\muex_6+1} \bar\Nf_{1,T}^\frac{1}{\muex_6+1} \Big\}
\le C  \ell_3  \bar\Nf_{1,T}^{\gamex_6} \bar\Nf_{2,T}^\frac{\muex_7}{\muex_6} ,
\end{aligned}
\eeq
where $\ell_3=\ell_1^\frac1{\muex_6} \ell_2^\frac{1}{{\gamex_1}+1}.$ 
Note that we used the facts $\gamex_3\ge 1$ and ${\gamex_1}<\muex_6<{\gamex_1}+1$.
Applying \eqref{Pinfty2}  with the use of \eqref{st10} and \eqref{st30}, we have
\beq\label{contfinite0}
\begin{aligned}
\sup_{[0,T]}  \|\bar P(t)\|_{L^\infty(U')}
  &\le 2 \|\bar P(0)\|_{L^\infty} +C \ell_0^\frac12 \ell_3 T^{\frac { {\gamex_1}}{2({\gamex_1}+1)}}\bar\Nf_{1,T}^{\gamex_6+\frac{2-a}{2(1-a)}} \bar\Nf_{2,T}^\frac{\muex_7}{\muex_6} \cdot D_T, 
 \end{aligned}
 \eeq
hence, obtaining \eqref{contfinite10} for $T\le 1$ with 
\beq\label{L13}
L_{11}=\ell_0^\frac12 \ell_3=\ell_0^\frac12\ell_1^\frac1{\muex_6} \ell_2^\frac{1}{{\gamex_1}+1}.
\eeq

Consider  $T>1$ now.  By \eqref{st0},
 \beq\label{st1}
\norm{\bar P}_{L^2(U\times (0,T) )}^{\frac {\gamex_1} {{\gamex_1}+1}}  +\norm{\bar P}_{L^2(U\times (0,T) )}\le C\ell_0^{1/2}  T  \bar\Nf_{1,T}^\frac{2-a}{2(1-a)} D_T.
\eeq
Using $\mathcal C_{T,\theta}$ in \eqref{Cthesimple} with bounds  \eqref{lamd}, \eqref{gradInt20} and \eqref{Psisim}, we have
\begin{align*}
 {\mathcal C}_{T,\theta}
&\le C \ell_1^\frac1{\muex_6} (T+1)^\frac{\muex_7}{\muex_6} \bar\Nf_{1,T}^\frac{\muex_7(2-a)}{\muex_6(1-a)}  \bar\Nf_{2,T}^\frac{\muex_7}{\muex_6} 
\Big\{
1+ T^{\frac 1 {2({\gamex_1}+1)} } \ell_2^\frac1{{\gamex_1}+1} (T+1)^\frac{\gamex_4}{{\gamex_1}+1} \bar\Nf_{1,T}^\frac{\gamex_3}{{\gamex_1}+1} 
 + T^{\frac 1 {2(\muex_6+1)} } \bar\Nf_{1,T}^\frac{1}{\muex_6+1} 
 \Big\}.
\end{align*}
Thus,  
\beq\label{st3}
 {\mathcal C}_{T,\theta}
\le C \ell_3 T^{\gamex_5} \bar\Nf_{1,T}^{\gamex_6}\bar\Nf_{2,T}^\frac{\muex_7}{\muex_6}.
\eeq
Therefore, combining \eqref{Pinfty1}, \eqref{st1} and \eqref{st3} yields
\beq\label{contfinite}
\sup_{[1,T]}\|\bar P\|_{L^\infty(U')} 
\le  C L_{11}  T^{\gamex_5+1}\bar\Nf_{1,T}^{\gamex_6+\frac{2-a}{2(1-a)} } \bar\Nf_{2,T}^\frac{\muex_7}{\muex_6}  \cdot D_T.
\eeq
Then combining \eqref{contfinite0} for $T=1$ with  \eqref{contfinite}, noting that $\gamex_5\ge \frac{{\gamex_1}}{2({\gamex_1}+1)}$, we obtain \eqref{contfinite10} for $T>1$.
The proof is complete.
\end{proof}


Now, we derive asymptotic estimates as $t\to\infty$. Let $t>2$. Applying Proposition \ref{cont:strictcond} to $T_0=t-1$,  $T=1$ and $\theta=1/2$, we have
\beq\label{bigPbar}
\|\bar P(t)\|_{L^\infty(U')} \le C  \tilde{\mathcal C}(t) \Big( \norm{\bar P}_{L^2(U\times (t-1,t) )}^{\frac {\gamex_1}{{\gamex_1}+1} }  +\norm{\bar P}_{L^2(U\times (t-1,t) )}  \Big),
\eeq
where 
 \beq\label{Clam}
 \tilde{\mathcal C}(t)= [\tilde \lambda(t)\tilde \vartheta(t)]^{ \frac{1}{{\gamex_1}+1}} + \tilde \lambda(t)^\frac1{\muex_6} +[\tilde \lambda(t) \sup_{[t-1,t]}\norm{\Psi}_{L^\infty}]^{ \frac{1}{\muex_6+1}}
 \eeq
with
\begin{align}
\label{lamtildefn}
&\tilde \lambda(t) = \sup_{\tau\in [t-1,t]} \Lambda(\tau)^{\muex_7},\\
\label{uptildef}
&\tilde \vartheta(t) = \Big[ \int_{t-1}^t\int_{V}\Big( |\nabla \bar p_1|^{2\mu(1-a)} + |\nabla \bar p_2|^{2\mu(1-a)}\Big) dx dt\Big]^\frac{1}{2\mu}.
\end{align}

Before going into specific estimates, we state a general result on the limit superior of $\|\bar P(t)\|_{L^\infty(U')}$ as $t\to\infty$.
It is of the same spirit as of the $L^2$-result \eqref{psub-bounded}.
 
 \begin{lemma}\label{limPprop}
\asdc.  Suppose 
  \beq\label{boundcond}
  \sup_{\tau\in[t-1,t]}\Lambda(\tau)\le \kappa_0 g(t),\quad \tilde{\mathcal C}(t)\le C B(t)
  \eeq
   for  sufficiently large $t$, with some functions $g(t),B(t)\ge 1$. 
Let $\kappa_1=d_2 \kappa_0^{-b}$, where $d_2$ is the positive constant in \eqref{psubt}.
If 
\beq\label{condlim1}
\lim_{t\to\infty} B(t)^\frac{2({\gamex_1}+1)}{{\gamex_1}}e^{-\kappa_1\int_{2}^t g(\tau)^{-b}d\tau}  =0
\quad\text{and}\quad
\lim_{t\to\infty}  g(t)^b\frac {B'(t)}{B(t)}  =0,
\eeq
then
\begin{align}\label{limsup00}
\limsup_{t\to\infty}  \|\bar P(t)\|_{L^\infty(U')}
&\le C \limsup_{t\to\infty} \Big\{  B(t) \Big( g(t)^{b} \norm{\Psi(t)}_{L^\infty} + \big[ g(t)^{b} \norm{\Psi(t)}_{L^\infty} \big]^\frac{\gamex_1}{{\gamex_1}+1}\Big)\Big\}.
\end{align}
 \end{lemma}
\begin{proof} 
Let $T>2$ be sufficiently large such that \eqref{boundcond} holds for all $t>T$.
 By \eqref{psubt} and \eqref{boundcond}, we have for $t'>T$ that
\begin{align*}
 \norm{\bar P(t')}_{L^2}^2 
&\le  e^{-\kappa_1\int_T^{t'} g^{-b}(\tau)d\tau}  \norm{\bar P(T)}_{L^2}^2 \\
&\quad +e^{-\kappa_1\int_T^{t'} g^{-b}(\tau)d\tau} \int_T^{t'} e^{\kappa_1\int_T^\tau g^{-b}(\theta)d\theta}\norm{\Psi(\tau)}_{L^\infty}^2 g(\tau)^{b} d\tau. 
\end{align*}
Then similar to \eqref{supptint}, we have
\beq\label{supPtprime}
 \sup_{[t-1,t]}  \norm{\bar P(t')}_{L^2}^2  
\le C\Big(e^{-\kappa_1\int_T^{t} g^{-b}(\tau)d\tau}  \norm{\bar P(T)}_{L^2}^2  +  \int_T^t e^{-\kappa_1\int_\tau^t g^{-b}(\theta)d\theta} \norm{\Psi(\tau)}_{L^\infty}^2 g(\tau)^b d\tau \Big).
\eeq
Combining \eqref{bigPbar} with \eqref{boundcond} and estimate \eqref{supPtprime},  we obtain
\begin{multline*}
  \norm{\bar P(t)}_{L^\infty(U')}
\le C B(t) \Big\{ 
e^{-\kappa_1\int_T^{t} g^{-b}(\tau)d\tau}  \norm{\bar P(T)}_{L^2}^2  +  \int_T^t e^{-\kappa_1\int_\tau^t g^{-b}(\theta)d\theta} \norm{\Psi(\tau)}_{L^\infty}^2 g(\tau)^b d\tau 
\Big\}^{\varrho_1/2} \\
\quad + C B( t)  \Big\{ 
e^{-\kappa_1\int_T^{t} g^{-b}(\tau)d\tau}  \norm{\bar P(T)}_{L^2}^2  +   \int_T^t e^{-\kappa_1\int_\tau^t g^{-b}(\theta)d\theta} \norm{\Psi(\tau)}_{L^\infty}^2 g(\tau)^b d\tau 
\Big\}^{1/2},
\end{multline*}
where $\varrho_1=\frac{\gamex_1}{{\gamex_1}+1}$.  
Note that $\varrho_1<1$, then by condition \eqref{condlim1}, we have 
\beq\label{stillg}
\begin{aligned}
  \limsup_{t\to\infty} \norm{\bar P(t)}_{L^\infty(U')}
&\le C \limsup_{t\to\infty}  \Big\{ B(t)^\frac2{\varrho_1} \int_T^t e^{-\kappa_1\int_\tau^t g^{-b}(\theta)d\theta} \norm{\Psi(\tau)}_{L^\infty}^2 g(\tau)^b d\tau 
\Big\}^{\varrho_1/2}\\
&\quad + C \limsup_{t\to\infty}\Big\{ B(t)^2   \int_T^t e^{-\kappa_1\int_\tau^t g^{-b}(\theta)d\theta} \norm{\Psi(\tau)}_{L^\infty}^2 g(\tau)^b d\tau\Big\}^{1/2}.
\end{aligned}
\eeq
The first condition in \eqref{condlim1} and the fact that $B(t)\ge 1$ imply $\int_2^\infty g(\tau)^{-b}d\tau=\infty$.
With this and the second condition in \eqref{condlim1}, we apply Lemma~\ref{difflem2} to each limit in \eqref{stillg} and obtain
\begin{align*}
\limsup_{t\to\infty} \norm{\bar P(t)}_{L^\infty(U')}
&\le C \limsup_{t\to\infty} \Big( B(t)^{\frac 2{\varrho_1}}\norm{\Psi(t)}_{L^\infty}^2 g(t)^{2b}\Big)^{\varrho_1/2} \\
&\quad + C \limsup_{t\to\infty} \Big( B(t)^2 \norm{\Psi(t)}_{L^\infty}^2 g(t)^{2b} \Big)^{1/2},
\end{align*}
thus, \eqref{limsup00} follows.
\end{proof}

Between two required estimates in \eqref{boundcond} of Lemma \ref{limPprop}, the second one needs more work. Therefore, we focus on  estimating $\tilde {\mathcal C}(t)$ defined by \eqref{Clam}, which contains $\tilde \vartheta(t)$ given by \eqref{uptildef}. For this one, we have from relation \eqref{Kestn} that
\begin{align}
\nonumber
\tilde \vartheta(t) &\le C+C\Big(\int_{t-1}^t \int_{\setV}  K(|\nabla \bar p_1|)|\nabla \bar p_1|^{\gamex_2+a} + K(|\nabla \bar p_2|)|\nabla \bar p_2|^{\gamex_2+a}dx dt\Big)^\frac1{2\mu}
\\
\label{pregam} 
&\le C+C\Big(\int_{0}^t \int_{\setV}  K(|\nabla \bar p_1|)|\nabla \bar p_1|^{\gamex_2+a} + K(|\nabla \bar p_2|)|\nabla \bar p_2|^{\gamex_2+a}dx dt\Big)^\frac1{2\mu}.
\end{align}
By \eqref{pregam} and \eqref{Kgrad55},
\begin{align}\label{newgamtil}
\tilde \vartheta(t)\le C L_{12} \Bfun_1(t)^\frac{\muex_4(\gamex_2+a-2)}{2\mu}\Bfun_2(t)^\frac1{2\mu},
\end{align}
where $L_{12}=\big\{\sum_{i=1,2} L_3\big(\gamex_2+a;[p_i(0)]\big)\big\}^\frac1{2\mu}$, 
\beqs
\Bfun_1(t)=1+M_F(t)\quad\text{and}\quad 
\Bfun_2(t)=1+\int_0^t F(\tau)d\tau.
\eeqs

\begin{theorem}\label{newP}
 \asdc. Suppose $\bar A<\infty$. 
Define 
\beqs
\bar\Upsilon_3 = 1+\sup_{[0,\infty)}(\|\psi_1\|_{L^\infty}+\|\psi_2\|_{L^\infty}),\quad
\omega(t) =1+ \int_{t-2}^t \tilde F(\tau)d\tau,\quad
\Bfun_3(t)=\omega(t)^\frac{\muex_7}{\muex_6} \Bfun_2(t)^\frac1{2\mu({\gamex_1}+1)}.
\eeqs
Let $d_4=d_2 d_3^{-b}$, where $d_2$ is in \eqref{psubt} and $d_3$ is in \eqref{lamtilde} below.
If 
\beq\label{3limcond}
\lim_{t\to\infty} \Bfun_3(t)^\frac{2(\gamex_1+1)}{\gamex_1} e^{-d_4 \bar\Upsilon_3^{-\frac{2b}{1-a}}\int_2^t \omega(\tau)^{-b}d\tau}=0,\quad
\lim_{t\to\infty} (\omega^b(t))'=\lim_{t\to\infty} \frac{\omega^b(t)  F(t)}{\int_0^t F(\tau)d\tau}=0 ,
\eeq
then
\beq\label{limsup70}
\limsup_{t\to\infty}  \|\bar P(t)\|_{L^\infty(U')} 
\le C \Upsilon_4 
\limsup_{t\to\infty} \Big\{ \Bfun_3(t) \Big(  \omega(t)^b\norm{\Psi(t)}_{L^\infty}  + \big[ \omega(t)^b\norm{\Psi(t)}_{L^\infty} \big]^\frac{\gamex_1}{\gamex_1+1}\Big) \Big\},
\eeq
where $\Upsilon_4$ is defined by \eqref{c1def} below.
\end{theorem}
\begin{proof}
Since $\bar A<\infty$, we have $\bar\Upsilon_3 <\infty$.
Then $$\bar A\le C (\bar\Upsilon_3^2+\bar\Upsilon_3^\frac{2-a}{1-a})\le C \bar\Upsilon_3^\frac{2-a}{1-a}<\infty .$$
By \eqref{boundedlarge}, there is $T_1>0$ such that for $\tau>T_1$, we have 
\beqs
\Lambda(\tau)\le C\Big(\bar\Upsilon_3^\frac2{1-a}+\int_{\tau-1}^\tau \tilde F(\tau)d\tau\Big),
\eeqs
hence, for $t>T_1+1$,
\beq\label{lamtilde}
  \sup_{\tau\in[t-1,t]}\Lambda(\tau)\le C\Big(\bar\Upsilon_3^\frac2{1-a}+\int_{t-2}^t \tilde F(\tau)d\tau\Big)
\le d_3 \bar\Upsilon_3^\frac2{1-a} \omega(t),\quad \tilde \lambda(t)\le C \bar\Upsilon_3^\frac{2\muex_7}{1-a} \omega(t)^{\muex_7},
\eeq
where $d_3$ is a positive constant.
Note also that $\norm{\Psi(t)}_{L^\infty}\le \bar\Upsilon_3$ for all $t\ge 0$.
Then by \eqref{Clam}, \eqref{lamtilde} and \eqref{newgamtil}, we have
\beq\label{Ctilbound}
 \tilde {\mathcal C}(t)
\le C  \bar\Upsilon_3^\frac{2\muex_7}{(1-a)\muex_6} \omega(t)^\frac{\muex_7}{\muex_6} \Big( 1 + L_{12}^\frac1{{\gamex_1}+1} \bar\Upsilon_3^\frac{\muex_4(\gamex_2+a-2)}{2\mu({\gamex_1}+1)} \Bfun_2(t)^\frac1{2\mu({\gamex_1}+1)}  + \bar\Upsilon_3^\frac1{\muex_6+1}\Big)\\
\eeq
Thus,
\beq\label{ctil2}
 \tilde {\mathcal C}(t)\le C \eta_1 \omega(t)^\frac{\muex_7}{\muex_6} \Bfun_2(t)^\frac1{2\mu({\gamex_1}+1)}=C \eta_1 \Bfun_3(t) ,
\eeq
where $\eta_1=L_{12}^\frac1{{\gamex_1}+1} \bar\Upsilon_3^\frac{2\muex_7}{(1-a)\muex_6}\big[  \bar\Upsilon_3^\frac{\muex_4(\gamex_2+a-2)}{2\mu({\gamex_1}+1)}   + \bar\Upsilon_3^\frac1{\muex_6+1}\big]$.
Using \eqref{lamtilde} and \eqref{ctil2}, we apply Lemma \ref{limPprop} with $g(t)=\bar\Upsilon_3^\frac{2}{1-a}\omega(t)$
and $B(t) =  \eta_1  \Bfun_3(t)$. Note that the last two limits in \eqref{3limcond} imply the second limit in \eqref{condlim1}.
As a result, we obtain from \eqref{limsup00} that
\begin{align*}
\limsup_{t\to\infty}  \|\bar P(\cdot,t)\|_{L^\infty}
&\le C  \limsup_{t\to\infty}\Big\{ \eta_1 \Bfun_3(t)
\Big[ \Big(  [\bar\Upsilon_3^\frac2{1-a}\omega(t)]^{b} \norm{\Psi(t)}_{L^\infty} \Big)^{\varrho_1} +  [\bar\Upsilon_3^\frac2{1-a}\omega(t)]^{b} \norm{\Psi(t)}_{L^\infty} \Big]
\Big\} .
\end{align*}
Hence \eqref{limsup70} follows  with 
\beq\label{c1def}
\Upsilon_4=  \bar\Upsilon_3^\frac{2b}{1-a} \eta_1=L_{12}^\frac1{{\gamex_1}+1} \bar\Upsilon_3^{\frac{2\muex_7}{(1-a)\muex_6}+\frac{2b}{1-a}}\big[  \bar\Upsilon_3^\frac{\muex_4(\gamex_2+a-2)}{2\mu({\gamex_1}+1)}   + \bar\Upsilon_3^\frac1{\muex_6+1}\big].
\eeq
The proof is complete.
\end{proof}

Next, we will treat the case $\bar A=\infty$. For that we recall some estimates from \cite{HI2}.

\begin{lemma}[cf. \cite{HI2}, Lemma 5.8]\label{lamda-pt-estimate}
Assume the Degree Condition and $\bar A=\infty$. Define
\begin{align*}
W_1(t) &= 1+M_\myB(t)^\frac2{2-a}+\int_{t-1}^t \tilde \myB(\tau)d\tau,\\
W_2(t) &=1+\bar \beta^\frac1{1-a}+\myB(t-1)^\frac2{2-a} +\myB(t)+\int_{t-1}^t \myB(\tau)+\tilde \myB(\tau)d\tau.
\end{align*}

{\rm (i)} There is $T_1>0$ such that   $\Lambda(t)\le CW_1(t)$ for all $t>T_1$.

{\rm (ii)} If $\bar \beta<\infty$ then there is $T_2>0$ such that $\Lambda(t)\le W_2(t)$ for all $t>T_2$.
\end{lemma}

Let $W(t)$ be defined, in the general case, by 
\beqs
W(t)= 1+M_\myB(t)^\frac2{2-a}+\int_{t-2}^t \tilde \myB(\tau)d\tau,
\eeqs 
and, in case $\bar \beta<\infty$, by
\beqs
W(t)=1+\bar \beta^\frac1{1-a}+\sup_{[t-2,t]}\myB(t)^\frac2{2-a} +\int_{t-2}^t \tilde \myB(\tau)d\tau.
\eeqs
Then for large $t$, we have from Lemma \ref{lamda-pt-estimate} that
\beq\label{lamW}
\sup_{\tau\in[t-1,t]}\Lambda(\tau)\le d_5 W(t) 
\quad\text{and}\quad
\tilde \lambda(t)\le CW(t)^{\muex_7},
\eeq
where $d_5$ is a positive constant.
With \eqref{lamW}, we restate Lemma \ref{limPprop} as the following.

\begin{lemma}\label{propAbar}
 \asdc. Suppose $\bar A=\infty$ and $\tilde {\mathcal C}(t)\le C \tilde B(t)$ for sufficient large $t$, with some function $\tilde B(t)\ge 1$.  
Let $d_6=d_2 d_5^{-b}$, where $d_2$ is in \eqref{psubt} and $d_5$ is in \eqref{lamW}.
If
\beq\label{condlim15}
\lim_{t\to\infty} \tilde B(t)^\frac{2({\gamex_1}+1)}{{\gamex_1}} e^{-d_6\int_{2}^t W(\tau)^{-b}d\tau}  =0
\quad\text{and}\quad
\lim_{t\to\infty}  W(t)^b\frac {\tilde B'(t)}{\tilde B(t)}  =0,
\eeq
then
\beq\label{limsup50}
\begin{aligned}
\limsup_{t\to\infty}  \|\bar P(t)\|_{L^\infty(U')}
&\le C   \limsup_{t\to\infty}\Big\{  \tilde B(t) \Big(  W(t)^{b}  \norm{\Psi(t)}_{L^\infty}
+ \big[ W(t)^{b} \norm{\Psi(t)}_{L^\infty}\big]^\frac{\gamex_1}{{\gamex_1}+1}\Big)\Big\} .
\end{aligned}
\eeq
\end{lemma}

Denote
\beq
\muex_{13}=\max\Big\{ \frac{\muex_7}{\muex_6} ,\frac{\muex_7}{\muex_6+1}+\frac{1-a}{2(\muex_6+1)}\Big \}\quad \text{and}\quad
\gamex_7=\frac{\muex_4(\gamex_2+a-2)}{2\mu}.
\eeq

\begin{theorem}\label{newthmAbar2}
  \asdc. Suppose $\bar A=\infty$ and $\int_0^\infty F(\tau)d\tau=\infty$.
Define
\beqs
\Bfun_4(t)=\int_0^t F(\tau)d\tau
\quad\text{and}\quad
 \Bfun_5(t)=M_F(t)^{\gamex_7} \Bfun_4(t)^\frac1{2\mu}.
 \eeqs
Let $d_6$ be defined as in Lemma \ref{propAbar}.
If 
\beq\label{newexpCond}
 \lim_{t\to\infty}  W(t)^{\frac {2\muex_{13}({\gamex_1}+1)}{{\gamex_1}}} \Bfun_5(t)^{\frac 2{{\gamex_1}}} e^{-d_6\int_2^{t} W^{-b}(\tau)d\tau}=0,
\eeq
\beq\label{newlemCond}
\lim_{t\to\infty}  (W^b(t))' =0,\quad \lim_{t\to\infty} W^b(t)\frac{M_F'(t)} {M_F(t)} =0 \quad\text{and}\quad  \lim_{t\to\infty} W^b(t)\frac{F(t)} {\int_0^t F(\tau)d\tau} =0,
\eeq
then
\beq\label{newlimsupP}
\begin{aligned}
&\limsup_{t\to\infty}  \|\bar P(t)\|_{L^\infty(U')}\\
&\le   C  \limsup_{t\to\infty} \Big\{ W(t)^{\muex_{13}}  \Bfun_5(t)^\frac{1}{{\gamex_1}+1} \Big(   W(t)^{b} \norm{\Psi(t)}_{L^\infty}
+ \big[ W(t)^{b}  \norm{\Psi(t)}_{L^\infty} \big]^\frac{{\gamex_1}}{{\gamex_1}+1}\Big) \Big\}.
\end{aligned}
\eeq
\end{theorem}
\begin{proof}
Let 
$\Bfun_6(t)=1+\sup_{[0,t]}(\|\psi_1\|_{L^\infty}+\|\psi_2\|_{L^\infty}).$
Then $\lim_{t\to\infty} \Bfun_6(t)=\infty$ and $\lim_{t\to\infty} \Bfun_4(t)=\infty$.
Note that $\Bfun_6(t)\le (1+M_F(t))^\frac{1-a}{2-a}$, and in both cases of definition of $W(t)$, we have
\beq\label{maxpsi}
\sup_{[t-1,t]}\|\Psi\|_{L^\infty}\le C W(t)^{\frac{1-a}{2-a}\frac{2-a}{2}}=CW(t)^\frac{1-a}2.
\eeq
We estimate $\tilde {\mathcal C}(t)$. By \eqref{pregam} and \eqref{gpr3}, we have for $t$ sufficiently large that
\begin{align}\label{newgamtil2}
 \tilde \vartheta(t) &\le C M_F(t)^{\frac{\muex_4(\gamex_2+a-2)}{2\mu}}\Bfun_4(t)^\frac1{2\mu}
= C  \Bfun_5(t).
\end{align}
By \eqref{Clam}, \eqref{lamW}, \eqref{newgamtil2} and \eqref{maxpsi},
\begin{align*}
\tilde  {\mathcal C}(t) 
&\le C  \Big\{ W(t)^\frac{\muex_7}{\muex_6} + \Bfun_5(t)^\frac1{{\gamex_1}+1} W(t)^\frac{\muex_7}{{\gamex_1}+1} + W(t)^\frac{\muex_7}{\muex_6+1} W(t)^\frac{1-a}{2(\muex_6+1)}\Big\}
\le C \Bfun_7(t),
\end{align*}
where $\Bfun_7(t)= W(t)^{\muex_{13}}\Bfun_5(t)^\frac{1}{{\gamex_1}+1}$.
We apply Lemma \ref{propAbar} with $\tilde B(t)=\Bfun_7(t)$.
The first condition in \eqref{condlim15} is replaced by
\beqs
 \lim_{t\to\infty}  ( W(t)^{\muex_{13}}\Bfun_5(t)^\frac{1}{{\gamex_1}+1} )^{\frac 2{\varrho_1}} e^{-d_6\int_2^{t} W^{-b}(\tau)d\tau}=0,
\eeqs
which is \eqref{newexpCond}. The second condition in \eqref{condlim15} is replaced by \eqref{newlemCond}.
Then we obtain \eqref{newlimsupP}  directly from \eqref{limsup50}.
\end{proof}

\begin{example}
Suppose that for $t$ sufficiently large, we have
\beqs
F(t)\le M_F(t)\le Ct^{\egexp_1},\quad W(t)\le Ct^{\egexp_2/b},
\eeqs
for some $\egexp_1>0$ and $0<\egexp_2<1$. 
Following the proof of Theorem \ref{newthmAbar2}, we can see that the statements still hold true if the functions $F(t)$, $M_F(t), W(t)$ are replaced by their upper bounds $Ct^{\egexp_1}$, $Ct^{\egexp_1}$, $Ct^{\egexp_2/b}$, respectively.
With such replacements, conditions \eqref{newexpCond} and \eqref{newlemCond} are met, and $\Bfun_5(t) =C t^{\egexp_3}$, where $\egexp_3=\frac{\muex_4(\gamex_2+a-2)\egexp_1+\egexp_1+1}{2\mu}$. 
Therefore, from \eqref{newlimsupP}, it follows
\beqs
\limsup_{t\to\infty}  \|\bar P(\cdot,t)\|_{L^\infty(U')}
\le C   \limsup_{t\to\infty} \Big( t^{\egexp_5} \norm{\Psi(t)}_{L^\infty} +  [t^{\egexp_4 } \norm{\Psi(t)}_{L^\infty}]^\frac{\gamex_1}{{\gamex_1}+1}\Big), 
\eeqs
where
$\egexp_4=\frac{\muex_4(\gamex_2+a-2)\egexp_1+\egexp_1+1}{2\mu{\gamex_1}}+\frac{\egexp_2\muex_{13}({\gamex_1}+1) }{b{\gamex_1}} +\egexp_2 $
and 
$\egexp_5= \frac{\muex_4(\gamex_2+a-2)\egexp_1+\egexp_1+1}{2\mu({\gamex_1}+1)}+\frac{\egexp_2\muex_{13} }b +\egexp_2 $.
\end{example}


\subsection{Results for pressure gradient}
\label{sec42}

In \cite{HI2}, the norm $\norm{ \nabla P(t) }_{L^{2-a}}$ is estimated for all $t>0$. Now we estimate
$\norm{\nabla P(t)}_{L^{s}(U')}$ and $\|\nabla P\|_{L^{s}(U'\times(0,T))}$ for any $s\in(2-a,2)$.

\begin{proposition}\label{PropGrad}
Let $\delta\in (0,a)$. 
Then for all $t>0$,
\begin{multline}\label{DG1}
\|\nabla P(t)\|_{L^{2-\delta}(U')}
\le C \|\bar P(t)\|_{L^2}^\frac12 \Big\{1+\sum_{i=1,2}\Big[ \| \bar p_{it}(t)\|_{L^2}^2 + \| \nabla p_i(t)\|_{L^{2-a}}^{2-a} 
+\|\psi_i(t)\|_{L^\infty}^2 \Big]\Big\}^\frac14\\
\cdot  \Big( 1 + \sum_{i=1,2} \int_{U'} |\nabla p_i(x,t)|^\frac{a(2-\delta)}{\delta}dx \Big)^\frac{\delta}{2(2-\delta)},
\end{multline}
and for any $T> 0$, 
\begin{multline}\label{InGradP}
\|\nabla P\|_{L^{2-\delta}(U'\times(0,T))}
\le C \|\bar P\|_{L^2(U\times(0,T))}^{1/2} \\
\cdot \Big[T+ \sum_{i=1,2} \Big(\norm{\bar p_{it}}_{L^2(U\times(0,T))}^2  +\int_0^T \int_U |\nabla p_{i}|^{2-a}dxdt+\int_0^T \norm{\psi_i(t)}_{L^\infty}^2dt\Big)  \Big]^{\frac {1}4}\\     
\cdot\Big[ T+\sum_{i=1,2}\int_{0}^T\int_{U'} |\nabla p_i|^\frac{a(2-\delta)}{\delta}dx dt\Big]^\frac{\delta}{2(2-\delta)}.
\end{multline}
Here, constant $C>0$ depends on $U$, $U'$ and $\delta$.
\end{proposition}
\begin{proof}
Note that $\nabla \bar P =\nabla P$.
Let $\zeta=\zeta(x)$ be a cut-off function such that $\zeta$ vanishes in neighborhood of $\Gamma$.
Multiplying equation \eqref{eq-1} by $\bar P\zeta^2$ and integrating over $U$, using integration by parts,  we have
\begin{align*}
\int_U \bar  P_t \bar P \zeta^2 dx
&= -\int_U \Big(K(\nabla p_1|)\nabla p_1-K(|\nabla p_2|)\nabla p_2 \Big ) \cdot (\nabla  P \zeta^2+2\bar P \zeta\nabla \zeta) dx \\
&\quad  +\frac1{|U|}\int_\Gamma \Psi(x,t)d\sigma \int_U\bar  P\zeta^2 dx.
\end{align*}
Let $\xi(x,t)=|\nabla p_1|\vee |\nabla p_2|$. By the monotonicity \eqref{mono1} in Lemma \ref{quasimono-lem}, we obtain
\beq\label{pregradP}
\begin{aligned}
\int_U \bar  P_t \bar P \zeta^2 dx
&\le -(1-a)\int_U K(\xi)|\nabla P\zeta|^2  dx  + 2\int_U ( |\nabla p_1|+ |\nabla p_2|)^{1-a}|P| \zeta|\nabla\zeta|  dx\\
&\quad + \norm{\Psi(t)}_{L^\infty}\int_U |\bar  P| \zeta^2 dx.
\end{aligned}
\eeq
Let $\setV\Subset U$ such that $U'\Subset \setV$. We select $\zeta$ such that $\zeta\equiv 1$ on $U'$ and supp $\zeta\subset \setV$.
We obtain from \eqref{pregradP} that  
\begin{align*}
&(1-a)\int_U K(\xi)|\nabla P\zeta|^2  dx 
\le \int_{U_1} |\bar  P_t| |\bar P| dx  +C\int_{\setV} ( |\nabla p_1|+ |\nabla p_2|)^{1-a}|\bar P|dx+ C \norm{\Psi(t)}_{L^\infty} \norm{\bar P}_{L^2}\\
&\le C (\sum_{i=1,2} \norm{\bar p_{it}}_{L^2} )\norm{\bar P}_{L^2}+C\Big(\sum_{i=1,2} \int_U  |\nabla p_i|^{2-2a}dx\Big)^{1/2} \norm{\bar P}_{L^2}  +C \sum_{i=1,2} \norm{\psi_i(t)}_{L^\infty} \norm{\bar P}_{L^2}.
\end{align*}
Hence 
\beq\label{sub1}
\int_{U'} K(\xi)|\nabla P|^2  dx 
\le C \norm{\bar P}_{L^2}  \Big[ \sum_{i=1,2} \norm{\bar p_{it}}_{L^2} +\Big(\int_U (1+\sum_{i=1,2}|\nabla p_i|)^{2-a} dx\Big)^{1/2}
+\sum_{i=1,2}\norm{\psi_i(t)}_{L^\infty}\Big].
\eeq
By H\"older's inequality and property \eqref{Kesta}, we have
\beq\label{DG2}
\int_{U'} |\nabla P|^{2-\delta}  dx 
\le C \Big ( \int_{U'} K(\xi)|\nabla P|^2  dx \Big )^\frac{2-\delta}{2} \Big( \int_{U'} (1+|\nabla p_1|+|\nabla p_2|)^\frac{a(2-\delta)}{\delta}dx \Big)^\frac{\delta}{2}.
\eeq
Combining \eqref{DG2} with \eqref{sub1} yields 
\begin{multline*}
\| \nabla P(t)\|_{L^{2-\delta}(U')} \le C \|\bar P(t)\|_{L^2}^\frac12 \Big\{1 + \sum_{i=1,2}\Big[ \| \bar p_{it}(t)\|_{L^2}^2 + \int_U |\nabla p_i(x,t)|^{2-a} dx +\|\psi_i(t)\|_{L^\infty}^2 \Big]  \Big\}^\frac14\\
\cdot  \Big( \int_{U'} (1+|\nabla p_1(x,t)|+|\nabla p_2(x,t)|)^\frac{a(2-\delta)}{\delta}dx \Big)^\frac{\delta}{2(2-\delta)},
\end{multline*}
then we obtain \eqref{DG1}.
We now prove \eqref{InGradP}.
Integrating \eqref{sub1}  in $t$ over $[0,T]$ and applying H\"older's inequality yield
\begin{align*}
&\int_0^T \int_{U'} K(\xi)|\nabla P|^2  dx 
\le C \Big( \int_0^T \norm{\bar P(t)}_{L^2}^2 dt \Big)^{1/2} 
\cdot\Big( \int_0^T  \norm{\bar p_{1t}(t)}_{L^2}^2 + \norm{\bar p_{2t}(t)}_{L^2}^2 dt \\
&\quad  + \int_0^T \int_U (|\nabla p_1|+|\nabla p_2|)^{2-2a} dx dt +\int_0^T \norm{\Psi(t)}_{L^\infty}^2dt \Big)^{1/2}.
\end{align*}
Then
\beq\label{sub10}
\begin{aligned}
&\int_0^T \int_{U'} K(\xi)|\nabla P|^2  dx 
\le C \|\bar P\|_{L^2(U\times(0,T))} \\
&\cdot \Big[ \sum_{i=1,2} \Big(\norm{\bar p_{it}}_{L^2(U\times(0,T))}^2  +\int_0^T \int_U (1+|\nabla p_{i}|^{2-a})dxdt+\int_0^T \norm{\psi_i(t)}_{L^\infty}^2dt\Big)  \Big]^{\frac 12}.
\end{aligned}
\eeq
Again, by H\"older's inequality and property \eqref{Kesta}, we have
\beq\label{GradPclose2}
\int_{0}^T \int_{U'} |\nabla P|^{2-\delta}  dx dt
\le C \Big ( \int_{0}^T \int_{U'} K(\xi)|\nabla P|^2  dxdt \Big )^\frac{2-\delta}{2} \Big( \int_{0}^T\int_{U'} (1+|\nabla p_1|+|\nabla p_2|)^\frac{a(2-\delta)}{\delta}dx dt\Big)^\frac{\delta}{2}.
\eeq
Using \eqref{sub10} in \eqref{GradPclose2} and taking the power $1/(2-\delta)$, we obtain \eqref{InGradP}. 
\end{proof}

We now have explicit estimates in terms of initial and boundary data.

\begin{theorem}\label{GradThm1}
 \asdc.
 For $\delta\in(0,a)$, $0<t_0<1$ and $T>t_0$ we have
\beq\label{DG5}
\sup_{[t_0,T]} \|\nabla P(t)\|_{L^{2-\delta}(U')} \le C M_{3,t_0,T}(\|\bar P(0)\|_{L^2} +\sup_{[0,T]}\|\Psi\|_{L^\infty})^\frac12,
\eeq
where $M_{3,t_0,T}$ is defined in \eqref{N8def} below.
\end{theorem}
\begin{proof}
We use estimate \eqref{DG1}. We bound $\|\bar P(t)\|_{L^2}$ by \eqref{pLip}:
\begin{equation}\label{pLip2}
\sup_{[0,T]}\norm{\bar P(t)}_{L^2}^2  \le C\cdot  M_{1,T}\Big(\norm{\bar P(0)}_{L^2}^2 
+ \sup_{[0,T]}\norm{\Psi(t)}_{L^\infty}^2 \Big),
\end{equation}
where $M_{1,T}=A_0+T+\int_0^T\big[ f_1(\tau)+f_2(\tau)\big] d\tau$.
We estimate time derivative by \eqref{Jpt-boundA0}, estimate $\int_U |\nabla p_i|^{2-a} dx$ by \eqref{H-bound-0}, then we have
\beq\label{pit}
\begin{aligned}
& \sum_{i=1,2}\Big[ \| \bar p_{it}(t)\|_{L^2}^2 + \int_U |\nabla p_i(x,t)|^{2-a} dx 
+\|\psi_i(t)\|_{L^\infty}^2 \Big]\\
&\le C\, t_0^{-1}\Big(\sum_{i=1,2} L_5(t_0;[p_i(0),\psi_i]) \Big) + (T+1)\bar\Nf_{1,T}^\frac{2-a}{1-a}+\bar\Nf_{2,T}\Big).
\end{aligned}
\eeq
Recalling $\nu_1=\max\{2,\frac{a(2-\delta)}{\delta}\}$, we apply \eqref{grad20pw} with $s=\nuex_1$ and obtain
\begin{multline}\label{gradalot}
 \sum_{i=1,2} \int_{U'} |\nabla p_i(x,t)|^\frac{a(2-\delta)}{\delta}dx 
\le \sum_{i=1,2} \int_{U'} 1+ |\nabla p_i(x,t)|^{\nu_1}dx \\
 \le C\Big[\sum_{i=1,2} L_2(\nuex_1;[p_i(0)])\Big]  (T+1)^{\frac{2(\nuex_1-2)}{2-a}+1} \bar\Nf_{1,T}^{\frac{\nuex_1-a}{1-a}}.
\end{multline}
Combining \eqref{pLip2}, \eqref{pit} and \eqref{gradalot} with \eqref{DG1}, we have \eqref{DG5} with
\beq\label{N8def}
 M_{3,t_0,T}
=(1+t_0^{-\frac14}) M_{1,T}^\frac14 \Big[\sum_{i=1,2} L_5(t_0;[p_i(0),\psi_i]) \Big]^\frac14
\Big[\sum_{i=1,2} L_2\big(\nuex_1;[p_i(0)]\big)\Big]^\frac{\delta}{2(2-\delta)}
(T+1)^{\gamex_8} 
\bar\Nf_{1,T}^{\gamex_9}
\bar\Nf_{2,T}^\frac14,
\eeq
where
$\gamex_8= \frac14+\frac{\delta}{2(2-\delta)}\big(\frac{2(\nuex_1-2)}{2-a}+1\big)$ and 
$\gamex_9= \frac{2-a}{4(1-a)}+\frac{\delta}{2(2-\delta)} \big(\frac{\nuex_1-a}{1-a}\big).$
The proof is complete.
\end{proof}

As $t\to\infty$, we have the following asymptotic estimate.

\begin{theorem}\label{GradThm2}
 \asdc. Let 
\beqs\bar\Upsilon_1 = 1+\sup_{[0,\infty)}F,\  \bar\Upsilon_2 = 1+\int_0^\infty F(t)dt
\quad \text{and} \quad
\bar\Ulim_2 = 1+\bar A^\frac2{2-a}+\limsup_{t\to\infty} \int_{t-1}^t \tilde F(\tau)d\tau.
\eeqs
If $\bar\Upsilon_1$, $\bar\Upsilon_2$ and $\bar\Ulim_2$ are finite then 
\beq\label{DG6}
\limsup_{t\to\infty} \|\nabla P(t)\|_{L^{2-\delta}(U')} \le C \Upsilon_5 \limsup_{t\to\infty}\|\Psi\|_{L^\infty}^{1/2}, 
\eeq
where $\Upsilon_5$ is defined by \eqref{U7} below.
\end{theorem}
\begin{proof}
 Taking limsup of \eqref{DG1} and making use the limits \eqref{psub-bounded}, \eqref{limsupH1} and estimate \eqref{Kgrad60pw} give
\begin{multline*}
 \limsup_{t\to\infty} \|\nabla P(t)\|_{L^{2-\delta}(U')} \le C \limsup_{t\to\infty}\|\Psi\|_{L^\infty}^{1/2}
\bar\Ulim_2^\frac b 2 \bar\Ulim_2^\frac14 \Big\{  \sum_{i=1,2} L_4(\nuex_1;[p_i(0)]) \bar\Upsilon_1^{\muex_4(\nuex_1-2)} \bar\Upsilon_2 \Big\}^\frac\delta{2(2-\delta)}.
\end{multline*}
Therefore, we obtain \eqref{DG6} with 
\beq\label{U7}
\Upsilon_5= \bar\Ulim_2^{\frac b 2 +\frac14}
\Big\{ \bar\Upsilon_1^{\muex_4(\nuex_1-2)}\bar\Upsilon_2  \sum_{i=1,2} L_4(\nuex_1;[p_i(0)])   \Big\}^\frac\delta{2(2-\delta)}.
\eeq
\end{proof}

\begin{remark}\label{rmk414}
Even though the limit estimate in \eqref{DG6} still depends on the initial data presented in $\Upsilon_5$, the smallness of the estimate can be controlled by the difference $\Psi(t)$ for large $t$. 
\end{remark}

The estimate in Theorem \ref{GradThm1} blows up when $t_0\to 0$.
To overcome this, we consider the Lebesgue norm in both $x$ and $t$.

\begin{theorem}\label{lem412}
 \asdc. Let $\delta\in(0,a)$. For any $ T> 0$, we have
 \beq\label{InGradPt0}
\|\nabla P\|_{L^{2-\delta}(U'\times(0,T))} \le C M_{4,T} \Big( \norm{\bar P(0)}_{L^2}^2+\int_0^T \norm{\Psi(t)}_{L^\infty}^2 dt \Big)^{\frac {1} 4},       
\eeq 
where $ M_{4,T}$ is defined in \eqref{hatN} below.
\end{theorem}
\begin{proof}
Applying  \eqref{grad39} to $s=\nuex_1\ge 2$ we have 
\begin{multline}\label{subs0}
T+\int_0^T\int_{U'} (1+|\nabla p_1|+|\nabla p_2|)^\frac{a(2-\delta)}{\delta}dx dt 
\le T+\int_0^T\int_{U'} (1+|\nabla p_1|+|\nabla p_2|)^{\nu_1}dx dt \\
\le C \Big[ \sum_{i=1,2} L_1(\nuex_1+a,[p_i(0)]) \Big]
(T+1)^{\frac {2\nuex_1}{2-a} - 1} \bar\Nf_{1,T}^{\frac{\nuex_1}{1-a}}.
\end{multline}
Inequality \eqref{H-bound-0} provides 
 \beqs
 \sum_{i=1}^2 \norm{\bar p_{it}}_{L^2(U\times(0,T))}^2 \le C(1+A_0+B_0) (T+1) \bar\Nf_{1,T}^\frac{2-a}{1-a} \bar\Nf_{2,T}.
 \eeqs
 Using  \eqref{p-bar-bound1} we see that
 \beqs
 T+\sum_{i=1}^2\int_0^T\int_U |\nabla p_i|^{2-a} dxdt \le C(1+A_0) (T+1) \bar\Nf_{1,T}^\frac{2-a}{1-a}.
 \eeqs
 Hence, 
 \beq\label{subs1}
T+ \sum_{i=1}^2\Big\{ \norm{\bar p_{it}}_{L^2(U\times(0,T))}^2+\int_0^T\int_U |\nabla p_i(x,t)|^{2-2a} dx dt +\int_0^T \norm{\psi_i(t)}_{L^\infty}^2 dt \Big\}
 \le C N_{1,T},
 \eeq
where $N_{1,T}=(1+A_0+B_0) (T+1) \bar\Nf_{1,T}^\frac{2-a}{1-a} \bar\Nf_{2,T}.$
According to \eqref{pcont} and \eqref{H-bound-0},
\beq\label{subs2}
\begin{aligned}
 \int_0^T  \norm{\bar P(t)}_{L^2}^2  dt
& \le  C T  \sup_{t\in [0,T]} \Lambda(t)^{b} \Big( \norm{\bar P(0)}_{L^2}^2 + \int_0^T \norm{\Psi(t)}_{L^\infty}^2  dt \Big)\\
&\le C T N_{1,T}^b \Big( \norm{\bar P(0)}_{L^2}^2 +\int_0^T \norm{\Psi(t)}_{L^\infty}^2  dt \Big).
\end{aligned}
\eeq
Combining \eqref{InGradP}, \eqref{subs0}, \eqref{subs1} and \eqref{subs2} we obtain 
 \begin{align*}
 &\|\nabla P\|_{L^{2-\delta}(U'\times(0,T))}
 \le C T^\frac{1}4 N_{1,T}^\frac{b}{4} \Big( \norm{\bar P(0)}_{L^2}^2 +\int_0^T \norm{\Psi(t)}_{L^\infty}^2  dt \Big)^\frac14  \cdot   N_{1,T}^\frac14  \\
&\quad   \cdot \Big\{\Big[ \sum_{i=1,2} L_1(\nuex_1+a,[p_i(0)]) \Big]
(T+1)^{\frac {2\nuex_1}{2-a} - 1} \bar\Nf_{1,T}^{\frac{\nuex_1}{1-a}}\Big\}^\frac{\delta}{2(2-\delta)}. 
 \end{align*}
Therefore, we obtain \eqref{InGradPt0} with
\beq\label{hatN}
M_{4,T} = (1+A_0+B_0)^\frac{b+1}4 \Big[ \sum_{i=1,2} L_1(\nuex_1+a,[p_i(0)]) \Big]^\frac{\delta}{2(2-\delta)}
 T^\frac14 (T+1)^{\gamex_{10}} \bar\Nf_{1,T}^{\gamex_{11}} \bar\Nf_{2,T}^\frac{b+1}4,
\eeq
where
$\gamex_{10}=\frac{\delta}{2(2-\delta)}\Big(\frac {2\nuex_1}{2-a} - 1\Big)+\frac{b+1}4$ and
$\gamex_{11}=\frac{\nuex_1}{1-a} \frac{\delta}{2(2-\delta)}+\frac{(b+1)(2-a)}{4(1-a)}$.
\end{proof}


\myclearpage
\section{Dependence on the Forchheimer polynomial}\label{polycont}

In this section we study the dependence of solutions of  IBVP \eqref{eqgamma} on the coefficients of the Forchheimer polynomial $g(s)$ in \eqref{gsa}.
Let $N\ge 1$, the exponent vector $\vec \alpha=(0,\alpha_1,\ldots,\alpha_N)$ and the boundary data $\psi(x,t)$  be fixed. 

Let ${\bf D}$ be a compact subset of $\{\vec a=(a_0,a_1,\ldots,a_N):a_0,a_N>0, a_1,\ldots,a_{N-1}\ge 0\}$.
Set 
$\hat \chi({\bf D})=\max\{\chi(\vec a):\vec a\in {\bf D}\}.$
Then $\hat \chi ({\bf D})$ is a number in $[1,\infty)$.

Let $g_1(s)=g(s,\vec a^{(1)})$ and $g_2(s)=g(s,\vec a^{(2)})$ be two functions of class FP($N,\vec \alpha$), where  $\vec a^{(1)}$ and $\vec a^{(2)}$ belong to  ${\bf D}$. For $k=1,2$, let $p_k=p_k(x,t;\vec a^{(k)})$ be the solution of \eqref{eqorig} and \eqref{BC} with $K=K(\xi,\vec{a})$ and the same boundary flux $\psi$.

Let $P=p_1-p_2$ and $\bar P=P-|U|^{-1}\int_U P dx$. 
Then
\beq\label{Erreq}
\bar P_t = \nabla \cdot (K (|\nabla \bar p_1|,\vec{a}^{(1)})\nabla \bar p_1 ) - \nabla \cdot (K (|\nabla \bar p_2|,\vec{a}^{(2)})\nabla \bar p_2 ) \quad \text {in }  U\times (0,\infty)
\eeq

As shown in \cite{HI1}, all constants $d_j$, $c_j$, $C_j$ and $C$ appearing in estimates in the previous sections when $\vec a$ varies among the vectors $\vec a^{(1)}$, $\vec a^{(2)}$, $\vec a^{(1)}\vee \vec a^{(2)}$ and $\vec a^{(1)}\wedge \vec a^{(2)}$, can be made dependent of $\hat \chi({\bf D})$, but independent of $\vec a$. 
We still denote them by $d_j$, $c_j$, $C_j$ and $C$, respectively, in this section.

Throughout this section, we continue to have $U'\Subset U$.
The flux-related quantities $f(t)$, $\tilde f(t)$, $M_f(t)$ $A$ and $\beta$ are defined in section \ref{supestimate}, from \eqref{b}
to \eqref{Abeta}. 
Also, $\Nf_{1,t}$ and $\Nf_{2,t}$ are defined in \eqref{Nf1} and \eqref{Nf2}, respectively.
Regarding initial data, $A_0$ is defined in \eqref{AB0}.
\subsection{Results for pressure}
\label{sec51}

We will estimate the interior $L^\infty$-norm for $\bar P(x,t)$ in terms of $\bar P(x,0)$ and $|\vec a^{(1)}-\vec a^{(2)}|$.
We start with the following general estimates which are counter parts of Propositions \ref{cont:strictcond} and \ref{newcont}.
\begin{proposition}\label{theo49}
We use the same assumptions and notation  as in Proposition \ref{cont:strictcond}. 

{\rm (i)} There exists a positive constant $C=C(U,\setV,U')$ such that 
\beq\label{Pcoe}
\sup_{[T_0+\theta T,T_0+T]}\|\bar P\|_{L^\infty(U')}\le C\widehat{\mathcal C}_{T_0,T,\theta}\Big(  \norm{\bar P}_{L^2(U\times (T_0,T_0+T) )}+\norm{\bar P}_{L^2(U\times (T_0,T_0+T) )}^{\frac {\gamex_1}{{\gamex_1}+1}}\Big),  
\eeq
where  
 \beq\label{Ctheta5}
 \widehat{\mathcal C}_{T_0,T,\theta}=  [\lambda(1+ (\theta T)^{-\frac1 2})]^{\frac 1\muex_6} + [\lambda  (T^{1/2}\vartheta_1+\vartheta_2)]^{\frac 1 {{\gamex_1}+1}},
\eeq
with $\lambda=\lambda_{T_0,T}$ defined in \eqref{lamdef2}, and  
 \begin{align}
\label{gam1}
 \vartheta_1&=\vartheta_{1,T_0,T} \eqdef   \Big(\int_{T_0}^{T_0+T}  \int_{\setV} |\nabla \bar p_1|^{2(1-a)\mu}+|\nabla \bar p_2|^{2(1-a)\mu} dxdt\Big)^\frac1{2\mu},\\
\label{gam2} 
\vartheta_2 &=\vartheta_{2,T_0,T} \eqdef  \Big(\int_{T_0}^{T_0+T} \int_{\setV} |\nabla \bar p_1|^{(2-a)\mu}+|\nabla \bar p_2|^{(2-a)\mu} dxdt\Big)^\frac1{2\mu}.
 \end{align}
 
{\rm (ii)} There exists a positive constant $C=C(U,\setV,U')$ such that 
\beq\label{Pcoe2}
\sup_{[0,T]}\|\bar P\|_{L^\infty(U')}\le 2\|\bar P(0)\|_{L^\infty} + C \widehat{\mathcal C}_{T}\Big(  \norm{\bar P}_{L^2(U\times (0,T) )}+\norm{\bar P}_{L^2(U\times (0,T) )}^{\frac {\gamex_1}{{\gamex_1}+1}}\Big)  ,
\eeq
where  
 \beq\label{Ctheta6}
 \widehat{\mathcal C}_{T}=  \lambda_T^{\frac 1\muex_6} + [\lambda_T  (T^{1/2}\vartheta_{1,T}+\vartheta_{2,T})]^{\frac 1 {{\gamex_1}+1}},
\eeq
with $\lambda_T=\lambda_{0,T}$, $\vartheta_{1,T}=\vartheta_{1,0,T}$ and  $\vartheta_{2,T}=\vartheta_{2,0,T}$.
 \end{proposition}
\begin{proof}
(i) Let $\zeta(x,t)$ be a cut-off function such that it is piecewise linear continuous in $t$, $\zeta(\cdot,0)=0$ and for each $t$, supp $\zeta(\cdot,t)\Subset \setV$.
Let $\bar P^{(k)}=\max\{\bar P-k,0\}$ and $\chi_k$ be characteristic function of the set $\{ (x,t): \bar P^{(k)}>0\}$.

Let $\xi=|\nabla \bar p_1|\vee |\nabla \bar p_2|$. 
 Multiplying equation \eqref{Erreq}  by $\bar P^{(k)} \zeta^2$ and integrating over $U$, we have
\begin{align*} 
&\frac12\ddt \int_U |\bar P^{(k)}\zeta|^2  dx- \int_U |\bar P^{(k)}|^2\zeta \zeta_t dx\\
&= \int_U ( K(|\nabla p_1|, \vec{a}^{(1)})\nabla \bar p_1 - K(|\nabla p_2|, \vec{a}^{(2)})\nabla \bar p_2  )\cdot  \nabla\bar P^{(k)}\zeta^2 dx\\
&\quad + \int_U ( K(|\nabla p_1|, \vec{a}^{(1)})\nabla \bar p_1 - K(|\nabla p_2|, \vec{a}^{(2)})\nabla \bar p_2  )\cdot [2P^{(k)}\zeta\nabla\zeta ]dx
= I_1+I_2.
\end{align*}
Let $\varepsilon>0$. For $I_2$, we use \eqref{Kesta} to estimate
\begin{align*}
I_2
&\le C \int_U (|\nabla \bar p_1|^{1-a}+|\nabla \bar p_2|^{1-a})|\bar P^{(k)}|\zeta|\nabla  \zeta| dx \\
&\le \varepsilon \int_{S_k} |\bar P^{(k)}\zeta|^2 dx 
+C\varepsilon^{-1}  \int_U \Big(|\nabla \bar p_1|+|\nabla \bar p_2| \Big )^{2(1-a)}  \chi_k |\nabla   \zeta|^{2}dx.
\end{align*}
For $I_1$, similar to estimate \eqref{oriI2} of term $I_2$ in Proposition \ref{cont:strictcond} but using monotonicity \eqref{quasimonotone} in place of \eqref{mono1}, we obtain
\begin{align*} 
I_1 \le -(1-a) \int_U K(\xi) |\nabla \bar P^{(k)}\zeta|^2dx + C|\vec a^{(1)}-\vec a^{(2)}| \int_U (|\nabla \bar p_1|^{1-a}+|\nabla \bar p_2|^{1-a})|\nabla \bar P^{(k)}| \zeta^2 dx.
\end{align*}
Similar to the way \eqref{I2} was derived from \eqref{oriI2}, one obtains for $\varepsilon>0$ that
\begin{align*}
I_1
&\le -  \frac {1-a} 2 \int_U  K(\xi)|\nabla (\bar P^{(k)} \zeta) |^2dx + C\int_U  |\bar P^{(k)} \nabla\zeta|^2dx \\
&\quad + C|\vec a^{(1)}-\vec a^{(2)}| \int_U (|\nabla \bar p_1|^{2-a}+|\nabla \bar p_2|^{2-a}) \chi_k \zeta^2 dx.
\end{align*}
Therefore,
\begin{align*}
&\frac12\ddt \int_U |\bar P^{(k)}\zeta|^2  dx+\frac{1-a}{2} \int_U K(\xi) |\nabla (\bar P^{(k)}\zeta)|^2dx\\
&\le \int_U |\bar P^{(k)}|^2\zeta |\zeta_t| dx+ C \int_U K(\xi)|P^{(k)}\nabla \zeta|^2 dx+ \varepsilon \int_{S_k} |\bar P^{(k)}\zeta|^2 dx \\
&\quad +C\varepsilon^{-1}  \int_U \Big(|\nabla \bar p_1|+|\nabla \bar p_2| \Big )^{2(1-a)}  \chi_k |\nabla   \zeta|^{2}dx
+ C|\vec a^{(1)}-\vec a^{(2)}| \int_U (|\nabla \bar p_1|^{2-a}+|\nabla \bar p_2|^{2-a}) \chi_k \zeta^2 dx. 
\end{align*}
Let $J=\sup_{[0,T]} \int_U |\bar P^{(k)}\zeta|^2  dx +\int_0^T \int_U  K(\xi)|\nabla (\bar P^{(k)}\zeta)|^2dxdt$.
Integrating the preceding inequality in time, taking supremum in $t$ over $[0,T]$, selecting $\varepsilon=1/(8T)$ and applying H\"older's inequality yield
\begin{align*}
&J \le C\int_0^T \int_U |\bar P^{(k)}|^2\zeta |\zeta_t| dx dt+C\int_0^T\int_U K(\xi)|P^{(k)}\nabla \zeta|^2 dxdt\\
& \quad +CT \Big\{\int_0^T  \int_U \Big(|\nabla \bar p_1|+|\nabla \bar p_2| \Big )^{2(1-a)\mu} |\nabla   \zeta|^{2\mu} dx dt \Big\}^{\frac 1\mu} \Big\{ \iint_{Q_T\cap{\rm supp}\zeta} \chi_k  dxdt\Big\}^{1-\frac 1\mu} \\
&\quad + C|\vec a^{(1)}-\vec a^{(2)}| \Big\{ \int_0^T \int_{S_k} \Big(|\nabla \bar p_1|+|\nabla \bar p_2| \Big )^{(2-a)\mu} \zeta^2 dx dt\Big\}^{\frac 1\mu} \Big\{ \int_0^T\int_U \chi_k \zeta^2 dxdt\Big\}^{1-\frac 1\mu}. 
\end{align*}
Note that $\vec a^{(i)}$ and $\vec a^{(i)}$ belong to the compact set ${\bf D}$, hence $ |\vec a^{(1)}-\vec{a}^{(2)}| \le C$. Then
\begin{align*}
J&\le C\int_0^T \int_U |\bar P^{(k)}|^2(|\zeta_t|+|\nabla \zeta| ) dx dt\\
& \quad +C  T \Big\{ \int_0^T  \int_U \Big(|\nabla \bar p_1|+|\nabla \bar p_2| \Big )^{2(1-a)\mu} |\nabla   \zeta|^{2\mu} dx dt \Big\}^{\frac 1\mu}  \Big\{ \iint_{Q_T\cap{\rm supp}\zeta} \chi_k  dxdt\Big\}^{1-\frac 1\mu}\\
& \quad +C \Big\{ \int_0^T \int_{S_k} \Big(|\nabla \bar p_1|+|\nabla \bar p_2| \Big )^{(2-a)\mu} \zeta^2 dx dt\Big\}^{\frac 1\mu}  \Big\{ \iint_{Q_T\cap{\rm supp}\zeta} \chi_k  dxdt\Big\}^{1-\frac 1\mu}. 
\end{align*}
Using inequality \eqref{Wemb} in Lemma \ref{ParaSob-3} with $W(x,t)=K(\xi(x,t))$, we have 
\begin{align*}
& \norm{P^{(k)}\zeta}_{L^{\mu_5}(Q_T)} \le C \lambda J 
\le C\lambda\Big\{ \Big( \int_0^T \int_U |\bar P^{(k)}|^2(|\zeta_t|+|\nabla \zeta| ) dx dt\Big)^{1/2} \\
 &\quad +T^{1/2}\Big[\int_0^T  \int_{\setV} \Big(|\nabla \bar p_1|+|\nabla \bar p_2| \Big )^{2(1-a)\mu}|\nabla   \zeta|^{2\mu} dx dt\Big]^{\frac 1{2\mu}}  \Big\{ \iint_{Q_T\cap{\rm supp}\zeta} \chi_k  dxdt\Big\}^{\frac12-\frac 1{2\mu}} \\
&\quad +\Big( \int_0^T\int_{\setV} (|\nabla \bar p_1|+|\nabla \bar p_2| )^{(2-a)\mu} dxdt \Big)^{\frac 1{2\mu}}   \Big\{ \iint_{Q_T\cap{\rm supp}\zeta} \chi_k  dxdt\Big\}^{\frac12-\frac 1{2\mu}}.
 \end{align*} 

Let $Y_i=\| \bar P^{(k_i)} \|_{L^2(A_i)}$. In the same way as in proof of Proposition \ref{cont:strictcond}  we find that 
  \begin{align*}
Y_{i+1}&\le C\lambda M_0^{-1+\frac 2 {\muex_5}}2^i
\Big\{ (1+(\theta T)^{-1} )^\frac 12 Y_i + T^{1/2}\vartheta_1M_0^{-1+\frac 1\mu}Y_i^{1-\frac 1\mu} + \vartheta_2 M_0^{-1+\frac 1\mu} Y_i^{1-\frac 1\mu}  \Big\} Y_i^{1-\frac 2 {\muex_5}}\\
&\le C \lambda2^i
\Big\{ M_0^{-\muex_6}(1+(\theta T)^{-1} )^\frac 12 Y_i^{1+\muex_6} +T^{1/2}\vartheta_1 M_0^{-1-{\gamex_1}} Y_i^{1+{\gamex_1}} +\vartheta_2 M_0^{-1-{\gamex_1}} Y_i^{1+{\gamex_1}}  \Big\}, 
\end{align*}
where the exponents ${\gamex_1},\muex_6$ are defined in Theorem \ref{cont:strictcond}.
Since $Y_0\le \norm{\bar P}_{L^2(U\times (0,T) )} $ we choose $M_0$ sufficiently large such that 
  \[
   \norm{\bar P}_{L^2(U\times (0,T) )}\le C\min\Big\{ ((\lambda (1+(\theta T)^{-1} )^\frac12 )^{-\frac 1\muex_6} M_0 , (\lambda T^{1/2} \vartheta_1)^{-\frac 1 {{\gamex_1}}} M_0^{\frac 1{\gamex_1}+1}, M_0^{\frac 1{\gamex_1}+1} (\lambda \vartheta_2)^{-\frac 1 {{\gamex_1}}} \Big  \},
\]
  thus 
  \begin{align*}
  M_0\ge C\max \Big\{  
  & (\lambda (1+(\theta T)^{-1} )^\frac12 )^{\frac 1\muex_6} \norm{\bar P}_{L^2(U\times (0,T) )},
    (\lambda T^{1/2} \vartheta_1)^{\frac 1 {{\gamex_1}+1}} \norm{\bar P}_{L^2(U\times (0,T) )}^{\frac {\gamex_1}{{\gamex_1}+1}} ,\\
  & (\lambda \vartheta_2)^{\frac 1 {{\gamex_1}+1}} \norm{\bar P}_{L^2(U\times (0,T) )}^{\frac{\gamex_1}{{\gamex_1}+1}} \Big\}.    
  \end{align*}
Using the fact that 
 $ (\lambda (1+(\theta T)^{-1} )^\frac12)^{\frac 1\muex_6}$, $(\lambda T^{1/2} \vartheta_1)^{\frac 1 {{\gamex_1}+1}}$, $(\lambda \vartheta_2)^{\frac 1 {{\gamex_1}+1}}$ are less than or equal to $C \widehat{\mathcal C}_{T_0,T,\theta}$, 
we then choose 
 \beqs
 M_0 =C \widehat{\mathcal C}_{T_0,T,\theta} \Big(  \norm{\bar P}_{L^2(U\times (0,T) )}+\norm{\bar P}_{L^2(U\times (0,T) )}^{\frac {\gamex_1}{{\gamex_1}+1}}\Big). 
 \eeqs
 Then applying Lemma~\ref{multiseq} for $m=3$ to sequence $\{Y_i\}$, and using the same argument as in Theorem \ref{cont:strictcond}, we obtain $|\bar p(x,t)|\le M_0$ in $U'\times (\theta T,T)$.  

(ii) Using same arguments as for deriving Proposition \ref{newcont} from the proof of Proposition \ref{cont:strictcond}, we can modify the proof in part (i) to obtain \eqref{Pcoe2}.
\end{proof}

We recall some results from \cite{HI2} which will be needed in subsequent developments.

\begin{theorem}[cf. \cite{HI2}, Theorems 5.16 and 5.17]\label{supgradthem}
{\rm (i)} For $0< T < \infty$, we have
\beq\label{coeffJsup1} \sup_{[0,T]}  \norm{\bar P(t)}_{L^2}^2 \le  \norm{\bar P(0)}_{L^2}^2 + C M_{5,T} |\vec a^{(1)}-\vec a^{(2)}| ,\eeq 
where $M_{5,T}=A_0+T+\int_0^T f(\tau)d\tau$.

{\rm (ii)} Assume the Degree Condition. 
Suppose $\sup_{[0,\infty)} f(t)<\infty$ and $\sup_{[1,\infty)} \int_{t-1}^t  \tilde f(\tau)d\tau<\infty$.
Then 
\beq\label{coeffJsup2} 
\sup_{[1,\infty)}  \norm{\bar P(t)}_{L^2}^2 \le  \norm{\bar P(0)}_{L^2}^2 + C \Upsilon_6^{b+1} |\vec a^{(1)}-\vec a^{(2)}| ,
\eeq 
\beq\label{coefflim}
\limsup_{t\to\infty}  \norm{\bar P(t)}_{L^2}^2 \le  C \Ulim_2^{b+1} |\vec a^{(1)}-\vec a^{(2)}| ,
\eeq
where $\Upsilon_6=1+A_0+\sup_{[0,\infty)} f^\frac2{2-a}(t)+\sup_{[1,\infty)} \int_{t-1}^t  \tilde f(\tau)d\tau$,
and $\mathcal A_2$ is defined by \eqref{Ulim2}.
\end{theorem}

We will take advantage of calculations in Theorem \ref{thm46}. However the exponents will be changed.
The new counterparts of exponents $\gamma_2,\gamma_3,\ldots,\gamma_6$ in \eqref{mu34}--\eqref{mu89} are
\begin{align*}
&\gamex_2'=(2-a)\mu,
&&\gamex_3'=\frac{\gamex_2'}{2\mu(1-a)}=\frac{2-a}{2(1-a)}\ge 1,\\
&\gamex_4'=\frac1{2\mu}\Big(\frac{2\gamex_2'}{2-a}-1\Big)=1-\frac1{2\mu},
&&\gamex_5'=\frac{\muex_7}{\muex_6}+ \frac{2\gamex_4'+1}{2({\gamex_1}+1)},\\
&\gamex_6'=\frac{\muex_7(2-a)}{\muex_6(1-a)}+\frac{\gamex_3'}{{\gamex_1}+1}.
\end{align*}
Firstly, we replace $\gamex_2$ by $\gamex_2'$ in estimate \eqref{gradInt20} to obtain
\beq\label{gradInt2}
\vartheta_{2,T}\le C (T+1)^{\gamex_4'}\Nf_{1,T}^{\gamex_3'}=C \ell_2'(T+1)^{\gamex_4'}\Nf_{1,T}^{\gamex_3'},
\eeq
where $\ell_2'=[\sum_{i=1,2} L_1(\gamex_2'+a;[p_i(0)])]^\frac1{2\mu}$.
Secondly, since $\gamex_4'\ge 1/(2\mu)$, it follows
\beq\label{upUp12}
\vartheta_{1,T}\le CT^\frac1{2\mu}+\vartheta_{2,T}\le C \ell_2'(T+1)^{\gamex_4'}\Nf_{1,T}^{\gamex_3'}.
\eeq
Thirdly, same as \eqref{lamd}, we have 
\beq\label{lamUp2}
\lambda_T\le C \ell_1' (T+1)^{\muex_7} \Nf_{1,T}^\frac{\muex_7(2-a)}{1-a} \Nf_{2,T}^{\muex_7},
\eeq
where $\ell_1'=1+A_0+\norm{\nabla p_1(0)}_{L^{2-a}}^{2-a}+\norm{\nabla p_2(0)}_{L^{2-a}}^{2-a}$.

For finite time intervals, we have the following estimates.

\begin{theorem}\label{thm54}
   \asdc.  
 
{\rm (i)} For $T\in(0, 1]$, we have
\beq\label{P-inter-t<1}
\sup_{[0,T]} \norm{\bar P}_{L^\infty(U')} \le  2\|\bar P(0)\|_{L^\infty} +C M_{6,T}\big ( \mathcal A+\mathcal A^{\frac {\gamex_1}{{\gamex_1}+1}}\big), 
\eeq
where $M_{6,T}$ is defined in \eqref{N5def} below and
$\mathcal A = \norm{\bar P(0)}_{L^2} +|\vec a^{(1)}-\vec a^{(2)}|^{1/2}$.

{\rm (ii)} For $T>1$, we have
\beq\label{contfinite40}
\begin{aligned}
\sup_{[1,T]}\|\bar P\|_{L^\infty(U')} 
&\le C M_{7,T} \Big (\mathcal A+\mathcal A^{\frac {\gamex_1}{{\gamex_1}+1}}\Big),
\end{aligned}
\eeq
where $M_{7,T}$ is defined in \eqref{N6def} below.
\end{theorem}
\begin{proof}
(i) 
Let $T\in(0,1]$. 
Applying \eqref{Pcoe2} of Proposition \ref{theo49} gives
\begin{align*}
\|\bar P(t)\|_{L^\infty(U')}
&\le 2\|\bar P(0)\|_{L^\infty}+C \widehat{\mathcal C}_{T}  \Big\{ T^{1/2}\sup_{[0,t]}  \norm{\bar P}_{L^2} + (T^{1/2}\sup_{[0,t]} \norm{\bar P}_{L^2} )^{\frac {\gamex_1}{{\gamex_1}+1} }\Big\}\\
&\le 2\|\bar P(0)\|_{L^\infty}+C \widehat{\mathcal C}_{T}  T^{\frac {\gamex_1}{2({\gamex_1}+1)} } \Big\{ 1+\sup_{[0,t]}  \norm{\bar P}_{L^2}\Big\}.
\end{align*}
Using \eqref{gradInt2}, \eqref{upUp12} and \eqref{lamUp2} to estimate corresponding terms in \eqref{Ctheta6}, we have 
 \begin{align*}
 \widehat{\mathcal C}_{T}\le C \lambda_T^\frac{1}{\muex_6} (\vartheta_{1,T}+\vartheta_{2,T})^\frac1{{\gamex_1}+1}
= C  \ell_3' \Nf_{1,T}^{\gamex_6'} \Nf_{2,T}^\frac{\muex_7}{\muex_6},
\end{align*}
where $\ell_3'=(\ell_1')^\frac1{\muex_6} (\ell_2')^\frac1{{\gamex_1}+1}$.
Moreover, by \eqref{coeffJsup1} 
\beq\label{CalA}
\sup_{[0,T]}  \norm{\bar P}_{L^2}
\le C( \norm{\bar P(0)}_{L^2} +|\vec a^{(1)}-\vec a^{(2)}|^\frac12 ) \Big[1+A_0+\int_0^T f(\tau)d\tau\Big]^\frac12
\le C \mathcal A (1+A_0)^\frac12 \Nf_{1,T}^\frac{2-a}{2(1-a)}.
\eeq  
Combining the above, we obtain \eqref{P-inter-t<1} with  
\beq\label{N5def}
M_{6,T}= \ell_3' (1+A_0)^{1/2} T^\frac{{\gamex_1}}{2({\gamex_1}+1)} \Nf_{1,T}^{\gamex_6'+\frac{2-a}{2(1-a)}}\Nf_{2,T}^\frac{\muex_7}{\muex_6}. 
\eeq

(ii) Let $T>1$. We apply \eqref{Pcoe} with $T_0=0$ and $\theta T=1$. Note that
 \begin{align*}
 \widehat{\mathcal C}_{0,T,\theta}
\le C  \lambda_T^\frac{\muex_7}{\muex_6} ( 1+ [T^{1/2} (\vartheta_{1,T}+\vartheta_{2,T}) ]^\frac 1{{\gamex_1}+1}  )
\le C \ell_3' T^{\gamex_5'}\Nf_{1,T}^{\gamex_6'} \Nf_{2,T}^\frac{\muex_7}{\muex_6}.
\end{align*}
Then similar calculations to those for proving \eqref{contfinite} in Theorem \ref{thm46}  lead to 
\eqref{contfinite40} with
\beq\label{N6def}
M_{7,T}= (1+A_0)^{1/2} \ell_3' T^{\gamex_5'+1}  \Nf_{1,T}^{\gamex_6'+\frac{2-a}{2(1-a)}} \Nf_{2,T}^\frac{\muex_7}{\muex_6}. 
\eeq
The proof is complete.
\end{proof}

We now consider estimates when $t\to\infty$. We use the same notation as in subsection \ref{sec41}.
For $t\ge 2$, we apply Proposition \ref{theo49}(i) to the interval $[t-1,t]$, that is $T_0=t-1$ and $T=1$, and use $\theta = 1/2$. Then we have 
\beq\label{rawJs}
\|\bar P(t)\|_{L^\infty(U')} \le C \widehat{\mathcal C}(t) \Big\{ \sup_{[t-1,t]} \norm{\bar P}_{L^2}+\sup_{[t-1,t]} \norm{\bar P}_{L^2}^{\frac {\gamex_1}{{\gamex_1}+1} }\Big\},
\eeq
where  
 \beq\label{newCtil}
\widehat{\mathcal C}(t) = \tilde \lambda(t)^\frac{1}{\muex_6} (1+\widehat\vartheta(t))^\frac1{{\gamex_1}+1},
\eeq
with $\tilde\lambda(t)$ defined by \eqref{lamtildefn}, and
\beq\label{newuptil}
\widehat\vartheta(t) = 1+  \Big(\int_{t-1}^t\int_{\setV} |\nabla \bar p_1|^{(2-a)\mu}+|\nabla \bar p_2|^{(2-a)\mu} dxdt\Big)^\frac1{2\mu}.
\eeq
Note that $\widehat\vartheta(t)$ is the same as $\tilde\vartheta(t)$ in \eqref{uptildef} with exponent $\gamex_2'=(2-a)\mu$ replacing $\gamex_2=2(1-a)\mu$.

By (5.53) in \cite{HI2}, there is $d_7>0$ such that
\begin{align} \label{coeffdP}
\frac 12\ddt \int_U \bar P^2 dx \le - d_7 \Lambda^{-b}(t) \int_U \bar P^2 dx
+C|\vec a^{(1)}-\vec a^{(2)}|\Lambda(t).
\end{align}
Hence, for $t'\ge 1$
\beq\label{JbarP1}
\begin{aligned}
 \norm{\bar P(t')}_{L^2}^2
&\le e^{-d_7\int_1^{t'} \Lambda(\tau)^{-b}d\tau}  \norm{\bar P(1)}_{L^2}^2+C|\vec a^{(1)}-\vec a^{(2)}|\int_1^{t'} e^{-d_7\int_\tau^{t'} \Lambda(\theta)^{-b}d\theta}\Lambda(\tau)d\tau.
\end{aligned}
\eeq
Similar to \eqref{supPtprime},  we obtain
\beq\label{JbarP2}
\begin{aligned}
\sup_{[t-1,t]}  \norm{\bar P(t')}_{L^2}^2
&\le C e^{-d_7\int_0^{t} \Lambda(\tau)^{-b}d\tau}  \norm{\bar P(1)}_{L^2}^2 +C|\vec a^{(1)}-\vec a^{(2)}|\int_1^{t} e^{-d_7\int_\tau^{t} \Lambda(\theta)^{-b}d\theta}\Lambda(\tau)d\tau.
\end{aligned}
\eeq

Let $\Upsilon_2$ be defined as in Corollary \ref{corMM}, and
$$\Upsilon_3=1+\sup_{[0,\infty)} \norm{\psi}_{L^\infty},\quad \Upsilon_7 = 1+\sup_{[2,\infty)}\int_{t-2}^t \tilde f(\tau) d\tau.$$

We have the following result for unbounded time intervals.

\begin{theorem}\label{thmsupP}

\asdc. Suppose $\Upsilon_2$, $\Upsilon_3$ and $\Upsilon_7$ are finite numbers. 
 Then we have
 \beq\label{supPlarge}
\sup_{[2,\infty)} \|\bar P\|_{L^\infty(U')}\le C \Upsilon_8 ( \mathcal A+\mathcal A^{\frac {\gamex_1}{{\gamex_1}+1}}),
 \eeq
 where $\Upsilon_8$ is defined by \eqref{Up4def} below, and $\mathcal A$ is defined in Theorem \ref{thm54}.
  \end{theorem}
\begin{proof}  
Use \eqref{lamtilde} to bound $\Lambda(t)$ and $\tilde \lambda(t)$, we have for all $t\ge 1$ that
\beq\label{supLg0}
\sup_{\tau\in[t-1,t]}\Lambda(\tau) \le C \Upsilon_7 \Upsilon_3^\frac2{1-a},
\quad \tilde \lambda(t) \le C \Upsilon_7^{\muex_7} \Upsilon_3^\frac{2\muex_7}{1-a}.
\eeq
Moreover, using \eqref{newgamtil} to bound $\widehat\vartheta(t)$ with $\gamex_2'$ replacing $\gamex_2$, we obtain
\beq\label{boundnu}
\widehat\vartheta(t)\le C L'_{12} \Bfun_1(t)^\frac{\muex_4(\gamex_2'+a-2)}{2\mu}\Bfun_2(t)^\frac1{2\mu}
\le C  L'_{12} \Upsilon_3^\frac{\muex_4(\gamex_2'+a-2)}{2\mu} \Upsilon_2^\frac1{2\mu},
\eeq
where $L'_{12}=\{\sum_{i=1,2} L_3(\gamex_2'+a,[p_i(0)])\}^\frac1{2\mu}$.
It follows from \eqref{newCtil}, \eqref{supLg0} and \eqref{boundnu} that 
\beq\label{preCpsi}
\widehat{\mathcal C}(t) \le C  \eta_2,\quad \text{where }
\eta_2=\Upsilon_7^\frac{\muex_7}{\muex_6} \Upsilon_2^\frac1{2\mu({\gamex_1}+1)}  (L'_{12})^\frac{1}{{\gamex_1}+1} \Upsilon_3^{ \frac{2\muex_7}{\muex_6(1-a)}  +\frac{\muex_4(\gamex_2'+a-2)}{2\mu({\gamex_1}+1)} }.
\eeq
Then we have from \eqref{rawJs} and \eqref{preCpsi} that
\beq\label{rawJt}
\norm{\bar P(t)}_{L^\infty(U')}
\le C \eta_2 \Big\{ \sup_{[t-1,t]} \norm{\bar P}_{L^2}+\sup_{[t-1,t]} \norm{\bar P}_{L^2}^{\frac {\gamex_1}{{\gamex_1}+1} }\Big\}.
\eeq
By  \eqref{JbarP2} and \eqref{supLg0}, we have
\begin{align*}
\sup_{[t-1,t]}  \norm{\bar P}_{L^2}^2\le C   \norm{\bar P(1)}_{L^2}^2+ C|\vec a^{(1)}-\vec a^{(2)}|(\Upsilon_7 \Upsilon_3^\frac2{1-a})^{b+1}.
\end{align*}
We estimate $ \norm{\bar P(1)}_{L^2}$ by \eqref{CalA} with $T=1$ and obtain
\begin{align*}
\sup_{[t-1,t]}  \norm{\bar P}_{L^2}\le C  \eta_3 \mathcal A, \quad \text{where }\eta_3=(1+A_0)^{1/2} \Nf_{1,1}^\frac{2-a}{2(1-a)} + (\Upsilon_7 \Upsilon_3^\frac2{1-a})^\frac{b+1}2,
\end{align*}
with $\Nf_{1,1}$ defined by \eqref{Nf1}.
Combining this with \eqref{rawJt}, we obtain \eqref{supPlarge} with 
\beq\label{Up4def}
 \Upsilon_8 = \eta_2 \eta_3.
 \eeq
The proof is complete.
\end{proof}

For results as $t\to\infty$, we have:

\begin{theorem}\label{thm55}
Assume the same as in Theorem \ref{thmsupP}. Then
\beq\label{JbarP4}
\limsup_{t\to\infty} \|\bar P(t)\|_{L^\infty(U')}
\le C \Upsilon_9 \Big(|\vec a^{(1)}-\vec a^{(2)}| + |\vec a^{(1)}-\vec a^{(2)}|^\frac{{\gamex_1}}{{\gamex_1}+1}\Big)^{1/2},
\eeq
where $\Upsilon_9$ is defined by \eqref{Up10def} below.
\end{theorem}
\begin{proof}
Define $\Bfun_7(t)=1+A^\frac2{2-a}+\int_{t-2}^t \tilde f(\tau)d\tau$.
Note that $\limsup_{t\to\infty} \Bfun_7(t)\le 2\Ulim_2$, where $\Ulim_2$ is defined by \eqref{Ulim2}.
 By \eqref{boundedlarge}, for $t$ sufficiently large we have
 \beq\label{Lamlim}
 \sup_{\tau\in[t-1,t]}\Lambda(\tau)\le  C \Bfun_7(t)
\quad \text{and}\quad
\tilde \lambda(t)\le C \Bfun_7(t)^{\muex_7}.
 \eeq
Using \eqref{Lamlim} and \eqref{boundnu} to estimate $\widehat{\mathcal C}(t)$ in \eqref{newCtil}, we derive from \eqref{rawJs} for large $t$ that
\beq\label{rawJg}
\|\bar P(t)\|_{L^\infty(U')} \le C \eta_4 \Bfun_7(t)^\frac{\muex_7}{\muex_6}  
\Big\{ \sup_{[t-1,t]} \norm{\bar P}_{L^2}+\sup_{[t-1,t]} \norm{\bar P}_{L^2}^{\frac {\gamex_1}{{\gamex_1}+1} }\Big\},
\eeq
where
$\eta_4=\big[ L'_{12} \Upsilon_2^\frac1{2\mu} \Upsilon_3^\frac{\muex_4(\gamex_2'+a-2)}{2\mu} \big]^\frac1{{\gamex_1}+1}.$
Similar to \eqref{supPtprime} and \eqref{JbarP2}, for large $T$ and $t>T$, we have 
\beq\label{PbarJ3}
\sup_{[t-1,t]}  \norm{\bar P}_{L^2}^2
\le C e^{-C\int_T^{t} \Bfun_7(\tau)^{-b}d\tau}  \norm{\bar P(T)}_{L^2}^2+C|\vec a^{(1)}-\vec a^{(2)}|\int_T^{t} e^{-C\int_\tau^{t} \Bfun_7(\theta)^{-b}d\theta} \Bfun_7(\tau)d\tau.
\eeq
Since $\Bfun_7(t)$ is bounded, one can easily see from \eqref{rawJg} and \eqref{PbarJ3} that
\begin{align*}
\limsup_{t\to\infty}  \|\bar P(t)\|_{L^\infty(U')}
&\le C \eta_4 \Big\{ \limsup_{t\to\infty} \Bfun_7(t)^\frac{\muex_7}{\muex_6} \Big\} \Big\{ 
|\vec a^{(1)}-\vec a^{(2)}|^\frac12 \Big[ \limsup_{t\to\infty} \int_T^{t} e^{-C\int_\tau^{t} \Bfun_7(\theta)^{-b}d\theta} \Bfun_7(\tau)d\tau\Big ]^\frac12\\
&\quad + |\vec a^{(1)}-\vec a^{(2)}|^\frac{{\gamex_1}}{2({\gamex_1}+1)}  \Big[\limsup_{t\to\infty}   \int_T^{t} e^{-C\int_\tau^{t} \Bfun_7(\theta)^{-b}d\theta} \Bfun_7(\tau)d\tau\Big]^\frac{{\gamex_1}}{2({\gamex_1}+1)}\Big\}.
\end{align*}
Applying Lemma \ref{difflem2} with $h\equiv 1$ to the last two limits, we obtain
\begin{align*}
\limsup_{t\to\infty}  \|\bar P(t)\|_{L^\infty(U')}
&\le C \eta_4 \Ulim_2^{\frac{\muex_7}{\muex_6}}\Big\{ 
|\vec a^{(1)}-\vec a^{(2)}|^\frac12 \Ulim_2^{\frac{b+1}{2}}
+ |\vec a^{(1)}-\vec a^{(2)}|^\frac{{\gamex_1}}{2({\gamex_1}+1)} \Ulim_2^{\frac{(b+1){\gamex_1}}{2({\gamex_1}+1)}}\Big\}.
\end{align*}
Then \eqref{JbarP4} follows with 
\beq\label{Up10def}
\Upsilon_9=\eta_4 \Ulim_2^{\frac{\muex_7}{\muex_6}+\frac{b+1}2}.
\eeq
The proof is complete.
\end{proof}

Similar to Remark \ref{rmk414}, estimate \eqref{JbarP4} shows that the smallness of  $\|\bar P(t)\|_{L^\infty(U')}$ as $t\to\infty$ can be controlled by $|\vec a^{(1)}-\vec a^{(2)}|$. 

\subsection{Results for pressure gradient}
\label{sec52}

We now study the dependence for pressure gradient.
The results are  parallel to those in subsection \ref{sec42}.
For $i=1,2$, we denote $H_i(\xi)=H(\xi,a^{(i)})$ defined in \eqref{Hxi}, and recall that the functional $J_H[\cdot]$ is defined by \eqref{J.def}.

\begin{proposition}\label{Grada}
 Let $\delta\in (0,a)$ and $U'\Subset \setV\Subset U$. There exists a constant $C>0$ depending on $U$, $U'$,$\setV$ and $\delta$ such that for any $t>0$, we have
\begin{multline}\label{DG9}
 \|\nabla P(t)\|_{L^{2-\delta}(U')} 
\le C\norm{\bar P(t)}_{L^2}^\frac12 \Big(1+\sum_{i=1,2}\norm{\bar p_{it}(t)}_{L^2}^2+\sum_{i=1,2} J_{H_i}[p_i](t)\Big)^\frac12 \\
\cdot \Big(1+ \sum_{i=1,2} \int_{\setV} |\nabla p_i(x,t)|^\frac{a(2-\delta)}{\delta}dx \Big)^\frac{\delta}{2(2-\delta)},
\end{multline}
and for any $T>0$, we have
\begin{multline}\label{GP}
\| \nabla P\|_{L^{2-\delta}(U'\times(0,T))} 
\le C\Big\{ \norm{\bar P}_{L^2(U\times(0,T))}^\frac12 \Big( T+\sum_{i=1,2}\norm{\bar p_{it}}_{L^2(U\times(0,T))}^2 + \sum_{i=1,2} \int_0^T J_{H_i}[p_i](t)dt\Big)^\frac14 \\
 +  |\vec a^{(1)}-\vec a^{(2)}|^\frac12 \Big(T+\sum_{i=1,2}\int_0^T J_{H_i}[p_i](t) dt\Big)^{\frac 1 2}\Big\}  
\cdot \Big\{ T+\sum_{i=1,2}\int_{0}^T\int_{\setV} |\nabla p_i(x,t)|^\frac{a(2-\delta)}{\delta}dx dt\Big\}^\frac{\delta}{2(2-\delta)}.
\end{multline}
\end{proposition}
\begin{proof}
Let $\zeta(x,t)$ be the same cut-off function as in Proposition \ref{PropGrad}.
Multiplying equation \eqref{Erreq} by $\bar P\zeta^2 $, integrating over $U$ and using integration by parts, we obtain
\begin{align*}
\int_U \bar  P_t \bar P \zeta^2 dx
&= -\int_U \Big(K(\nabla p_1|,\vec a^{(1)})\nabla p_1-K(|\nabla p_2|,\vec a^{(2)})\nabla p_2 \Big ) \cdot (\nabla  P \zeta^2+2\bar P \zeta\nabla \zeta) dx.
\end{align*}
Again, the fact  $\nabla \bar P =\nabla P$ was used in the above.
Let $\xi(x)=|\nabla p_1|\vee |\nabla p_2|$. By the monotonicity \eqref{quasimonotone},
\begin{align*}
&\int_U \bar  P_t \bar P \zeta^2 dx
\le -(1-a) \int_U K(\xi,\vec a^{(1)}\vee a^{(2)})|\nabla P\zeta|^2  dx \\
&\quad + C|\vec a^{(1)}-\vec a^{(2)}| \int_{U} (|\nabla p_1|^{1-a}+|\nabla p_2|^{1-a})|\nabla P|\zeta^2 dx
    +C\int_{U} ( |\nabla p_1|+ |\nabla p_2|)^{1-a}|\bar P|\zeta dx.        
\end{align*}
Hence,
\begin{align*}
&(1-a) \int_U K(\xi,\vec a^{(1)}\vee a^{(2)})|\nabla P\zeta|^2  dx
\le C \norm{\bar P_{t}}_{L^2}\norm{\bar P}_{L^2}\\
&\quad  +C\Big(\int_U ( |\nabla p_1|+ |\nabla p_2|)^{2-2a}dx\Big)^{1/2} \norm{\bar P}_{L^2} 
 + C|\vec a^{(1)}-\vec a^{(2)}| \int_U (|\nabla p_1|^{2-a}+|\nabla p_2|^{2-a})dx.
\end{align*}
By constructing appropriate function $\zeta(x)\in[0,1]$ with $\zeta\equiv 1$ on $V$, we have
\begin{multline}\label{s1}
\int_{\setV} K(\xi,\vec a^{(1)}\vee a^{(2)})|\nabla P|^2  dx 
\le \int_{U} K(\xi,\vec a^{(1)}\vee a^{(2)})|\nabla P|^2 \zeta^2 dx \\
\le C  \norm{\bar P}_{L^2} \Big\{  \norm{\bar p_{1t}}_{L^2}+ \norm{\bar p_{2t}}_{L^2}+\Big(\int_U ( |\nabla p_1|+ |\nabla p_2|)^{2-2a}dx\Big)^{1/2} \Big\} \\
 + C|\vec a^{(1)}-\vec a^{(2)}| \int_U (|\nabla p_1|^{2-a}+|\nabla p_2|^{2-a})dx.
\end{multline}
Using \eqref{DG2} with $K(\xi)=K(\xi,\vec a^{(1)}\vee a^{(2)})$ and \eqref{s1} we obtain 
\begin{multline*}
 \|\nabla P(t)\|_{L^{2-\delta}(U')} 
\le C\norm{\bar P(t)}_{L^2}^{1/2} \Big\{ \Big(1+\norm{\bar p_{1t}(t)}_{L^2}^2+\norm{\bar p_{2t}(t)}_{L^2}^2 + J_{H_1}[p_1]+J_{H_2}[p_2]\Big)^{1/4} \\
\quad  +  |\vec a^{(1)}-\vec a^{(2)}|^{1/2} \Big( 1+J_{H_1}[p_1]+J_{H_2}[p_2]  \Big)^{1/2}\Big\}  \Big( \int_{\setV} (1+|\nabla p_1|+|\nabla p_2|)^\frac{a(2-\delta)}{\delta}dx \Big)^\frac{\delta}{2(2-\delta)}.
\end{multline*}
Then simply due to $|\vec a^{(1)}-\vec a^{(2)}|\le C$, inequality \eqref{DG9} follows.

To prove \eqref{GP}, we have from \eqref{s1} that
\beq
\begin{aligned}
&\int_0^T\int_{\setV} K(\xi)|\nabla P|^2  dxdt \\
&\le C \Big(\int_0^T  \norm{\bar P}_{L^2}^2 dt\Big)^{1/2} \Big( \int_0^T \norm{\bar p_{1t}}_{L^2}^2+ \norm{\bar p_{2t}}_{L^2}^2 dt+ \int_0^T\int_U ( |\nabla p_1|+ |\nabla p_2|)^{2-2a}dx dt\Big)^{1/2} \\
&\quad  + C|\vec a^{(1)}-\vec a^{(2)}| \int_0^T \int_U (|\nabla p_1|^{2-a}+|\nabla p_2|^{2-a})dx dt.
\end{aligned}
\eeq
According to \eqref{GradPclose2},
\beq\label{s2}
\int_{0}^T \int_{U'} |\nabla P|^{2-\delta}  dx dt
\le C \Big ( \int_{0}^T \int_U K(\xi)|\nabla P\zeta|^2  dxdt \Big )^\frac{2-\delta}{2} \Big( \int_{0}^T\int_{\setV} (1+|\nabla p_1|+|\nabla p_2|)^\frac{a(2-\delta)}{\delta}dx dt\Big)^\frac{\delta}{2}.
\eeq
Using \eqref{s1} in \eqref{s2}, we obtain 
\begin{multline*}
\int_{0}^T \int_{U'} |\nabla P|^{2-\delta}  dx dt \\
\le C\Big\{\Big(\int_0^T \norm{\bar P}_{L^2}^2 dt\Big)^{1/2} \Big( \int_0^T \norm{\bar p_{1t}}_{L^2}^2+ \norm{\bar p_{2t}}_{L^2}^2 dt+ \int_0^T\int_U ( |\nabla p_1|+ |\nabla p_2|)^{2-2a}dx dt\Big)^{1/2} \\
  + |\vec a^{(1)}-\vec a^{(2)}| \int_0^T \int_U (|\nabla p_1|^{2-a}+|\nabla p_2|^{2-a})dx dt\Big\}^{\frac {2-\delta} 2}  \Big( \int_{0}^T\int_{\setV} (1+|\nabla p_1|+|\nabla p_2|)^\frac{a(2-\delta)}{\delta}dx dt\Big)^\frac{\delta}{2}.
\end{multline*}
Thus, we have \eqref{GP}.  
\end{proof}

\begin{theorem}\label{GradThm3}
\asdc.  For $\delta\in(0,a)$, $0<t_0<1$, and $T>t_0$ we have
\beq\label{DG7}
\sup_{[t_0,T]} \|\nabla P(t)\|_{L^{2-\delta}(U')} \le C M_{8,t_0,T}(\|\bar P(0)\|_{L^2} +|\vec a^{(1)}-\vec a^{(2)}|^{1/2})^\frac12,
\eeq
where $M_{8,t_0,T}$ is defined in \eqref{N9def} below.
\end{theorem}
\begin{proof}
 Using estimates \eqref{coeffJsup1}, \eqref{Jpt-boundA0}, \eqref{H-bound-0} and \eqref{gradalot} in \eqref{DG9}, we obtain \eqref{DG7} with
\begin{multline}\label{N9def}
 M_{8,t_0,T}=M_{5,T}^{1/4} \Big( t_0^{-1} \sum_{i=1,2} L_5(t_0;[p_i(0)]) + (T+1)\Nf_{1,T}^\frac{2-a}{1-a} +\Nf_{2,T}\Big)^{1/2}\\
\cdot \Big\{  \sum_{i=1,2}L_2(\nuex_1;[p_i(0)])  (T+1)^{\frac{2(\nuex_1-2)}{2-a}+1} \Nf_{1,T}^{\frac{\nuex_1-a}{1-a}}\Big\}^\frac{\delta}{2(2-\delta)}.
\end{multline}
\end{proof}

\begin{theorem}\label{GradThm4}
\asdc. 
Suppose $\Upsilon_1$, $\Upsilon_2$ (defined in \eqref{xgood}) and  $\Ulim_2$ (defined by \eqref{Ulim2}) are finite.
Then for any $\delta\in(0,a)$, we have
\beq\label{DG8}
\limsup_{t\to\infty} \|\nabla P(t)\|_{L^{2-\delta}(U')} \le C \Upsilon_{10} |\vec a^{(1)}-\vec a^{(2)}|^{1/4},
\eeq
where $\Upsilon_{10}$ is defined by \eqref{U8} below.
\end{theorem}
\begin{proof}
Taking limsup of \eqref{DG9} and making use the limits \eqref{coefflim}, \eqref{limsupH1} and estimate \eqref{Kgrad60pw} with $s=\nuex_1$ yield
\beqs
 \limsup_{t\to\infty} \|\nabla P(t)\|_{L^{2-\delta}} 
\le C \Big(\Ulim_2^{b+1} |\vec a^{(1)}-\vec a^{(2)}|\Big)^{1/4} 
\Ulim_2^{1/2} \Big\{ \sum_{i=1,2} L_4(\nuex_1;[p_i(0)]) \Upsilon_1^{\muex_4(\nuex_1-2)} \Upsilon_2 \Big\}^\frac\delta{2(2-\delta)}.
\eeqs
Therefore, we obtain \eqref{DG8} with 
\beq\label{U8}
\Upsilon_{10}=\Ulim_2^{\frac {b+3}4} \Big\{ \Upsilon_1^{\muex_4(\nuex_1-2)} \Upsilon_2 \sum_{i=1,2}L_4(\nuex_1;[p_i(0)])  \Big\}^\frac\delta{2(2-\delta)}.
\eeq
\end{proof}

\begin{theorem}\label{lastthm}
\asdc.  
Let $\delta\in (0,a)$ and $T> 0$. Then there exists a constant $C>0$ depending on $U,U', \delta$ such that
\beq\label{GP2}
\| \nabla P\|_{L^{2-\delta}(U'\times(0,T))}
\le C  M_{9,T} \Big(  \norm{\bar P(0)}_{L^2}^{1/2} +|\vec a^{(1)}-\vec a^{(2)}|^{1/2} +|\vec a^{(1)}-\vec a^{(2)}|^{1/4}\Big), 
\eeq
where $M_{9,T}$ is defined by \eqref{N7def} below.
\end{theorem}
\begin{proof}
Define $N_{2,T}=\sum_{i=1}^2 \norm{\bar p_i(0)}_{L^2}^2 + \sum_{i=1}^2 \|\nabla p_i(0)\|_{L^{2-a}}^{2-a }  +(T+1)\Nf_{1,T}^\frac{2-a}{1-a} +\Nf_{2,T}$.
From \eqref{p-bar-bound1} and \eqref{H-bound-0} we have
\beq\label{s3}
\sum_{i=1}^2 \norm{\bar p_{it}}_{L^2(U\times(0,T))}^2 + \sum_{i=1}^2\int_0^T\int_U |\nabla p_i(x,t)|^{2-a} dx dt \le C N_{2,T}.
\eeq
Neglecting the negative term on the right-hand side of \eqref{coeffdP} and integrating in $t$ twice yield 
 \beq\label{s4}
 \norm{\bar P}_{L^2(U\times(0,T))}^2 \le  C T \Big  (  \norm{\bar P(0)}_{L^2}^2 +|\vec a^{(1)}-\vec a^{(2)}| \int_0^T \Lambda(t)dt \Big)
\le  C T N_{2,T} \Big  (  \norm{\bar P(0)}_{L^2}^2 +|\vec a^{(1)}-\vec a^{(2)}|\Big).
 \eeq
Combining  \eqref{s3}, \eqref{s4}, \eqref{subs0} and \eqref{GP} we get
\begin{align*}
&\| \nabla P\|_{L^{2-\delta}(U'\times(0,T))} 
\le C \Big\{ T^{1/4} N_{2,T}^{1/4} ( \norm{\bar P(0)}_{L^2}^2 +|\vec a^{(1)}-\vec a^{(2)}|)^{1/4} \cdot N_{2,T}^{1/4}\\
&\quad +|\vec a^{(1)}-\vec a^{(2)}|^{1/2} N_{2,T}^{1/2}  \Big\}
 \cdot \Big\{\Big[ \sum_{i=1,2} L_1(\nuex_1+a;[p_i(0)]) \Big]
(T+1)^{\frac {2\nuex_1}{2-a} - 1}  \Nf_{1,T}^{\frac{\nuex_1}{1-a}}\Big\}^\frac\delta{2(2-\delta)}.
\end{align*}
Therefore, we obtain \eqref{GP2} with
\beq\label{N7def}
M_{9,T}=(T+1)^{\frac14}N_{2,T}^{1/2} \Big\{\Big[ \sum_{i=1,2} L_1(\nuex_1+a;[p_i(0)]) \Big]
(T+1)^{\frac {2\nuex_1}{2-a} - 1}  \Nf_{1,T}^{\frac{\nuex_1}{1-a}}\Big\}^\frac\delta{2(2-\delta)}.
\eeq
The proof is complete.
\end{proof}


\myclearpage

\appendix

\section{Auxiliaries}\label{Aux}

We recall a classical result on fast decaying geometry sequences.

\begin{lemma}[cf. \cite{LadyParaBook68}, Chapter II, Lemma 5.6]\label{oriseq} Let $\{Y_i\}_{i=0}^\infty$ be a sequence of non-negative numbers satisfying
\beq \label{oABi}
Y_{i+1}\le AB^i Y_i^{1+\mu},  \quad 
i =0,1,2,\cdots,
\eeq
where  $A>0$, $B>1$ and $\mu>0$.   
Then
\beq\label{Yiineq}
Y_i \le A^{\frac{(1+\mu)^i -1}{\mu}} B^{\frac{(1+\mu)^i -1}{\mu^2} - \frac{i}{\mu}}Y_0^{(1+\mu)^{i}}, 
 \quad 
i =0,1,2,\cdots.
\eeq
Consequently, if $Y_0\le  A^{-1/\mu}B^{-1/\mu^2}$ then  $\displaystyle{\lim_{i\to\infty} Y_i=0}$. 
\end{lemma}

The following generalization is an important ingredient for our iteration technique used in this paper.

\begin{lemma}\label{multiseq} Let $\{Y_i\}_{i=0}^\infty$ be a sequence of non-negative numbers satisfying
\beq \label{mABi}
Y_{i+1}\le \sum_{k=1}^m A_k B_k^i  Y_i^{1+\mu_k}, \quad 
i =0,1,2,\cdots,
\eeq
where  $A_k>0$, $B_k>1$ and $\mu_k>0$ for $k=1,2,\ldots,m$.
Let $B=\max\{B_k : 1\le k\le m\}$ and $\mu=\min\{\mu_k : 1\le k\le m\}$. Then the following statement holds true.
\beq\label{Y0andD}
\text{If}\quad \sum_{k=1}^m  A_k Y_0^{\mu_k} \le B^{-1/\mu}
\quad\text{then } \lim_{i\to\infty} Y_i=0.
\eeq
In particular,
\begin{align}\label{Y0mcond}
\text{if}\quad  Y_0\le \min\{ (m^{-1} A_k^{-1} B^{-\frac 1 {\mu}})^{1/\mu_k} : 1\le k\le  m\}
\quad \text{then } \lim_{i\to\infty} Y_i=0.
\end{align} 
\end{lemma}
\begin{proof}
It is obvious that \eqref{Y0mcond} is a consequence of \eqref{Y0andD} when we require
 $A_k Y_0^{\mu_k} \le m^{-1}B^{-1/\mu}$ for each $k=1,2,\ldots,m$.
We now prove \eqref{Y0andD}. To make a clear presentation, we consider $m=2$. The proof can be robustly modified to cover general $m$.
 Therefore, we consider a sequence  $\{Y_i \}_{i=0}^\infty $  of non-negative numbers that satisfies   
\beqs Y_{i+1} \le B^i(A_1Y_i^{1+\mu_1} +A_2Y_i^{1+\mu_2} )\quad \text{for all }i\ge 0, 
\eeqs
where $A_1,A_2>0$, $B>1$ and $0<\mu_1\le \mu_2$, and $Y_0$ satisfies
\beq\label{Dmu200}
A_1Y_0^{\mu_1} + A_2Y_0^{\mu_2} \le B^{-\frac 1 {\mu_1}}.
\eeq 
Note, in this case, that $\mu=\mu_1$.

\noindent\textbf{Claim.} If there is $D>0$ such that 
\beq\label{Y02}
Y_0\le D ( A_1D^{\mu_1} +A_2D^{\mu_2} )^{-\frac 1 {\mu_1}} B^{-\frac 1 {\mu_1^2}}  \le D,
\eeq 
then 
\beq\label{YD2}
 Y_i\le D\quad \text{for all } i\ge 0,
\eeq 
and 
\beq
\label{Ylim2} \lim_{i\to\infty} Y_i=0.
\eeq 
\textit{Proof of the Claim.}
First, we prove \eqref{YD2} by induction.
By condition \eqref{Y02}, the inequality \eqref{YD2} holds for $i=0$.
Let $i\ge 0$. Assume
 $Y_j\le D$ for all  $0\le j\le i$.
We then have for $0\le j\le i$ that 
\beqs
\begin{aligned}
Y_{j+1} &\le B^j(A_1Y_j^{1+\mu_1} +A_2Y_j^{1+\mu_2} ) = B^j\Big[ A_1D^{1+\mu_1}  \Big(\frac{ Y_j}{D}\Big) ^{1+\mu_1} +A_2 D^{1+\mu_2}  \Big(\frac{ Y_j}{D}\Big)^{1+\mu_2} \Big]\\
         &\le  B^j  (A_1 D^{1+\mu_1}   +A_2D^{1+\mu_2}) \Big (\frac{ Y_j}{D}\Big)^{1+\mu_1} 
         \le  B^j\Big[ \frac{A_1D^{\mu_1}+A_2D^{\mu_2}}{D^{\mu_1}}\Big]    Y_j^{1+\mu_1}.
\end{aligned}        
\eeqs
Then 
$Y_{j+1}\le \tilde A B^j Y_j^{1+\mu_1}$ for all $0\le j\le i$, where $\tilde A = (A_1 D^{\mu_1}   + A_2 D^{\mu_1})/D^{\mu_1}$.
Note from \eqref{Y02} that
\beq\label{yj}
Y_0\le \tilde A^{-1/\mu_1}B^{-1/\mu_1^2} \le D. 
\eeq
Applying Lemma \ref{oriseq} to sequence $\{Y_0,\ldots,Y_{i+1}, 0, 0,\ldots\}$, we have
\begin{align*}
Y_{i+1} \le \tilde A^{\frac{(1+\mu_1)^{i+1} -1}{\mu_1}} B^{\frac{(1+\mu_1)^{i+1} -1}{\mu_1^2} - \frac{i+1}{\mu_1}}Y_0^{(1+\mu_1)^{i+1}}.
\end{align*}
This and the first inequality in \eqref{yj} give
\beq\label{Yitrue2}
Y_{i+1} \le \tilde A^{-1/\mu_1} B^{-1/\mu_1^2} B^{ - \frac{i+1}{\mu_1}}.
\eeq
It follows from this and the second inequality in \eqref{yj} that
$Y_{i+1}\le \tilde A^{-1/\mu_1} B^{-1/\mu_1^2} \le D $.
Thus we obtain \eqref{YD2} for $i+1$, and by the induction principle, statement \eqref{YD2} must hold true for all $i\ge 0$.
Since \eqref{YD2} now holds, the estimate \eqref{Yitrue2} also holds for all $i$, thus the limit \eqref{Ylim2} follows. 
The proof of Claim is complete.

By the virtue of the Claim, it suffices to find $D>0$ that satisfies two inequalities in \eqref{Y02}.
Let $D>0$ be the unique positive number such that $A_1 D^{\mu_1}+A_2 D^{\mu_1}=B^{-1/\mu_1}$. 
Obviously, $D ( A_1D^{\mu_1} +A_2D^{\mu_2} )^{-\frac 1 {\mu_1}} B^{-\frac 1 {\mu_1^2}}=D$.
Then the second inequality in \eqref{Y02} is satisfied.
Since \eqref{Dmu200} implies $D\ge Y_0$, the first inequality in \eqref{Y02} is also satisfied. Our proof is complete. 
\end{proof}

The following is a useful inequality in estimating limits superior of solutions.

\begin{lemma}\label{difflem2}
For $f(t)\ge 0$, $g(t),h(t)>0$, $y(t)\ge 0$,  $t>T$, let
\beqs
y(t)= h(t) \int_T^t e^{-\int_\tau^t g(\theta)d\theta} f(\tau) d\tau.
\eeqs 
Assume
\beq\label{expcond}
\int_T^\infty g(t)dt = \infty\quad\text{and}\quad \lim_{t\to\infty} \frac{h'(t)}{h(t)g(t)} =0.
\eeq
Then
\beq\label{ylim2}
\limsup_{t\to\infty} y(t)\le  \limsup_{t\to\infty} \frac{h(t)f(t)}{g(t)}.  
\eeq
\end{lemma}
\begin{proof}
Let $\delta>0$. We have for $t>T$ that
\begin{align*}
 y'=\Big(-g+\frac{h'}{h}\Big) y + hf=-\Big(1-\frac{h'}{hg}\Big)gy + hf.
\end{align*}
By the second condition in \eqref{expcond}, for sufficiently large $t$, we have 
\begin{align*}
 y'\le -(1-\delta)gy + hf.
\end{align*}
With the first condition in \eqref{expcond}, we apply Lemma A.1(ii) in \cite{HIKS1} and obtain
\beqs
\limsup_{t\to\infty} y(t) \le \frac1{1-\delta}\limsup_{t\to\infty} \frac{h(t)f(t)}{g(t)}.
\eeqs
Letting $\delta\to 0$, we obtain \eqref{ylim2}.
\end{proof}


\def\cprime{$'$}

 \bibliographystyle{acm}

\end{document}